\documentclass[12pt]{amsart}
\usepackage{amsfonts}
\usepackage{amsmath}
\usepackage{amssymb}
\usepackage[utf8]{inputenc}
\usepackage[english]{babel}
\usepackage{enumerate}
\usepackage[a4paper, total={7in, 9in}]{geometry}
\usepackage{mathrsfs}

\newcommand{\iprod}{\mathbin{\lrcorner}}

\theoremstyle{plain}
\newtheorem{thmy}{Theorem}

\newtheorem{cormy}[thmy]{Corollary}

\newtheorem{theorem}{Theorem}[section]
\newtheorem{corollary}[theorem]{Corollary}
\newtheorem{lemma}[theorem]{Lemma}
\newtheorem{proposition}[theorem]{Proposition}

\newtheorem{remark}[theorem]{Remark}
\theoremstyle{definition}

\newtheorem{example}[theorem]{Example}

\newtheorem{notation}[theorem]{Notation}

\numberwithin{equation}{section}
\input{tcilatex}

\begin{document}
\title[Spectral bundles]{Spectral bundles on Abelian varieties, complex
projective spaces and Grassmannians}
\author{Ching-Hao Chang}
\address{Xiamen University Malaysia, Selangor, Malaysia}
\email{chinghao.chang@xmu.edu.my}
\author{Jih-Hsin Cheng}
\address{Institute of Mathematics, Academia Sinica and National Center for
Theoretical Sciences,\\
\hspace*{10pt} Taipei, Taiwan, R.O.C.}
\email{cheng@math.sinica.edu.tw}
\author{I-Hsun Tsai}
\address{Department of Mathematics, National Taiwan University, Taipei,
Taiwan, R.O.C.}
\email{ihtsai@math.ntu.edu.tw}
\thanks{}

\begin{abstract}
In this paper we study the spectral analysis of Bochner-Kodaira Laplacians
on an Abelian variety, complex projective space $\mathbb{P}^{n}$ and a
Grassmannian with a holomorphic line bundle. By imitating the method of
creation and annihilation operators in physics, we convert those
eigensections (of the \textquotedblleft higher energy" level) into
holomorphic sections (of the \textquotedblleft lowest energy" level). This
enables us to endow these spectral bundles, which are defined over the dual
Abelian variety, with natural holomorphic structure. Using this conversion
expressed in a concrete way, all the higher eigensections are explicitly
expressible using holomorphic sections formed by theta functions. Moreover,
we give an explicit formula for the dimension of the space of higher-level
eigensections on $\mathbb{P}^{n}$ through vanishing theorems and the
Hirzebruch-Riemann-Roch theorem. These give a theoretical study related to
some problems newly discussed by string theorists using numerical analysis.
Some partial results on Grassmannians are proved and some directions for
future research are indicated.
\end{abstract}

\maketitle


\section{\textbf{Introduction and Statement of the results}}

\noindent \hspace*{12pt} Let $L$ be an ample line bundle on an Abelian
variety $M.$ In our previous paper \cite{CCT24} we study the direct image
bundle $E$ on the Picard variety Pic$^{0}(M)$ of $M$, whose fiber $E_{\eta
}, $ $\eta \in $Pic$^{0}(M),$ is the vector space spanned by the holomorphic
sections of the line bundle $L\otimes \eta .$ This $E$ can be considered as
a spectral bundle of the \textquotedblleft lowest energy" level. In this
paper we study the spectral bundles of the \textquotedblleft higher energy"
level. One main result is Theorem \ref{T-main}, based on Theorem \ref{TeoA}
and Corollary \ref{TeoB} below. Theorem \ref{TeoA} allows us to convert
eigensections of the \textquotedblleft higher energy" level to holomorphic
sections of the \textquotedblleft lowest energy" level. We also analyze the
spectrum for the case of complex projective space (see Theorem \ref{TeoC}
below) and compute the dimension of eigensections based on the
Hirzebruch-Riemann-Roch theorem and a new type of vanishing theorems \cite%
{Mani97} that are particularly suitable for our need here (see Corollary \ref%
{TeoD} and Theorem \ref{TeoE} below). Furthermore, we have a partial result
on Grassmannians (see Theorem \ref{T-7-1}). All the auxiliary results
mentioned above appear to be of independent interest.

Our study into this research subject was influenced by the work of Prieto (%
\cite{Prieto06(b)}, cf. \cite{Prieto06(a)}), in which he studied similar
questions on compact Riemann surfaces. In \cite{CCT22} and \cite{CCT24} we
computed the \textit{full curvature} of the \textquotedblleft spectral
bundle" (formed by the space of holomorphic sections of line bundles) and
used theta functions (realization of holomorphic sections) for explicit
computations in place of his approach. In fact, we also obtain a globally
well-defined metric on the family of line bundles $L_{\hat{\mu}}$ with the
parameter [$\hat{\mu}]$ $\in $ Pic$^{0}(M)$ which naturally induces an $%
L^{2} $-metric on certain direct image bundle $E$ (see \cite{CCT24} or (\ref%
{EqB}) below with $q=0).$ In this paper we work on the spectral analysis of
the \textquotedblleft higher energy" level for Abelian varieties and complex
projective space. For Grassmannians, we work on the first and second energy
levels.

Prieto's idea is to use annihilation/creation operators to lower/raise
energy levels. We try to make this idea work also in higher dimensions. For
this purpose, our idea is to introduce suitable \textit{contraction operators%
} (see below) in deducing new types of Bochner-Kodaira identities whose
formulation heavily uses these operators. In fact, the computational details
reveal their complexity most often in dimensions higher than one. These
operators become quite trivial in the one-dimensional case. Another new
aspect is that the higher energy eigensections have values in \textit{%
symmetrized} tangent or cotangent bundles, which we realized only after
explicit computations\footnote{%
In view of the recent work \cite{CH24} this phenomenon (the appearance of
symmetrized tangent bundles) might admit a conceptual explanation. See also 
\textit{Notes added in proof} near the end of this Introduction.}. Our
results can yield explicit formulas for high-energy eigensections based on
holomorphic sections which are easily understood in the case of Abelian
varieties (via theta functions). In the $\mathbb{P}^{n}$ case, the part of
the holomorphic sections is less obvious because the bundle involves $%
L\otimes (\odot ^{q}T),$ whose holomorphic sections do not seem easily
constructible in an explicit way. Fortunately, we can prove a dimension
formula for these holomorphic sections; see Theorem \ref{TeoE} below.

\noindent \hspace*{12pt} We remark that many of the computational results
(e.g. Bochner-Kodaira type identities of the third order, raising or
lowering energies via contraction operators) do require the precise forms as
proved here in order to be usable. We do not have \textit{a priori} or 
\textit{conceptual }explanation of these results. This may partly account
for the length of this paper. Nevertheless, we have tried to reduce the
repetition to a minimum. For instance, some proofs for the Abelian variety
are omitted and only indicated as simplified versions of those as in the $%
\mathbb{P}^{n}$ case. Whenever the precise coefficients matter, we choose to
point out where the difference between Abelian variety and $\mathbb{P}^{n}$
lies.

\noindent \hspace*{12pt} Many of our operations work for a K\"{a}hler
manifold $N.$ Let $E\rightarrow N$ be a complex vector bundle over $N$. For
any $r,s\in \{0\}\cup \mathbb{N}$ we denote by $\Omega ^{r,s}(N,E)$ the
space of $(r,s)$-forms on $N$ with values in $E$. Let $T$ (resp. $T^{\ast })$
be the holomorphic tangent (resp. cotangent) bundle of $N.$ \ 

Let $\odot ^{q}T^{\ast }$ (resp. $\odot ^{q}T)$ be the symmetric part of $%
\otimes ^{q}T^{\ast }$ (resp. $\otimes ^{q}T)$. Let $L$ be a holomorphic
Hermitian line bundle over $N.$ On $L\otimes ^{q}T$ and $L\otimes
^{q}T^{\ast },$ there are connections $\nabla ^{q}$ and $\nabla ^{-q}$
naturally induced from the Chern connection $\nabla ^{L}$ on $L$ and the
Levi-Civita connection $\nabla $ on $T$. We can then define associated
operators $\partial ^{q}$, $\partial ^{-q}$, $\overline{\partial }^{q}$, $%
\overline{\partial }^{-q}$ ((\ref{2-1-1} ) - (\ref{2-1-4})), their
contraction operators $I^{q}$, $I^{\,-q}$, $\overline{I}^{\,q}$, $\overline{I%
}^{\,-q}$ ((\ref{2-2-1}) - (\ref{2-2-4})) and the Laplacians $\Delta ^{q}$, $%
\Delta _{q}$, $\Delta ^{-q}$ and $\Delta _{-q}$ ((\ref{partial q star}) -(%
\ref{Delta q})). See Section 2 for more details. For integers $k,q$ we
define \textquotedblleft annihilation operators" $A_{k}:\Omega ^{0,0}\left(
N,L\right) $ $\rightarrow $ $\Omega ^{0,0}(N,$ $L\otimes (\odot ^{k}T))$ and
\textquotedblleft creation operators" $A_{k}^{\#}:\Omega ^{0,0}\left( N,%
\text{ }L\otimes (\odot ^{q}T)\right) $ $\rightarrow $ $\Omega ^{0,0}\left(
N,\text{ }L\otimes (\odot ^{k}T)\right) $ by%
\begin{eqnarray}
&& A_{k}s^{0}\ {:=}\left( I^{-k}\partial ^{-k}\right) ^{\ast }\cdot \cdot
\cdot \left( I^{-1}\partial ^{-1)}\right) ^{\ast }s^{0},\text{ }0<k\leq q,%
\text{ }A_{0}:=Id,  \label{ACs} \\
&& A_{k}^{\#}t^{-q} {:=}\left( I^{-k-1}\partial ^{-k-1}\right) \left(
I^{-k-2}\partial ^{-k-2}\right) \cdot \cdot \cdot \left( I^{-q}\partial
^{-q}\right) t^{-q},\text{ }0\leq k<q,\text{ }A_{q}^{\#}:=Id,\text{
respectively.}  \notag
\end{eqnarray}

%
%

\noindent \hspace*{12pt} To state the results let us first restrict
ourselves to an Abelian variety $M$. We equip $M$ with a special K\"{a}hler
metric $\omega =\frac{2\pi }{B}\omega ^{h}$ where $B>0$ and $\omega ^{h}$ is
a Hodge metric. After a linear change of coordinates, we express the K\"{a}%
hler metric as 
\begin{equation*}
\omega =i\overset{n}{\underset{\alpha =1}{\sum }}\,dz\wedge d\overline{z}^{%
\overline{\alpha }}.
\end{equation*}%
Let $T$ and $T^{\ast }$ be the holomorphic tangent and cotangent bundles of $%
M$ respectively. Let $\nabla $ be the Levi-Civita connection on $T$. Let $L$
be an holomorphic line bundle on $M$ with $c_{1}(L)=[\omega ^{h}]$. On this $%
L$ we equip it with a Hermitian metric $h_{L}$ and the Chern connection $%
\nabla ^{L}$ so that the curvature $\Theta =-iB\omega $. In fact we can
obtain a globally well-defined metric on the family of line bundles $L_{\hat{%
\mu}}$ with the parameter [$\hat{\mu}]$ $\in $ Pic$^{0}(M)$ (see the second
paragraph above) so that the above correspondence (\ref{ACs}) can be treated
within a natural family-version framework.

Denote by $E_{qB}$ (or $E_{qB}(L)\subset $ $\Omega ^{0,0}\left( M,L\right) )$
(resp. $E_{(q-k)B}\subset \Omega ^{0,0}(M,$ $L\otimes (\odot ^{k}T)))$ the
vector space of all eigensections of $\Delta _{0}$ (resp. $\Delta _{-k})$
with the eigenvalues $qB$ (resp. $(q-k)B)$. Note that the notations $%
E_{(q-k)B}$, $0\leq k\leq q$, depend on the bundles which are suppressed for
the sake of simplicity. Let $H^{0}(M,L\otimes (\odot ^{q}T))$ ($\subset $ $%
\Omega ^{0.0}\left( M,L\otimes (\odot ^{q}T)\right) $ denote the space of
holomorphic sections of $L\otimes (\odot ^{q}T)$ over $M.$ With these
notations and settings, we have a one-to-one correspondence between $E_{qB}$
and $H^{0}(M,L\otimes (\odot ^{q}T))$ among others.

\begin{thmy}
\label{TeoA} With the notation above, we have the correspondence between
certain eigensections of high-low energies:


$i)$ $A_{q}$ $:$ $E_{qB}$ $\rightarrow $ $E_{0}$ $=$ $H^{0}(M,$ $L\otimes
(\odot ^{q}T))$ of (\ref{ACs}) is a linear isomorphism with the inverse ($%
q!B^{q})^{-1}A_{0}^{\#}.$

\noindent More precisely, for integers $k,q$ such that $0\leq k\leq q,$ we
have

$ii)$ Let $0\neq s^{0}\in E_{qB}$ ($\subset $ $\Omega ^{0,0}\left( M,\text{ }%
L\right) )$. Then $0\neq s^{-k}:=A_{k}s^{0}$ $\in $ $E_{(q-k)B}$ \ ($\subset 
$ $\Omega ^{0,0}(M,$ $L\otimes (\odot ^{k}T)))$ (note $A_{k=0}$ $=$ $Id).$
In particular $0\neq s^{-q}$ $=$ $A_{q}s^{0}$ $\in $ $E_{0}$ $=$ $H^{0}(M,$ $%
L\otimes (\odot ^{q}T))$ is holomorphic.

$iii)$ %
%
%
Let $0\neq t^{-q}\in H^{0}(M,L\otimes (\odot ^{q}T))$ be a holomorphic
section. Then $0\neq t^{-k}:=A_{k}^{\#}t^{-q}$ $\in $ $E_{(q-k)B}$ (note $%
A_{k=q}^{\#}$ $=$ $Id)$.
\end{thmy}

\noindent \hspace*{12pt} It turns out that the set of eigenvalues of $\Delta
_{0}$ is exactly the set $\{qB$ \TEXTsymbol{\vert} $q\in \{0\}\cup \mathbb{N}%
\}$ (see Remark \ref{R-3-16} $a)$); it is independent of the dimensions of
Abelian varieties.

\begin{cormy}
\label{TeoB} For every $q\in \{0\}\cup \mathbb{N}$, $\dim _{\mathbb{C}%
}E_{qB}=C_{q}^{\,n+q-1}\delta _{1}...\delta _{n}$\ where $C_{\,k}^{\,n}$
denote the binomial coefficients and $\delta _{1},\cdot \cdot \cdot ,$ $%
\delta _{n}$ are elementary divisors of the polarization $c_{1}(L)$ ($=$ $%
[\omega ^{h}]$) \cite[p.306]{GH84}.
\end{cormy}

As an Abelian variety is a flat space, many of our computations may not
reach their full effectivity. Our methodology also extends to
curved spaces: the complex projective space $\mathbb{P}^{n}$ and
Grassmannians. Let $\omega _{FS}$ be the Fubini-Study metric ((\ref{FS})
depending on a scaling constant $c>0$) on $\mathbb{P}^{n}$ and $L$ be a
holomorphic line bundle such that $c_{1}(L)$ $\in $ $H^{2}(\mathbb{P}^{n},%
\mathbb{Z})$ $\cong $ $\mathbb{Z}$ is denoted by $B,$ called the degree of $%
L $.\ Equip $L$ with the Hermitian metric $h_{L}$ and the Chern connection $%
\nabla ^{L}$ such that the curvature $\Theta =-iB\omega _{FS}\frac{c}{2}$.
Some works in theoretical physics (e.g. \cite{AHHO23}) use a vector bundle $V$ carrying a
Hermitian-Yang-Mills connection (see \textit{Notes added in proof} near the
end of this Introduction)\footnote{%
The interest arising from physics does not stick to the line bundle case of $%
V$ as we assume here. Even in the line bundle case, the work \cite{AHHO23}
refers to that of Prieto on elliptic curves \cite{Prieto06(a)}. None of the
results here were obtained by, or known to, \cite{AHHO23} except the $%
\mathbb{P}^{n}$ case, which is discussed more in the Introduction later.}.
In normal coordinates we can express the curvature tensor of $\mathbb{P}^{n}$
as follows: ($c$ is set to be $2$ in what follows) 
\begin{equation*}
R_{i\overline{j}k\overline{l}}=\frac{c}{2}\left( g_{i\overline{j}}g_{k%
\overline{l}}+g_{i\overline{l}}g_{k\overline{j}}\right) ,\text{ \ \ \ \ }R_{%
\overline{i}j\overline{k}l}=\overline{R_{i\overline{j}k\overline{l}}}.%
\hspace*{100pt}
\end{equation*}

To write the eigenvalues, let $N_{k,l}$ $:=$ $(k-l)B$ $+$ $\frac{c}{2}%
[(k^{2}-l^{2})$ $+$ $n(k-l)],$ $k-l$ $\geq $ $0,$ $k,$ $l$ $\in $ $\{0\}\cup 
\mathbb{N}$. On $\mathbb{P}^{n}$ we denote by $E_{0,q,B,c}$ ($\subset $ $%
\Omega ^{0,0}\left( \mathbb{P}^{n},L)\right) $ (resp. $E_{k,q,B,c}(\subset $ 
$\Omega ^{0,0}(\mathbb{P}^{n},L\otimes (\odot ^{k}T)))$ the vector space of
eigensections of $\Delta _{0}$ (resp. $\Delta _{-k})$ \textbf{with the
eigenvalues} 
\begin{eqnarray*}
&&qB+\frac{cq(n+q)}{2}( = N_{q,0}),q\in \{0\}\cup \mathbb{N} \\
&&(\text{resp}.(q-k)B+\frac{c}{2}(q^{2}-k^{2})+\frac{c}{2}n(q-k)( =
N_{q,k}),0\leq k\leq q).
\end{eqnarray*}%
\noindent It turns out that we have results for $\mathbb{P}^{n}$ similar to
those of the Abelian variety case.

\begin{thmy}
\label{TeoC} With the notation above, we have the following for $B>0$ and
every $q\in \mathbb{N}$:


$i)$ $A_{q}$$:$$E_{0,q,B,c}$$\rightarrow $$E_{q,q,B,c}$$=$$H^{0}(\mathbb{P}%
^{n},$ $L\otimes (\odot ^{q}T))$ of (\ref{ACs}) is a linear isomorphism with
the inverse $(\prod\limits_{k=0}^{q-1}N_{q,k})^{-1}A_{0}^{\#}.$

$ii)$ Let $0\neq s^{0}\in E_{0,q,B,c}$ ($\subset $ $\Omega ^{0,0}\left( 
\mathbb{P}^{n},L\right) $)$.$ Then $0\neq s^{-k}:=A_{k}s^{0}$ $\in $ $%
E_{k,q,B,c}(\subset $ $\Omega ^{0,0}(\mathbb{P}^{n},$ $L\otimes (\odot
^{k}T))$ for $k=1,...,q$. In particular $0\neq s^{-q}$ $=$ $A_{q}s^{0}$ $\in 
$ $E_{0}$ $=$ $H^{0}(\mathbb{P}^{n},$ $L\otimes (\odot ^{q}T))$ is
holomorphic.

$iii)$ 
Let $0\neq t^{-q}\in H^{0}(\mathbb{P}^{n},L\otimes (\odot ^{q}T))$ be a
holomorphic section. Then $0$ $\neq $ $t^{-k}$ $:=$ $A_{k}^{\#}t^{-q}$ $\in $
$E_{k,q,B,c}$ for $k=0,1,...,q-1$.
\end{thmy}

\begin{cormy}
\label{TeoD} For every $q\in \{0\}\cup \mathbb{N}$, $\dim _{\mathbb{C}%
}E_{0,q,B,c}=h^{0}(\mathbb{P}^{n},$ $L\otimes (\odot ^{q}T))$ $\left( =\dim
_{\mathbb{C}}H^{0}(\mathbb{P}^{n},\text{ }L\otimes (\odot ^{q}T))\right) $ $%
> $ $0$.
\end{cormy}

Note that by Remark \ref{R-5-13} $a),$ $\cup _{q\in \{0\}\cup \mathbb{N}%
}E_{0,q,B,c}$ is exactly the set of all eigensections of $\Delta _{0}.$

We can give an explicit formula for the quantity in Corollary \ref{TeoD}.
Let td$(T)$ denote the Todd class of $T\rightarrow \mathbb{P}^{n}$ and let $%
\omega =c_{1}(\mathcal{O}(1))$ denote the first Chern class of $\mathcal{O}%
(1)$ (the hyperplane bundle on $\mathbb{P}^{n}$). Denote by $%
\begin{pmatrix}
a \\ 
n%
\end{pmatrix}%
=\frac{a(a-1)(a-2)...(a-n+1)}{n!}$ the generalized binomial coefficients for 
$a\in \mathbb{C}$, $n\in \mathbb{N}$. Notice that $%
\begin{pmatrix}
k \\ 
n%
\end{pmatrix}%
=C_{n}^{k}$ when $k\geq n$, and $k\in \mathbb{N}$, $n\in \{0\}\cup \mathbb{N}
$.

\begin{thmy}
\label{TeoE} $\displaystyle$ With the notation above, we have%
\begin{equation}
h^{0}(\mathbb{P}^{n},L\otimes (\odot ^{q}T))=\underset{k_{1}+\cdot \cdot
\cdot +k_{n}=q}{\sum }\int_{\mathbb{P}^{n}}\mbox{td}(T)\cdot
e^{(y_{k_{1},\cdot \cdot \cdot ,k_{n}}+B)\,\omega }=\underset{k_{1}+\cdot
\cdot \cdot +k_{n}=q}{\sum }%
\begin{pmatrix}
y_{k_{1},\cdot \cdot \cdot ,k_{n}}+n+B \\ 
n%
\end{pmatrix}
\label{E-1}
\end{equation}
\noindent where $y_{k_{1},\cdot \cdot \cdot ,k_{n}}:=\underset{j}{\sum }%
k_{j}\lambda _{j}$\hspace*{4pt}for any $k_{j}\in \mathbb{N}\cup \{0\}$ and $%
\lambda _{j}$ are $1-e^{i\frac{2\pi j}{n+1}},$ $j=1,$ $\cdot \cdot \cdot $ $%
, $ $n,$ the roots of $%
x^{n}+(-1)^{1}C_{1}^{n+1}x^{n-1}+...+(-1)^{n}C_{n}^{n+1}=0$ and $B=\deg
(L)\in \mathbb{N}$.
\end{thmy}

The proof of the above dimension formula in Theorem \ref{TeoE} is based on
the Hirzebruch-Riemann-Roch theorem and vanishing theorems proved by Manivel
(see Theorem \ref{l+n-p} \cite{Mani97}). We also give an explicit formula
for $n=2$ in Example \ref{E-6-1} which is deduced from some basic facts and
is consistent with the above formula (\ref{E-1}).

We turn now to study the holomorphic structure of spectral bundles over the
Picard variety $\hat{M}$ $=$ Pic$^{0}(M).$ Using the line bundle $\tilde{E}%
=\pi _{1}^{\ast }L\otimes P$ $\rightarrow $ $M\times \hat{M}$ where $P$ is
the Poincar\'{e} line bundle, we have the direct image ($\pi _{2})_{\ast }%
\tilde{E}$ $\rightarrow $ $\hat{M},$ whose fibre at $[\hat{\mu}]$ $\in $ $%
\hat{M}$ is (($\pi _{2})_{\ast }\tilde{E})_{[\hat{\mu}]}$ $=$ $H^{0}(M,%
\tilde{E}|_{M\times \{[\hat{\mu}]\}})$ which is $E_{(q=0)B}(\tilde{E}%
|_{M\times \{[\hat{\mu}]\}})$ in the notation above (see $E_{qB}$ and $%
E_{qB}(L)$)$.$ For higher energy with $q>0$ we consider the fibre (($\pi
_{2}^{q})_{\ast }\tilde{E})_{[\hat{\mu}]}$ to be $E_{qB}(\tilde{E}|_{M\times
\{[\hat{\mu}]\}})$ and our spectral bundles are first defined to be, as a set%
\begin{equation}
\mathbb{E}_{qB}:=\bigcup\limits_{[\hat{\mu}]\in \hat{M}}E_{qB}(\tilde{E}%
|_{M\times \{[\hat{\mu}]\}})\rightarrow \hat{M}.  \label{EqB}
\end{equation}

The fact that $\mathbb{E}_{qB}$ forms a smooth bundle is well known (cf. 
\cite[p.167]{BF86}). The idea for the theorem below is that for $q=0$ the
full curvature using a\ natural $L^{2}$-metric on $\mathbb{E}_{(q=0)B}$ is
obtained in \cite[Theorem 1.2]{CCT24}, whose computation went through an
auxiliary holomorphic bundle over $M$ and then pushed forward to $\hat{M}$
via isogeny $M\rightarrow \hat{M}$. This holomorphic case and the linear
isomorphism (Theorem \ref{TeoA}) lead to (\ref{FC}) below for eigenbundles.
See Section \ref{Sec3}.


\begin{theorem}
\label{T-main} With the notation above, $\mathbb{E}_{qB}$ can be endowed
with a natural holomorphic bundle structure. Moreover, the full curvature $%
\Theta (\mathbb{E}_{qB},h_{q})$ associated with a natural metric $h_{q}$ on $%
\mathbb{E}_{qB}$ is given by (see (\ref{HodgeMetric}) for the notation $%
W_{\alpha \beta }$ and \cite[second paragraph on p.7]{CCT24} for the
notation $\hat{\mu}_{\alpha })$%
\begin{equation}
\Theta (\mathbb{E}_{qB},h_{q})=\Big{(}-\pi \sum_{\alpha ,\beta
=1}^{n}W_{\alpha \beta }\frac{\delta _{\alpha }\delta _{\beta }}{\delta
_{n}\delta _{n}}d\hat{\mu}_{\alpha }\wedge d\overline{\hat{\mu}}_{\beta }%
\Big{)}\cdot I_{\Delta _{q}\times \Delta _{q}}  \label{FC}
\end{equation}%
where $I_{\bullet }$ denotes the identity matrix of rank $\Delta
_{q}:=C_{q}^{\,n+q-1}\delta _{1}...\delta _{n}$\ (= $\dim _{\mathbb{C}%
}E_{qB};$ see Corollary \ref{TeoB} above)$.$
\end{theorem}

There is a vast literature\footnote{%
In particular, the famous works of S. Helgason deal with invariant
differential operators acting on functions \cite{Hel00} and mostly on
noncompact Riemannian symmetric spaces \cite{Hel08}. The works of A. Terras
are closely related to automorphic forms associated to certain discrete
subgroups (\cite{Ter85}, \cite{Ter88}). As we are concerned with
bundle-valued sections on Hermitian symmetric spaces of compact type, none
of these works can be immediately adopted for the purpose of the present
paper.} on the study of eigenvalues of Laplace-type operators on Riemannian
manifolds, including the famous Selberg's $\frac{1}{4}$ conjecture on
Riemann surfaces (compare the discussion below). In a physical context,
Landau Hamiltonians $H$ with various energy levels/eigenvalues occur
naturally on two-dimensional tori (see \cite{Kato50}, \cite{ASY87} and \cite%
{Lev95}, where the $H$ parametrized by Aharanov-Bohm potentials and the
evolution of spectral bundles on the parameter space are discussed). When
the parameter space is Pic$^{0}(M)$ ($M$ as an Abelian variety), the
eigenvalues and spectral bundles are studied fully and explicitly in the
present paper. Since the sort of \textquotedblleft periodic conditions" in
physics can usually be satisfied also by a non-algebraic complex torus, it
seems natural to include the case of complex tori. Via the description of
line bundles on complex tori (e.g. \cite{BL99} or \cite{DEP05}), we leave it
to the reader to generalize the results of the present paper to complex
tori. 
For another study, one can ask for the similar spectral problem on Hermitian
symmetric spaces of compact type such as Grassmannians or even rational
homogeneous spaces. At least on Grassmannians, the similar results hold for
the second lowest eigenvalue (and the lowest), but for general eigenvalues
the Bochner-Kodaira type identities require revision by extra terms (see
Section \ref{Sec7}). The Grassmannians have these extra terms arising from
the \textit{nullity} of the space (regarded as a feature of symmetric spaces
of rank $\geq 2$); the nullity also accounts for the rigidity phenomena
initiated by Ngaiming Mok (\cite{Mok87a}, \cite{Mok87b}). This new
phenomenon resulting from the new Bochner-Kodaira type identities seems
quite interesting and worth a closer study in the future. As for negative
curvature, the answer to the spectral problem gets much more complicated and
is not quite analogous to the cases of nonnegative curvature. In \cite%
{Prieto06(a)}, \cite{Prieto06(b)} Prieto has studied it on hyperbolic
(compact) Riemann surfaces. Unlike the cases of genus $\leq $ $1,$ although
some eigensections do get identified with certain holomorphic sections, some
others can only be obtained by using the\textit{\ Maass automorphic forms} (%
\cite[p.387]{Prieto06(b)} and \cite[Section 4.3]{Prieto06(a)}). A natural
question appears to concern similar phenomena for negatively curved spaces
in higher dimensions.

A further study may be related to the group representation theory which we
have not touched upon. For instance, an answer to the question how/whether
the above dimension formula for $\mathbb{P}^{n}$ (Theorem \ref{TeoE}) can be
interpreted by the representation theory of $U(n+1)$ is unknown to us, let
alone that for Hermitian symmetric spaces (if any).

It is worth mentioning that the works R. Kuwabara \cite{Kuw88} and D.
Bykov-A. Smilga \cite{BS23} (on which the physical work \cite{AHHO23} is
partly based) consider (real) Bochner Laplacians on $\mathbb{P}^{n}$, in
contrast to the Dolbeault Laplacians of ours, and obtain the eigenvalues and
multiplicities. Their methods are quite special\footnote{%
The work \cite{Kuw88} uses Hopf fibrations $S^{2N+1}$ $\rightarrow $ $%
\mathbb{P}^{N\text{ }}$and converts the eigensections on $\mathbb{P}^{N\text{
}}$to the well-known spherical harmonics. The work \cite{BS23} is
essentially \textit{physical}, and give no precise proof that the
eigensections obtained \ there are exhaustive. Neither of the works \cite%
{Kuw88}, \cite{BS23} deals with Abelian varieties.} as far as
generalizations to other spaces are concerned; for instance, it is unclear
how the Abelian variety case here can be treated using their methodology. As
aforementioned, our method is more general and paves a way to working for
Hermitian symmetric spaces and K\"{a}hler manifolds. See Section \ref{Sec7}
and in particular, Theorem \ref{T-7-1}, a partial result for a Grassmannian.
Further, we give an explicit construction of higher energy eigensections
starting from holomorphic sections (lowest energy), which is not considered
in \cite{Kuw88} and \cite{BS23}. As a consequence, our method gives a
uniform and natural holomorphic structure on these spectral bundles (which
are \textit{a priori} non-holomorphic).

In \cite[Introduction]{CCT24} we wrote \textquotedblleft\ ... the theta
function method here is expected to be also relevant to analogous problems
in a $p$-adic setting \cite{Ro70} $...$ ". As such, let us mention in
passing that it is possible to give a formulation of \textquotedblleft
higher energy states" in suitable $p$-adic settings. Indeed, this has been
discussed (and a vast literature is thus entailed) at least in the sense of $%
p$-adic quantum mechanics (see e.g. \cite[p.236]{VVZ94}, \cite[Section 3]%
{Dr04}, \cite[Subsection 8.2.3]{ZG24}). Mathematically, it is unknown to us
whether the higher energy study exists in a $p$-adic sense (and can be
connected to concrete problems) or not, but it seems to be a natural and
interesting subject for future research.

\textit{Notes added in proof}. The recently published work by L. Charles 
\cite{CH24} focuses on problems (a kind of \textit{semi-classical analysis})
closely related to the ones here, having completely different methods and
results. For instance, in contrast to our Corollary \ref{TeoD} (no Theorem %
\ref{TeoE}-like version was given there) they consider the power $L^{k}$ and
obtained the dimension formula only when $k>>1.$ Another new and very
interesting publication by superstring theorists A. Ashmore, Y-H. He, E.
Heyes and B. A. Ovrut discussed the spectrum of $\Delta _{\bar{\partial}%
_{V}} $ on $\mathbb{P}^{3}$ including some numerical results (where $V$ is a
bundle carrying a Hermitian-Yang-Mills connection) \cite[Section 3]{AHHO23}.
Some physical works on the aforementioned semi-classical analysis on Riemann
surfaces are discussed, e.g. by M. V. Karasev in \cite{Kar07(a)}, \cite%
{Kar07(b)}. A mathematical counterpart of $\Delta _{\bar{\partial}_{V}}$ has
been studied by M. Jardim and R. F. Le\~{a}o in \cite{JL09} where they
obtained a lower bound estimate of eigenvalues on compact K\"{a}hler
manifolds and compare this result to the exact formula on Riemann surfaces
by Prieto et al. \cite{AP06}. The vector bundle $V$ (or the line bundle $L$)
of \cite{JL09} means the same as that of \cite[Section 3]{AHHO23} and is of
non-positive degree (in contrast to our line bundle $L$ here which is of
positive degree).

The paper is organized as follows. In Section \ref{Sec2} we give basic
definitions of some operators on K\"{a}hler manifolds, and introduce
Bochner-Kodaira type identities on Abelian varieties. We prove Theorem \ref%
{TeoA}, Corollary \ref{TeoB} and Theorem \ref{T-main} in Section \ref{Sec3}.
For complex projective spaces we discuss relevant Bochner-Kodaira type
identities in Section \ref{Sec4}. We then prove Theorem \ref{TeoC} and
Corollary \ref{TeoD} in Section \ref{Sec5}. Finally, we give a proof of
Theorem \ref{TeoE} in Section \ref{Sec6} and that of Theorem \ref{T-7-1} in
Section \ref{Sec7} respectively. Some proofs, bearing similarities to
previous ones, are placed in the Appendix.

\bigskip

\textbf{Acknowledgements. }The first author is supported by Xiamen
University Malaysia Research Fund: grant no.XMUMRF/2024-C14/IMAT/0034. The
second author would like to thank the National Science and Technology
Council of Taiwan for the support: grant no. 112-2115-M-001-012 and the
National Center for Theoretical Sciences for the constant support. The third
author thanks the Ministry of Education of Taiwan for the financial support.

\section{\textbf{Contraction operators and Bochner-Kodaira type identities
of the third order$\label{Sec2}$}}

\begin{notation}
\label{notation MN} If not specifically mentioned otherwise, we will use $M$
to denote an Abelian variety and $(N,\omega )$ to denote any K\"{a}hler
manifold with a K\"{a}hler form $\omega $ $=$ $i\,\overset{n}{\underset{%
\alpha ,\beta =1}{\dsum }}g_{\alpha \overline{\beta }}\,dz^{\alpha }\wedge d%
\overline{z}^{\overline{\beta }}$. Notations given in the Introduction are
adopted in what follows.
\end{notation}

\noindent \hspace*{12pt} For $(N,\omega )$ and a holomorphic Hermitian line
bundle $L$ over $N$, denote by $\Omega ^{r,s}(N,L\otimes T^{\ast q})$ the
space of $(r,s)$-forms on $N$ with values in $L\otimes T^{\ast q}.$ Write
the local expression of any given element in $\Omega ^{r,s}(N,L\otimes
T^{\ast q})$ as%
\begin{equation}
dz^{I}\wedge d\overline{z}^{\overline{J}}\otimes s_{I\overline{J}}\otimes
f_{\alpha _{1}...\alpha _{q}}(\overset{q}{\underset{k=1}{\otimes }}%
dz^{\alpha _{k}})  \label{2-1}
\end{equation}%
\noindent where $s_{I\overline{J}}$ are smooth local sections on $L$ with
multi-indices $I$, $\overline{J}$, $|I|=r,|\overline{J}|=s$ and $f_{\alpha
_{1}...\alpha _{q}}$ are complex-valued smooth local functions with $\alpha
_{i}\in \{1,...,n\}$, $i=1,...,q$. 
Similar notations apply to $\Omega ^{r,s}(N,L\otimes T^{q})$.%
\newline

\noindent \hspace*{12pt} With the connection $\nabla ^{q}$ on $L\otimes
T^{\ast q}$ one has the operators $\partial ^{q}\ (q\in \{0\}\cup \mathbb{N}%
):$ (often omitting the dependence of the connection $\nabla ^{q}$ in
notation$)$%
\begin{align}
{\normalsize \partial }^{q}& {\normalsize :}\ {\normalsize \Omega }^{r,s}%
{\normalsize (N,L\otimes T}^{\ast q}{\normalsize )\longrightarrow \Omega }%
^{r+1,s}{\normalsize (N,L\otimes T}^{\ast q})\hspace*{244pt}  \notag
\label{2-1-1} \\
& \ \ 
\mbox{by sending $dz^{I} \wedge d\overline{z}^{\overline{J}} \otimes s_{I
\overline{J}}\otimes f_{\alpha_{1}...\alpha_{q}}  (
\overset{q}{\underset{k=1}{\otimes}} dz^{\alpha_{k}} )$}\hspace*{180pt} 
\notag \\
& \ \ 
\mbox{to\hspace*{10pt}$\overset{n}{\underset{i=1}{\sum}}\, 
 dz^{i}
\wedge  dz^{I} \wedge d\overline{z}^{\overline{J}}\otimes\nabla^{L}_{\frac{\partial}{\partial z^{i}}}s_{I \overline{J}}    \otimes
f_{\alpha_{1}...\alpha_{q}} ( \overset{q}{\underset{k=1}{\otimes}}
dz^{\alpha_{k}} )$}\hspace*{140pt} \\
& \ \ 
\mbox{$+\overset{n}{\underset{i=1}{\sum}}\,  dz^{i}
\wedge dz^{I} \wedge d\overline{z}^{\overline{J}} \otimes s_{I \overline{J}}\otimes
\nabla_{\frac{\partial}{\partial z^{i}}} \big( f_{\alpha_{1}...\alpha_{q}}  
( \overset{q}{\underset{k=1}{\otimes}} dz^{\alpha_{k}} ) \big)$}.  \notag
\end{align}%
\noindent The operator $\partial ^{q}$ is globally well-defined (cf. \cite[%
p.221]{Wells08}).

\noindent \hspace*{12pt} Similarly, we define $\overline{\partial }^{q}$, $%
\partial ^{-q}$, and $\overline{\partial }^{-q}$ $(q\in \{0\}\cup \mathbb{N}%
) $ as variants of the usual exterior covariant derivatives. Although it is
not strictly necessary to use all of these operators (compare the footnote
attached to the Laplacians (\ref{Delta q})), we introduce them simply to
streamline our presentation. For the sake of clarity\footnote{%
In conformity with the standard notations, these operators are the $%
D^{\prime }$ and $D^{\prime \prime }$ in \cite[p.221]{Wells08}. For example,
if $E=L\otimes T^{*q}$ then $D^{\prime q}=\partial ^{q}$ and $D^{\prime
\prime }=\overline{\partial }^{q}$.}, we choose to spell out the details: 
\begin{eqnarray}
&&\ \ \ \ \ \ 
\mbox{$\overline{\partial} ^{q} {\normalsize :}\ {\normalsize \Omega
}^{r,s}{\normalsize (N, L\otimes T}^{\ast q}{\normalsize )\longrightarrow
\Omega }^{r,s+1}{\normalsize (N, L\otimes T}^{\ast q} )$}\hspace*{244pt}
\label{2-1-2} \\
&&\ \ \ \ \ \ \ \ \ \ \ 
\mbox{$dz^{I} \wedge d\overline{z}^{\overline{J}} \otimes s_{I
\overline{J}}\otimes f_{\alpha_{1}...\alpha_{q}} (
\overset{q}{\underset{k=1}{\otimes}} dz^{\alpha_{k}} )\longmapsto$}\hspace*{%
214pt}  \notag \\
&&\ \ \ \ \ \ \ \ \ \ \ 
\mbox{$\overset{n}{\underset{j=1}{\sum}}\,  
  d\overline{z}^{\overline{j}} \wedge dz^{I} \wedge
d\overline{z}^{\overline{J}}  \otimes \nabla^{L}_{\frac{\partial}{\partial \overline{z}^{\overline{j}}}} s_{I
	\overline{J}}\otimes f_{\alpha_{1}...\alpha_{q}}  (
\overset{q}{\underset{k=1}{\otimes}} dz^{\alpha_{k}} )$}\hspace*{156pt} 
\notag \\
&&\ \ \ \ \ \ \ \ \ \ \ 
\mbox{$+\overset{n}{\underset{j=1}{\sum}}\,  
d\overline{z}^{\overline{j}} \wedge dz^{I} \wedge
d\overline{z}^{\overline{J}}   \otimes s_{I \overline{J}} \otimes \nabla_{\frac{\partial}{\partial
\overline{z}^{\overline{j}}}}\big( f_{\alpha_{1}...\alpha_{q}}  (
\overset{q}{\underset{k=1}{\otimes}} dz^{\alpha_{k}} )\big)$},\hspace*{96pt}
\notag
\end{eqnarray}%
\begin{eqnarray}
&&%
\mbox{${\normalsize \partial }^{-q} {\normalsize :}\ {\normalsize \Omega
}^{r,s}{\normalsize (N, L\otimes T}^{q}{\normalsize )\longrightarrow \Omega
}^{r+1,s}{\normalsize (N, L\otimes T}^{q} )$}\hspace*{208pt}  \label{2-1-3}
\\
&&\ \ \ \ \ \ 
\mbox{$dz^{I} \wedge d\overline{z}^{\overline{J}} \otimes s_{I
\overline{J}}\otimes f^{\alpha_{1}...\alpha_{q}} (
\overset{q}{\underset{k=1}{\otimes}}  \frac{\partial}{\partial
z^{\alpha_{k}}} ) \longmapsto$}\hspace*{208pt}  \notag \\
&&\ \ \ \ \ \ 
\mbox{$\overset{n}{\underset{i=1}{\sum}}\,
 dz^{i}\wedge dz^{I} \wedge d\overline{z}^{\overline{J}}\otimes \nabla^{L}_{\frac{\partial}{\partial z^{i}}}s_{I \overline{J}}  \otimes
f^{\alpha_{1}...\alpha_{q}} ( \overset{q}{\underset{k=1}{\otimes}} 
\frac{\partial}{\partial z^{\alpha_{k}}})$}\hspace*{152pt}  \notag \\
&&\ \ \ \ \ \ 
\mbox{$+\overset{n}{\underset{i=1}{\sum}}\, dz^{i}
\wedge dz^{I} \wedge d\overline{z}^{\overline{J}} \otimes s_{I\overline{J}}\otimes 
\nabla_{\frac{\partial}{\partial z^{i}}}\big( f^{\alpha_{1}...\alpha_{q}}  (
\overset{q}{\underset{k=1}{\otimes}}  \frac{\partial}{\partial
z^{\alpha_{k}}}) \big),  $}\hspace*{98pt}  \notag
\end{eqnarray}%
\begin{eqnarray}
&&%
\mbox{$\overline{\partial} ^{-q} {\normalsize :}\ {\normalsize \Omega
}^{r,s}{\normalsize (N, L\otimes T}^{q}{\normalsize )\longrightarrow \Omega
}^{r,s+1}{\normalsize (N, L\otimes T}^{q} )$}\hspace*{212pt}  \label{2-1-4}
\\
&&\ \ \ \ \ \ 
\mbox{$dz^{I} \wedge d\overline{z}^{\overline{J}} \otimes s_{I
\overline{J}}\otimes f^{\alpha_{1}...\alpha_{q}}  (
\overset{q}{\underset{k=1}{\otimes}}  \frac{\partial}{\partial
z^{\alpha_{k}}})\longmapsto$}\hspace*{208pt}  \notag \\
&&\ \ \ \ \ \ 
\mbox{$\overset{n}{\underset{j=1}{\sum}}\,
 d\overline{z}^{\overline{j}} \wedge dz^{I} \wedge
d\overline{z}^{\overline{J}} \otimes \nabla^{L}_{\frac{\partial}{\partial \overline{z}^{\overline{j}}}}s_{I
	\overline{J}} \otimes f^{\alpha_{1}...\alpha_{q}}  (
\overset{q}{\underset{k=1}{\otimes}}  \frac{\partial}{\partial
z^{\alpha_{k}}})$}\hspace*{152pt}  \notag \\
&&\ \ \ \ \ \ 
\mbox{$+\overset{n}{\underset{j=1}{\sum}}\, 
d\overline{z}^{\overline{j}} \wedge dz^{I} \wedge
d\overline{z}^{\overline{J}} \otimes s_{I \overline{J}}\otimes  \nabla_{\frac{\partial}{\partial
\overline{z}^{\overline{j}}}}\big( f^{\alpha_{1}...\alpha_{q}}  (
\overset{q}{\underset{k=1}{\otimes}}  \frac{\partial}{\partial
z^{\alpha_{k}}}) \big)$}.\hspace*{100pt}  \notag
\end{eqnarray}%
\noindent Note that $\partial ^{0}=\partial ^{-0}$, $\overline{\partial }%
^{0}=\overline{\partial }^{\,-0}$ by the definitions above.

\begin{notation}
\label{notation}


$i)$ We also write $\nabla _{\frac{\partial }{\partial z^{i}}}\big(f_{\alpha
_{1}...\alpha _{q}}(\overset{q}{\underset{k=1}{\otimes }}dz^{\alpha _{k}})%
\big)$ as $f_{\alpha _{1}...\alpha _{q},i}(\overset{q}{\underset{k=1}{%
\otimes }}dz^{\alpha _{k}})$ $=:$ \textbf{f }where $\nabla _{\frac{\partial 
}{\partial z^{i}}}$ means $\nabla _{\frac{\partial }{\partial z^{i}}%
}^{T^{\ast q}}$. Similar notations apply in a related context. For the order
of covariant derivatives, $f_{\alpha _{1}...\alpha _{q},i\bar{j}}$ means
taking the $i$-derivative first, and then $\bar{j}$ using $\nabla _{\frac{%
\partial }{\partial z^{\overline{j}}}}^{T^{\ast q}}$ (rather than $\nabla _{%
\frac{\partial }{\partial z^{\overline{j}}}}^{T^{\ast (q+1)}})$. Here, note
that $f_{\alpha _{1}...\alpha _{q},i}$ is regarded as the section of $%
T^{\ast q}$ (rather than $T^{\ast (q+1)})$ that is obtained by pairing the
section $\nabla $\textbf{f }of $(T^{\ast q})\otimes T^{\ast }$ with $\frac{%
\partial }{\partial z^{i}}$ \cite[p.148]{Taubes11}$.$\vspace*{-8pt}\newline

$ii)$ If $s$ is a local section of $L$, we write $\nabla _{{\tiny \frac{%
\partial }{\partial z^{i}}}}^{L}s$ simply as $s_{i}$, $\nabla _{{\tiny \frac{%
\partial }{\partial \overline{z}^{\overline{j}}}}}^{L}s$ as $s_{\overline{j}%
} $ and $s_{i\overline{j}}$ as $\nabla _{{\tiny \frac{\partial }{\partial 
\overline{z}^{\overline{j}}}}}^{L}s_{i}$. 
This is not to be confused with $s_{I\bar{J}}$ in (\ref{2-1}).
\end{notation}

\noindent \hspace*{12pt} Write ${\alpha }^{r,s}$ for a local expression of
an element of $\Omega ^{r,s}(N,L)\ (r,s\in \{0\}\cup \mathbb{N})$. We first
define\footnote{%
No multiplicative constant $r$! is involved. More precisely, we identify the 
$(r,0)$-forms on $N$ with the skew-symmetric covariant tensor fields of rank 
$r$ on $N$:
\par
\begin{equation*}
dz^{1}\wedge ...\wedge dz^{r}=\underset{\sigma \in S_{r}}{\sum }sgn%
\begin{pmatrix}
1 & 2 & ... & r \\ 
\sigma (1) & \sigma (2) & ... & \sigma (r)%
\end{pmatrix}%
dz^{\sigma (1)}\otimes ...\otimes dz^{\sigma (r)}.
\end{equation*}%
\par
\noindent Similar definition applies to (0,s)-form, and $dz^{\alpha
_{1}}\wedge ...\wedge dz^{\alpha _{r}}\wedge d\overline{z}^{\overline{\beta }%
_{1}}\wedge ...\wedge d\overline{z}^{\overline{\beta }_{s}}$ identifies with 
$(dz^{\alpha _{1}}\wedge ...\wedge dz^{\alpha _{r}})\otimes (d\overline{z}^{%
\overline{\beta }_{1}}\wedge ...\wedge d\overline{z}^{\overline{\beta }%
_{s}}) $ (see \cite[p.107 and p.141]{Kod86}).} the operator $I^{q}$: (The
notation $\otimes $ is often omitted below if there is no danger of
confusion)%
\begin{eqnarray}
&&{I}^{q}\ {\normalsize :}\ \Omega ^{r+1,s}(N,L\otimes T^{\ast
q})\longrightarrow \Omega ^{r,s}(N,L\otimes T^{\ast (q+1)}),\text{ }q\in
\{0\}\cup \mathbb{N}\hspace*{120pt}  \label{2-2-1} \\
&&\ \ \ \ \ \ 
\mbox{by sending $\alpha^{r+1,s} (\overset{q}{\underset{k=1}{\otimes}}
dz^{\alpha_{k}})=(-1)^{r+s} \overset{n}{\underset{i=1}{\sum}}\,
\big((\frac{\partial}{\partial z^{i}} \iprod \alpha^{r+1,s})\wedge
dz^{i}\big)(\overset{q}{\underset{k=1}{\otimes}} dz^{\alpha_{k}})$}\hspace*{%
48pt}  \notag \\
&&\ \ \ \ \ \ 
\mbox{to\hspace*{12pt} $(-1)^{r+s}\big( \overset{n}{\underset{i=1}{\sum}}\,
\frac{\partial}{\partial z^{i}} \iprod \alpha^{r+1,s} \big) \big( dz^{i}
\otimes (\overset{q}{\underset{k=1}{\otimes}} dz^{\alpha_{k}})\big)$}%
\hspace*{148pt}  \notag
\end{eqnarray}

\noindent using the interior product $\mathbin{\lrcorner}$ on differential
forms on $N$. 
%
Similarly we define $\overline{I}^{q}$, $I^{-q}$ and $\overline{I}^{-q}$. We
choose to put down the details (see Notation \ref{notation MN} for $%
g^{\alpha _{q}\bar{j}}$ below): 
\begin{eqnarray}
&&%
\mbox{${\overline{I}}^{q}\ {\normalsize :}\ \Omega ^{r,s+1}(N, L\otimes
T^{\ast q})\longrightarrow \Omega^{r,s}(N, L\otimes T^{\ast (q-1)}),\text{
}q\in  \mathbb{N}$}\hspace*{136pt}  \label{2-2-2} \\
&&\ \ \ \ \ \ 
\mbox{$\alpha^{r,s+1} (\overset{q}{\underset{k=1}{\otimes}}
dz^{\alpha_{k}})=(-1)^{r+s} \overset{n}{\underset{j=1}{\sum}}\, \big(
(\frac{\partial}{\partial \overline{z}^{\overline{j}}} \iprod
\alpha^{r,s+1})\wedge d\overline{z}^{\overline{j}}
\big)\big(\overset{q}{\underset{k=1}{\otimes}} dz^{\alpha_{k}})\longmapsto$}%
\hspace*{64pt}  \notag \\
&&\ \ \ \ \ \ 
\mbox{$(-1)^{r+s} \overset{n}{\underset{j=1}{\sum}}\,\big(
\frac{\partial}{\partial \overline{z}^{\overline{j}}} \iprod \alpha^{r,s+1}
\big)  g^{\alpha_{q}\overline{j}}(\overset{q-1}{\underset{k=1}{\otimes}}
dz^{\alpha_{k}})$}.\hspace*{40pt}  \notag
\end{eqnarray}

%

\begin{eqnarray}
&& 
\mbox{${I}^{-q}\ {\normalsize :}\ \Omega ^{r+1,s}(N, L\otimes
T^{q})\longrightarrow \Omega^{r,s}(N, L\otimes T^{(q-1)}),\text{ }q\in
\mathbb{N}$}\hspace*{136pt}  \label{2-2-3} \\
&& \ \ \ \ \ \ 
\mbox{$\alpha^{r+1,s} (\overset{q}{\underset{k=1}{\otimes}}
\frac{\partial}{\partial z^{\alpha_{k}}})=(-1)^{r+s}
\overset{n}{\underset{i=1}{\sum}}\, \big( (\frac{\partial}{\partial z^{i}}
\iprod \alpha^{r+1,s})\wedge
dz^{i}\big)(\overset{q}{\underset{k=1}{\otimes}} \frac{\partial}{\partial
z^{\alpha_{k}}} )\longmapsto$}\hspace*{68pt}  \notag \\
&& \ \ \ \ \ \ 
\mbox{$(-1)^{r+s}\overset{n}{\underset{i=1}{\sum}}\, \big(
\frac{\partial}{\partial z^{i}} \iprod \alpha^{r+1,s} \big)   \delta
_{\alpha _{q}}^{i}(\underset{k=1}{\overset{q-1}{\otimes }}\frac{\partial
}{\partial z^{\alpha _{k}}}),$}\hspace*{48pt}  \notag
\end{eqnarray}%
\begin{eqnarray}
&& 
\mbox{${\overline{I}}^{\, -q}\ {\normalsize :}\ \Omega ^{r,s+1}(N, L\otimes
T^{q})\longrightarrow \Omega^{r,s}(N, L\otimes T^{(q+1)}),\text{ }q\in \{0\}
\cup \mathbb{N}$}\hspace*{100pt}  \label{2-2-4} \\
&& \ \ \ \ \ \ 
\mbox{$ \alpha^{r,s+1} (\overset{q}{\underset{k=1}{\otimes}}
\frac{\partial}{\partial z^{\alpha_{k}}})=(-1)^{r+s}
\overset{n}{\underset{j=1}{\sum}}\, \big( (\frac{\partial}{\partial
\overline{z}^{\overline{j}}} \iprod \alpha^{r,s+1})\wedge
d\overline{z}^{\overline{j}}\big)(\overset{q}{\underset{k=1}{\otimes}}
\frac{\partial}{\partial z^{\alpha_{k}}} )\longmapsto$}\hspace*{64pt}  \notag
\\
&& \ \ \ \ \ \ 
\mbox{$(-1)^{r+s} \overset{n}{\underset{i, j=1}{\sum}}\, \big(
\frac{\partial}{\partial \overline{z}^{\overline{j}}} \iprod \alpha^{r,s+1}
\big)    g^{i\overline{j}}\big( \frac{\partial}{\partial z^{i}}\otimes
(\underset{k=1}{\overset{q}{\otimes }}\frac{\partial }{\partial z^{\alpha
_{k}}})\big).$}\hspace*{8pt}  \notag
\end{eqnarray}

\noindent These operators are globally well-defined.\newline

\noindent \hspace*{12pt} For any K\"{a}hler manifold $(N,\omega )$, let $%
\Omega ^{r,s}(N)$ be the space of $C^{\infty }(r,s)$-form on $N$ and $\ast $
be the (real) Hodge star operator associated with the K\"{a}hler form $%
\omega =i\,\overset{n}{\underset{\alpha ,\beta =1}{\dsum }}g_{\alpha 
\overline{\beta }}\,dz^{\alpha }\wedge d\overline{z}^{\overline{\beta }}$.
For $\psi =\underset{A_{r},B_{s}}{\sum }\,\psi _{A_{r}\overline{B}%
_{s}}(z)\,dz^{A_{r}}\wedge d\overline{z}^{\overline{B}_{s}}\in \Omega
^{r,s}(N)$ recall the definition of $\ast $ in \cite[p.150, (3.119)]{Kod86}: 
\begin{equation}
\ast \psi =i^{n}(-1)^{\frac{n(n-1)}{2}+nr}\underset{A_{r},B_{s}}{\sum }\,sgn%
\begin{pmatrix}
A_{r} & A_{n-r} \\ 
B_{s} & B_{n-s}%
\end{pmatrix}%
g(z)\,\psi ^{\overline{A_{r}}B_{s}}(z)\,dz^{B_{n-s}}\wedge d\overline{z}^{%
\overline{A}_{n-r}}  \label{stardef}
\end{equation}%
where 
\begin{equation}
g(z)=\det (g_{i\bar{j}}),\text{ }\psi ^{\overline{A_{r}}B_{s}}=\sum g^{%
\overline{\alpha }_{1}\lambda _{1}}...g^{\overline{\alpha }_{r}\lambda
_{r}}g^{\overline{\mu }_{1}\beta _{1}}...g^{\overline{\mu }_{s}\beta
_{s}}\psi _{\lambda _{1}...\lambda _{r}\overline{\mu }_{1}...\overline{\mu }%
_{s}}.\hspace*{80pt}  \label{gz}
\end{equation}%
\noindent Here $A_{r}=\alpha _{1}...\alpha _{r}$, $B_{s}=\beta _{1}...\beta
_{s}$ with $\alpha _{1}<...<\alpha _{r}$, $\beta _{1}<...<\beta _{s}$, $%
dz^{A_{r}}=dz^{\alpha _{1}}\wedge ...\wedge dz^{\alpha _{r}}$ and $d%
\overline{z}^{\overline{B}_{s}}=d\overline{z}^{\overline{\beta }_{1}}\wedge
...\wedge d\overline{z}^{\overline{\beta }_{s}}$. Moreover, for $%
A_{r}=\alpha _{1}...\alpha _{r}$ we put 
\begin{equation*}
A_{n-r}=\alpha _{r+1}...\alpha _{n}
\end{equation*}%
where $\alpha _{r+1}<...<\alpha _{n}$ and $\{\alpha _{1},...,\alpha
_{r},\alpha _{r+1},...\alpha _{n}\}$ is a permutation of $\{1,...,n\}$.
Similarly, we define $B_{n-s}$ if $B_{s}$ is given.\newline

\noindent From (\ref{stardef}), it is easy to see: 
\begin{align}
\ast & (dz^{j})=-i^{n}g(z)\sum\limits_{k}g^{\overline{k}j}\,dz^{k}\wedge (%
\mbox{$\underset{\alpha \neq k}{\bigwedge }$}dz^{\alpha }\wedge d\overline{z}%
^{\overline{\alpha }})\hspace*{120pt}  \label{star1} \\
\ast \,& d\overline{z}^{\overline{j}}=i^{n}g(z)\sum\limits_{k}\overline{g^{%
\overline{k}j}}\,d\overline{z}^{\overline{k}}\wedge (%
\mbox{$\underset{\alpha \neq
k}{\bigwedge}$}dz^{\alpha }\wedge d\overline{z}^{\overline{\alpha }})
\label{star2} \\
\ast & (i^{n}g(z)\,\mbox{$\underset{\alpha }{\bigwedge }$}\,dz^{\alpha
}\wedge d\overline{z}^{\overline{\alpha }})=\ast (\,\frac{\omega ^{n}}{n!}%
\,)=1\in \mathbb{C}.  \label{star3}
\end{align}%
\noindent Here dVol(z)$\displaystyle=i^{n}g(z)\,%
\mbox{$\underset{\alpha }{\bigwedge
}$}\,dz^{\alpha }\wedge d\overline{z}^{\overline{\alpha }}=\frac{\omega ^{n}%
}{n!}$ is the volume form.\newline

\noindent \hspace*{12pt} A standard inner product $<\ ,\,>$ for $\Omega
^{r,s}(N)$ can be defined as in (\cite[p.147, (3.115)]{Kod86}):\vspace*{8pt}%
\newline
\hspace*{28pt} For $\varphi =\frac{1}{r!s!}\sum \,\varphi _{\alpha
_{1}...\alpha _{r}\overline{\beta }_{1}...\overline{\beta }_{s}}dz^{\alpha
_{1}}\wedge ...\wedge dz^{\alpha _{r}}\wedge d\overline{z}^{\overline{\beta }%
_{1}}\wedge ...\wedge d\overline{z}^{\overline{\beta }_{s}}$,\newline
\hspace*{52pt}$\psi =\frac{1}{r!s!}\sum \,\psi _{\alpha _{1}...\alpha _{r}%
\overline{\beta }_{1}...\overline{\beta }_{s}}dz^{\alpha _{1}}\wedge
...\wedge dz^{\alpha _{r}}\wedge d\overline{z}^{\overline{\beta }_{1}}\wedge
...\wedge d\overline{z}^{\overline{\beta }_{s}}$, we first put 
\begin{equation*}
(\varphi ,\psi )(z)=\frac{1}{r!s!}\,\sum \,\varphi _{\alpha _{1}...\alpha
_{r}\overline{\beta }_{1}...\overline{\beta }_{s}}\overline{\psi }^{\,\alpha
_{1}...\alpha _{r}\overline{\beta }_{1}...\overline{\beta }_{s}}\hspace*{%
100pt}
\end{equation*}%
where 
\begin{equation*}
\overline{\psi }^{\,\alpha _{1}...\alpha _{r}\overline{\beta }_{1}...%
\overline{\beta }_{s}}=\sum g^{\overline{\lambda }_{1}\alpha _{1}}...g^{\,%
\overline{\lambda }_{r}\alpha _{r}}g^{\overline{\beta }_{1}\mu _{1}}...g^{\,%
\overline{\beta }_{s}\mu _{s}}\overline{\psi _{_{\lambda _{1}...\lambda _{r}%
\overline{\mu }_{1}...\overline{\mu }_{s}}}}\hspace*{108pt}
\end{equation*}%
and define the inner product $<\ ,\,>$ to be 
\begin{equation}
<\varphi ,\psi >=\int_{N}(\varphi ,\psi )(z)\,\frac{\omega ^{n}}{n!}\hspace*{%
192pt}  \label{2-16a}
\end{equation}

\noindent with (\cite[p.150, (3.122)]{Kod86}): 
\begin{equation*}
<\varphi ,\ \psi >=\int_{N}\varphi \wedge \overline{\ast \psi }%
=\int_{N}\varphi \wedge \ast \overline{\psi },\hspace*{30pt}\varphi ,\psi
\in \Omega ^{r,s}(N).\hspace*{20pt}
\end{equation*}%
One can naturally extend the inner product $<\ ,\,>$ from $\Omega ^{r,s}(N)$
to $\Omega ^{r,s}(N,E)$ for a holomorphic vector bundle $E$ with a Hermitian
metric \cite[(3.139) on p.159]{Kod86} including $E=L\otimes T^{q}$ or $%
L\otimes T^{\ast q}$ and we still use the same notation $<\ ,\,>$ \textit{%
without the subscript }$E$ . \newline

\noindent \hspace*{12pt} With this inner product $<\,,>$ on (compact) $N$
(with $E)$, one can easily show that the $L^{2}$-adjoints of $\ {\scriptsize %
\partial }^{q},\overline{\partial }^{q}$, ${\scriptsize \partial }^{-q}$,
and $\overline{\partial }^{\,-q}$ are given by: 
\begin{equation}
\begin{cases}
({\scriptsize \partial }^{q})^{\ast }={\scriptsize -\ast }\overline{\partial 
}^{q}{\scriptsize \ast }\text{,}\hspace*{20pt} & ({\scriptsize \partial }%
^{-q})^{\ast }={\scriptsize -\ast }\overline{\partial }^{-q}{\scriptsize %
\ast }\text{,} \\ 
({\scriptsize \overline{\partial }}^{q})^{\ast }={\scriptsize -\ast }%
\partial ^{q}{\scriptsize \ast }\text{,} & ({\scriptsize \overline{\partial }%
}^{-q})^{\ast }={\scriptsize -\ast }\partial ^{-q}{\scriptsize \ast }\text{,
\ }q\in \{0\}\cup \mathbb{N}.%
\end{cases}%
\hspace*{110pt}  \label{partial q star}
\end{equation}%
(cf. \cite[(3.142) on p.160]{Kod86}, \cite[p.227]{Wells08}). For instance,
our $\partial ^{q}$ corresponds to $D^{\prime }(=\partial +\theta )$ in the
notation of \cite[p.227]{Wells08}. It is understood that the star operator $%
\ast $ acts componentwise on vector-valued forms \cite[p.227]{Wells08}.

\noindent \hspace*{12pt} One defines the following Bochner-Kodaira type
Laplacians\footnote{%
These B-K Laplacians are not absolutely independent of one another, since
the connection on a holomorphic bundle $E$ and that on its dual $E^{\ast }$
are closely related (compare (7-7-1) in the proof of Theorem \ref{T-7-1}).
To make our treatment \textit{uniform}, we find it natural to ignore their
dependence and work on them as if they were independent ones.} (%
\mbox{cf.
\cite[p.296]{Prieto06(a)}}): for $q\in \{0\}\cup \mathbb{N}$%
\begin{equation}
\begin{cases}
\Delta ^{q}:=(\partial ^{q})^{\ast }{\scriptsize \partial }^{q}\text{, }%
\Delta _{q}:=(\overline{\partial }^{q})^{\ast }{\scriptsize \overline{%
\partial }^{q}}\text{\ \ \ \ \ \ \ \ \ \ \ \ \ \ \ \ \ \ on }\Omega
^{0,0}(N,L\otimes T^{\ast q}) \\ 
\Delta ^{-q}:=({\scriptsize \partial }^{-q})^{\ast }\text{ }{\scriptsize %
\partial }^{-q}\text{, }\Delta _{-q}:=({\scriptsize \overline{\partial }}%
^{\,-q})^{\ast }\text{ }{\scriptsize \overline{\partial }}^{\,-q}\text{\ \ \
\ \ \ on }\Omega ^{0,0}(N,L\otimes T^{q}).%
\end{cases}
\label{Delta q}
\end{equation}



Let $V$ be a complex vector space of dimension $n$ and $\Lambda \subseteq V$
a discrete lattice of rank $2n$ with integral basis $\{\lambda
_{1},...,\lambda _{2n}\}$. Assume that $M=V/\Lambda $ is an Abelian variety
with the Hodge metric%
\begin{equation}
\omega ^{h}=\frac{i}{2}\overset{n}{\underset{\alpha ,\beta =1}{\sum }}%
\,W_{\alpha \beta }\,dz^{\alpha }\wedge d\overline{z}^{\overline{\beta }}
\label{HodgeMetric}
\end{equation}%
\noindent given in the complex coordinates $(z^{1},...,z^{n})$ on $V$ with
respect to the complex basis $v_{\alpha }$ where $\lambda _{\alpha }=\delta
_{\alpha }v_{\alpha }$, $\lambda _{n+\alpha }=\sum \,\tau _{\alpha k}v_{k}$
with $\delta _{\alpha }\in \mathbb{N},\delta _{\alpha }|\delta _{\alpha +1}$
for $\alpha =1,...,n-1$, and $(W_{\alpha \beta })=(\func{Im}\tau _{\alpha
\beta })^{-1}$ a symmetric and positive-definite matrix (\cite[pp.191, 306]%
{GH84}). For our setting, we would like to be more flexible and equip $M$
with the K\"{a}hler metric 
\begin{equation}
\omega =\frac{2\pi }{B}\,\omega ^{h}=\frac{\pi i}{B}\overset{n}{\underset{%
\alpha ,\beta =1}{\sum }}\,W_{\alpha \beta }\,dz^{\alpha }\wedge d\overline{z%
}^{\overline{\beta }}\hspace*{10pt}\mbox{for some $B>0$.}
\end{equation}%
\noindent Since the matrix $(W_{\alpha \beta })$ is symmetric and
positive-definite, after some linear change of complex coordinates we can
express the K\"{a}hler metric as (retaining the same $z^{\alpha }$-notation)%
\begin{equation}
\omega =i\overset{n}{\underset{\alpha =1}{\sum }}\,dz^{\alpha }\wedge d%
\overline{z}^{\overline{\alpha }}\hspace*{10pt}\mbox{on $M$.}  \label{omega}
\end{equation}

\noindent \hspace*{12pt} On a K\"{a}hler manifold $N$ we denote the
holomorphic tangent, cotangent bundle by $T$, $T^{\ast }$ respectively, and $%
\otimes ^{q}T$, $\otimes ^{q}T^{\ast }$ by $T^{q}$, $T^{\ast q}$
respectively, $q\in \{0\}\cup \mathbb{N}$. Write $\nabla $ for the
associated Levi-Civita connection on $T$. Let $L\rightarrow N$ be a
holomorphic line bundle with a Hermitian metric $h_{L}$. On $L$ there exists
a unique holomorphic connection $\nabla ^{L}$ (i.e.\thinspace\ ${\nabla ^{L}}%
^{(0,1)}=\overline{\partial }_{L}$) compatible with $h_{L}$ (called the 
\textit{Chern connection} \cite[p.11]{MM07} ). If $c_{1}(L)=[\omega ^{h}]$
on an Abelian variety $M,$ we choose the above $h_{L}$ to be of the
curvature 
\begin{equation}
\Theta (=\overline{\partial }\partial \log (h_{L}))=-iB\omega .
\label{curvature}
\end{equation}%
\noindent With the connections $\nabla ^{L}$ on $L$ and $\nabla $ on $%
T^{\ast }$ (dual to $\nabla $ on $T$) we have the induced connections $%
\nabla ^{q}$ on $L\otimes T^{\ast q}$ and $\nabla ^{-q}$ on $L\otimes T^{q}$%
, $q\in \{0\}\cup \mathbb{N}$.\newline

\textbf{(C.C.)} Curvature condition on $L$ over $N$ with a Kahler form $%
\omega _{N}$ (the line bundle case of Hermitian-Yang-Mills connections as
noted in the Introduction)$:$ we assume that we can equip $L$ with a
Hermitian metric $h_{L}$ and $\nabla ^{L}$ the Chern connection on $L$ such
that the curvature 
\begin{equation*}
\Theta (=\overline{\partial }\partial \log (h_{L}))=-iB\omega _{N}\text{ for
some }B>0.
\end{equation*}

By \textbf{c.g.c. }at $p$ we mean a system of complex geodesic coordinates
around $p$, that is, $g_{i\overline{j}}(p)=\delta _{i\overline{j}}$ and $%
dg_{i\overline{j}}(p)=0.$ The condition $dg_{i\overline{j}}(p)=0$ is not
always computationally needed; let us still adopt the c.g.c. for the sake of
convenience.

\begin{lemma}
\label{Lemma 2.1} Let $s$ be a (local) holomorphic section of $L$ on $N$
with (C.C.). For $1\leq i,j\leq n$\newline
\noindent i) $s_{\overline{j}}=s_{\overline{j}i}=0,$\hspace*{4pt} ii) $s_{i%
\overline{j}}=-Bg_{i\overline{j}}\,s$\hspace*{4pt} iii) $s_{ij}=s_{ji}$
where in $ii)$ $\sum_{i,j}ig_{i\overline{j}}dz^{i}\wedge d\bar{z}^{\overline{%
j}}$ $=$ $\omega _{N}.$ (On the Abelian variety $M,$ $g_{i\overline{j}%
}\equiv \delta _{i\bar{j}}$ while on $\mathbb{P}^{n},$ $g_{i\overline{j}}$
is the Fubini-Study metric$,$ see (\ref{4-3}).$)$
\end{lemma}

\begin{proof}
\begin{equation*}
\nabla _{\frac{\partial }{\partial \overline{z}^{\overline{j}}}}^{L}\nabla _{%
\frac{\partial }{\partial z^{i}}}^{L}\,s=%
\mbox{$
\frac{\partial}{\partial \overline{z}^{\overline{j}}} 
\frac{\partial}{\partial z^{i}} (\log h_{L}(z)\, )  s $}=Bg_{i\overline{j}%
}dz^{i}\wedge d\bar{z}^{\overline{j}}(\frac{\partial }{\partial \bar{z}^{%
\overline{j}}},\frac{\partial }{\partial z^{i}})s=-Bg_{i\overline{j}}s.
\end{equation*}
\end{proof}

\noindent \hspace*{12pt} Unless specifically stated otherwise, we assume
that $s$ is\textit{\ holomorphic} for the rest of this paper.

\begin{lemma}
\label{Lemma 2.2} On $N,$ for the local sections $f_{\alpha _{1}...\alpha
_{q}}s\otimes (\overset{q}{\underset{k=1}{\otimes }}dz^{\alpha _{k}})$, $%
f^{\alpha _{1}...\alpha _{q}}s\otimes (\overset{n}{\underset{k=1}{\otimes }}%
\frac{\partial }{\partial z^{\alpha _{k}}})$ of $L\otimes T^{\ast q}$, $%
L\otimes T^{q}$ respectively,\ $q\in \{0\}\cup \mathbb{N}$, we have the
following (w.r.t. some c.g.c. at $p$), where the curvature condition (C.C.)
is assumed for $i)$ and $iii)$ below


$i)$ $\Delta ^{q}(f_{\alpha _{1}...\alpha _{q}}s{{\normalsize \otimes }(%
\underset{k=1}{\overset{q}{\otimes }}}{\normalsize dz}^{\alpha _{k}}))=-%
\underset{i}{\sum }(f_{\alpha _{1}...\alpha _{q},i\overline{i}}s+f_{\alpha
_{1}...\alpha _{q},\overline{i}}s_{i}-Bf_{\alpha _{1}...\alpha
_{q}}s)\otimes {(\overset{q}{\underset{k=1}{\otimes }}}{\normalsize dz}%
^{\alpha _{k}}).$\newline

$ii)$ $\Delta _{q}(f_{\alpha_{1}...\alpha_{q}}s{{\normalsize \otimes }(%
\overset{q}{\underset{k=1}{\otimes }}}{\normalsize dz}^{\alpha _{k}}))=-%
\underset{i}{\sum } (f_{\alpha_{1}...\alpha_{q},\overline{i}i}
s+f_{\alpha_{1}...\alpha_{q},\overline{_{i}}} {s}_{i} )\otimes{(\underset{k=1%
}{\overset{q}{\otimes }}} dz^{\alpha _{k}} )$.\newline

$iii)$ $\Delta ^{-q}(f^{\alpha _{1}...\alpha _{q}}s{{\normalsize \otimes }(%
\overset{q}{\underset{k=1}{\otimes }}}\frac{\partial }{\partial z^{\alpha
_{k}}}))=-\underset{i}{\sum }({f^{\alpha _{1}...\alpha _{q}}}_{,i\overline{i}%
}s+{f^{\alpha _{1}...\alpha _{q}}}_{,\overline{i}}s_{i}-Bf^{\alpha
_{1}...\alpha _{q}}s)\otimes {(\overset{q}{\underset{k=1}{\otimes }}}\frac{%
\partial }{\partial z^{\alpha _{k}}}).$\newline

$iv)$ $\Delta _{-q}(f^{\alpha_{1}...\alpha_{q}} s{{\normalsize \otimes }(%
\overset{q}{\underset{k=1}{\otimes }}}\frac{\partial }{\partial z^{\alpha
_{k}}}))=-\underset{i}{\sum }({f^{\alpha_{1}...\alpha_{q}}}_{,\overline{i}i}
s+{f^{\alpha_{1}...\alpha_{q}}}_{,\overline{i}} s_{i})\otimes {(\overset{q}{%
\underset{k=1}{\otimes }}}\frac{\partial }{\partial z^{\alpha _{k}}})$. 
\end{lemma}

\begin{proof}
Let us show the first equation. We have (with $\nabla g_{i\overline{j}%
}\equiv 0)$%
\begin{align*}
& \hspace*{24pt}\Delta ^{q}(f_{\alpha _{1},,,\alpha _{q}}s\otimes (\underset{%
k=1}{\overset{q}{\otimes }}{\normalsize dz}^{\alpha _{k}}))=-\ast \overline{%
\partial }^{q}\ast \partial ^{q}(f_{\alpha _{1}...\alpha _{q}}s\otimes (%
\underset{k=1}{\overset{q}{\otimes }}\,{\normalsize dz}^{\alpha _{k}})) \\
& =-\ast \overline{\partial }^{q}\ast \Big(\mbox{$\underset{i}{\sum }$}%
\,dz^{i}\otimes (f_{\alpha _{1}...\alpha _{q},i}s+f_{\alpha _{1}...\alpha
_{q}}s_{i})\otimes (\underset{k=1}{\overset{q}{\otimes }}{\normalsize dz}%
^{\alpha _{k}})\Big) \\
& \overset{(\ref{star1})}{=}i^{n}\ast \overline{\partial }^{q}\Big(%
\mbox{$\underset{i,k}{\sum }$}gg^{\overline{k}i}(dz^{k}\wedge (%
\mbox{$\underset{\alpha \neq k}{\bigwedge} $}dz^{\alpha }\wedge d\overline{z}%
^{\overline{\alpha }}))\otimes (f_{\alpha _{1}...\alpha _{q},i}s+f_{\alpha
_{1}...\alpha _{q}}s_{i})\otimes (\underset{k=1}{\overset{q}{\otimes }}%
{\normalsize dz}^{\alpha _{k}})\Big) \\
& \overset{\text{at }p}{=}-i^{n}\ast \Big(\mbox{$\underset{i}{\sum } $}gg^{%
\bar{\imath}i}(\mbox{$\underset{\alpha
}{\bigwedge }$}dz^{\alpha }\wedge d\overline{z}^{\overline{\alpha }%
})(f_{\alpha _{1}...\alpha _{q},i\overline{i}}s+f_{\alpha _{1}...\alpha
_{q},i}\otimes s_{\overline{i}}+f_{\alpha _{1}...\alpha _{q},\overline{i}%
}s_{i}+f_{\alpha _{1}...\alpha _{q}}s_{i\overline{i}})\otimes (\underset{k=1}%
{\overset{q}{\otimes }}{\normalsize dz}^{\alpha _{k}})\Big) \\
& \overset{Lem.\ref{Lemma 2.1}\text{ }ii)}{=}-\mbox{$\underset{i}{\sum }$}%
(f_{\alpha _{1}...\alpha _{q},i\overline{i}}s+f_{\alpha _{1}...\alpha _{q},%
\overline{i}}{\normalsize s}_{i}-Bf_{\alpha _{1}...\alpha _{q}}{\normalsize s%
})\otimes (\underset{k=1}{\overset{q}{\otimes }}{\normalsize dz}^{\alpha
_{k}}).
\end{align*}

\noindent The others are similarly proved (via Lemma \ref{Lemma 2.1} $i))$.
Note that there is no $B$ involved in $ii)$ and $iv)$ as the $\overline{i}$%
-covariant derivative is taken first for these two cases.
\end{proof}

Propositions \ref{10.2PZ} and \ref{10.1} in the following can be viewed as
new types of Bochner-Kodaira identities\footnote{%
The familiar Bochner-Kodaira identities give the difference between two
second order (elliptic) operators on a fixed Hilbert space. Here, our
identity is similar in nature despite that our operators are of the third
order.} involving contraction operators. For the sake of clarity we prove
them in full details. See Propositions \ref{Prop 4.2} and \ref{Prop 4.3} as
analogues on $\mathbb{P}^{n},$ for which the curvature of $\mathbb{P}^{n}$
enters. We most often use the $T$-valued version ($T:$ tangent bundle); see
Theorem \ref{q+n+k} for the $T^{\ast }$-valued one. Another notable
difference between Abelian varieties and $\mathbb{P}^{n}$ is that on $%
\mathbb{P}^{n}$ those BK identities usually hold only for \textit{symmetric} 
$T^{\ast }$ (or $T$)-tensors while on Abelian varieties the symmetric
condition is not always needed.

\begin{proposition}
\label{10.2PZ} ($T^{\ast }$-valued) For all $q\in \{0\}\cup \mathbb{N}$,%
\begin{eqnarray}
&&{\scriptsize \Delta }^{q}{\scriptsize \overline{I}}^{(q+1)}\overline{%
\partial }^{(q+1)}{\scriptsize -\overline{I}}^{(q+1)}\overline{\partial }%
^{(q+1)}{\scriptsize \Delta }^{(q+1)}  \label{-BBB} \\
&&{\scriptsize =-B\overline{I}^{(q+1)}\overline{\partial }^{(q+1)}}:\Omega
^{0,0}\left(M, L\otimes T^{\ast (q+1)}\right)\rightarrow\Omega
^{0,0}\left(M, L\otimes T^{\ast q}\right).  \notag
\end{eqnarray}
\end{proposition}

\begin{proof}
We first derive formulas (see $(\ref{2-3-1})$ and $(\ref{2-3-2})$ below) for
any K\"{a}hler manifold $N$ with (C.C.); see Proposition \ref{Prop 4.2} for
use of them on $\mathbb{P}^{n}.$ Choosing c.g.c. $z_{i}$ at $p$ and letting $%
f_{\alpha _{1}...\alpha _{q+1}}s\otimes (\overset{q+1}{\underset{k=1}{%
\otimes }}dz^{\alpha _{k}})$ be a local section of $L\otimes T^{\ast (q+1)},$
we have 
\begin{eqnarray}
&&{\normalsize \Delta }^{q}{\normalsize \overline{I}^{(q+1)}\overline{%
\partial }^{(q+1)}}({\normalsize f_{\alpha _{1}...\alpha _{q+1}}s\otimes (}%
\underset{k=1}{\overset{q+1}{\otimes }}{\normalsize dz}^{\alpha _{k}})%
{\normalsize )}={\normalsize \Delta }^{q}{\normalsize \overline{I}}^{(q+1)}%
\big(%
\mbox{$\underset{j}{\sum
}d\overline{z}^{\overline{j}}(f_{\alpha_{1}...\alpha_{q+1},\overline{j}} s)
\otimes {\normalsize (}\underset{k=1}{\overset{q+1}{\otimes }}{\normalsize
dz}^{\alpha _{k}})$}\big)  \label{2-3-1} \\
&& \overset{(\ref{2-2-2})}{=}\Delta ^{q}\big(%
\mbox{$\underset{j,l}{\sum }\,  (\frac{\partial}{\partial \overline{z}^{\overline{l}}} \iprod
d\overline{z}^{\overline{j}})\otimes(f_{\alpha_{1}...\alpha_{q+1},\overline{j}} s) \otimes g^{\overline{l}\alpha_{q+1}} {\normalsize
(}\underset{k=1}{\overset{q}{\otimes }}{\normalsize dz}^{\alpha _{k}})$}\big)
\notag \\
&& =\ \Delta ^{q}\big((f_{\alpha _{1}...\alpha _{q+1},\overline{\alpha }%
_{q+1}}{\normalsize s})\otimes {\normalsize (}\underset{k=1}{\overset{q}{%
\otimes }}{\normalsize dz}^{\alpha _{k}})\big)  \notag \\
&& \overset{Lem.\ref{Lemma 2.2}\,i)}{=}%
\mbox{$ -\underset{i}{\sum
}(f_{\alpha_{1}...\alpha_{q+1},\overline{\alpha}_{q+1}i\overline{i}}
s+f_{\alpha_{1}...\alpha_{q+1},\overline{\alpha}_{q+1}\overline{i}}{\normalsize  s}_{i} -Bf_{\alpha_{1}...\alpha_{q+1},\overline{\alpha}_{q+1}}s
{\normalsize )\otimes}( \underset{k=1}{\overset{q}{\otimes }}dz^{\alpha
_{k}})$}.  \notag
\end{eqnarray}%
\noindent We also have 
\begin{eqnarray}
&&\ \ \ \ \ {\normalsize \overline{I}}^{(q+1)}\overline{\partial }^{(q+1)}%
{\normalsize \Delta }^{(q+1)}\big({\normalsize f_{\alpha _{1}...\alpha
_{q+1}}s\otimes (}\underset{k=1}{\overset{q+1}{\otimes }}{\normalsize dz}%
^{\alpha _{k}}){\normalsize \big)}\hspace*{204pt}  \label{2-3-2} \\
&&\ \ \ \ \ \overset{Lem.\ref{Lemma 2.2}\,i)}{=}{\normalsize -\overline{I}}%
^{(q+1)}\overline{\partial }^{(q+1)}\big(%
\mbox{$\underset{i}{\sum }{\normalsize
(f}_{\alpha_{1}...\alpha_{q+1},i\overline{i}}{\normalsize 
s+f}_{\alpha_{1}...\alpha_{q+1},\overline{i}}{\normalsize  s}_{i}
-Bf_{\alpha_{1}...\alpha_{q+1}} s )\otimes (
\underset{k=1}{\overset{q+1}{\otimes }}dz^{\alpha _{k}})$}\big)\hspace*{60pt}
\notag \\
&&\ \ \ \ \ =-\overline{I}^{q+1}\big(%
\mbox{$\underset{i,j}{\sum }\,
d\overline{z}^{\overline{j}}
\otimes(f_{\alpha_{1}...\alpha_{q+1},i\overline{i}\,\overline{j}}
s+f_{\alpha_{1}...\alpha_{q+1},\overline{i}\,\overline{j}}{\normalsize 
s}_{i})\otimes( \underset{k=1}{\overset{q+1}{\otimes }}dz^{\alpha _{k}})$}%
\hspace*{146pt}  \notag \\
&&\ \ \ \ \ \hspace*{20pt}+%
\mbox{$ \underset{i,j}{\sum }\, d\overline{z}^{\overline{j}}
\otimes ( f_{\alpha_{1}...\alpha_{q+1},\overline{i}}s_{i\overline{j}}
-Bf_{\alpha_{1}...\alpha_{q+1},\,\overline{j}}  s {\normalsize
)}\otimes(\underset{k=1}{\overset{q+1}{\otimes }}dz^{\alpha_{k}})$}\big)%
\hspace*{144pt}  \notag \\
&&\ \ \ \ \ \overset{(\ref{2-2-2})}{=}-\big(%
\mbox{$\underset{i,j,l}{\sum }\,
(\frac{\partial}{\partial \overline{z}^{\overline{l}}} \iprod
d\overline{z}^{\overline{j}})
\otimes(f_{\alpha_{1}...\alpha_{q+1},i\overline{i}\,\overline{j}}
s+f_{\alpha_{1}...\alpha_{q+1},\overline{i}\,\overline{j}}{\normalsize 
s}_{i}))g^{\alpha_{q+1}\overline{l}}\otimes(
\underset{k=1}{\overset{q}{\otimes }}dz^{\alpha _{k}})$}\hspace*{52pt} 
\notag \\
&&\ \ \ \ \ \hspace*{20pt}+%
\mbox{$ \underset{i,j,l}{\sum }\,(\frac{\partial}{\partial
\overline{z}^{\overline{l}}} \iprod d\overline{z}^{\overline{j}}) \otimes (
f_{\alpha_{1}...\alpha_{q+1},\overline{i}}s_{i\overline{j}}
-Bf_{\alpha_{1}...\alpha_{q+1},\,\overline{j}}  s 
)g^{\alpha_{q+1}\overline{l}}\otimes(\underset{k=1}{\overset{q}{\otimes
}}dz^{\alpha_{k}})$}\big)\hspace*{80pt}  \notag \\
&&\ \ \ \ \ =-\mbox{$\underset{i}{\sum }$}{\normalsize (f}_{\alpha
_{1}...\alpha _{q+1},i\overline{i}\overline{\alpha }_{q+1}}s+f_{\alpha
_{1}...\alpha _{q+1},\overline{i}\overline{\alpha }_{q+1}}{\normalsize s}%
_{i})\otimes (\underset{k=1}{\overset{q}{\otimes }}dz^{\alpha _{k}})\hspace*{%
180pt}  \notag \\
&&\ \ \ \ \ \hspace*{20pt}-\mbox{$\underset{i}{\sum }$}({\normalsize f}%
_{\alpha _{1}...\alpha _{q+1},\overline{i}}s_{i\overline{\alpha }%
_{q+1}}-Bf_{\alpha _{1}...\alpha _{q+1},\overline{\alpha }_{q+1}}s)\otimes (%
\underset{k=1}{\overset{q}{\otimes }}dz^{\alpha _{k}}).\hspace*{156pt} 
\notag
\end{eqnarray}%
Therefore, $(\ref{2-3-1})-(\ref{2-3-2})$: 
\begin{align*}
& \big(\Delta ^{q}{\normalsize \overline{I}}^{(q+1)}\overline{\partial }%
^{(q+1)}{\normalsize -\overline{I}}^{(q+1)}\overline{\partial }^{(q+1)}%
{\normalsize \Delta }^{(q+1)}\big)\big(f_{\alpha _{1}...\alpha
_{q+1}}s\otimes (\underset{k=1}{\overset{q+1}{\otimes }}{\normalsize dz}%
^{\alpha _{k}})\big)\hspace*{168pt} \\
& =-B(f_{\alpha _{1}...\alpha _{q+1},\overline{\alpha }_{q+1}}s)\otimes (%
\underset{k=1}{\overset{q}{\otimes }}{\normalsize dz}^{\alpha _{k}})=-B\big(%
\mbox{$\underset{j,k}{\sum}(  \frac{\partial}{\partial
\overline{z}^{\overline{k}}}\iprod d\overline{z}^{\overline{j}})$}\otimes
(f_{\alpha _{1}...\alpha _{q+1},\,\overline{j}}s)g^{\overline{\alpha }%
_{q+1}k}\otimes (\underset{k=1}{\overset{q}{\otimes }}{\normalsize dz}%
^{\alpha _{k}})\big) \\
& \overset{(\ref{2-2-2})}{=}-B\overline{I}^{q+1}\big(\mbox{$\underset{j}{%
\sum}$}(d\overline{z}^{\overline{j}}\otimes (f_{\alpha _{1}...\alpha
_{q+1},\,\overline{j}}s))\otimes (\underset{k=1}{\overset{q+1}{\otimes }}%
{\normalsize dz}^{\alpha _{k}})\big)=-B\overline{I}^{q+1}\overline{\partial }%
^{q}\big((f_{\alpha _{1}...\alpha _{q+1}}s)\otimes (\underset{k=1}{\overset{%
q+1}{\otimes }}{\normalsize dz}^{\alpha _{k}})\big).\hspace*{40pt}
\end{align*}
\end{proof}

\begin{proposition}
\label{10.1} \noindent For all $q\in \{0\}\cup \mathbb{N}$, 
\begin{eqnarray}
&&{\scriptsize \Delta }_{-q}{\scriptsize I}^{-(q+1)}{\scriptsize \partial }%
^{-(q+1)}{\scriptsize -I}^{-(q+1)}{\scriptsize \partial }^{-(q+1)}%
{\scriptsize \Delta }_{-(q+1)}  \label{BBB} \\
&&{\scriptsize =BI}^{-(q+1)}{\scriptsize \partial }^{-(q+1)}:%
\mbox{$\Omega ^{0,0}\left(M,
L\otimes T^{(q+1)}\right) \rightarrow \Omega ^{0,0}\left(M,
L\otimes T^{q}\right) $.}  \notag
\end{eqnarray}
\end{proposition}

\begin{proof}
At $p$ we compute the first term of (\ref{BBB}), using c.g.c. at $p,$ (the
two formulas (\ref{2-4-1}) and (\ref{2-4-2}) below will also be used for $%
\mathbb{P}^{n},$ see Proposition \ref{Prop 4.3})%
\begin{eqnarray}
&&\Delta _{-q}{\normalsize I}^{-(q+1)}{\normalsize \partial }^{-(q+1)}%
{\normalsize \big(f^{\alpha _{1}...\alpha _{q+1}}s}\underset{k=1}{\overset{%
q+1}{\otimes }}\mbox{$(\frac{\partial }{\partial z^{\alpha _{k}}} )\big)$}%
\hspace*{228pt}  \label{2-4-1} \\
&=&{\normalsize \Delta }_{-q}{\normalsize I}^{-(q+1)}\big(%
\mbox{$\underset{i}{\sum}( ({f^{\alpha_{1}...\alpha_{q+1}}}_{,i} 
s+f^{\alpha_{1}...\alpha_{q+1}} s_{i}) \otimes  dz^{i} ) \otimes
(\underset{k=1}{\overset{q+1}{\otimes }}\mbox{$(\frac{\partial }{\partial
z^{\alpha _{k}}} )$}) $}\big)\hspace*{112pt}  \notag \\
&&\overset{(\ref{2-2-3})}{=}{\normalsize \Delta }_{-q}\big(%
\mbox{$\underset{j,i}{\sum}(\frac{\partial }{\partial
z^{j}}$}\lrcorner \text{ }({f^{\alpha _{1}...\alpha _{q+1}}}_{,i}s+f^{\alpha
_{1}...\alpha _{q+1}}s_{i})\otimes dz^{i})\delta _{\alpha _{q+1}}^{j}\otimes
(\underset{k=1}{\overset{q}{\otimes }}%
\mbox{$(\frac{\partial }{\partial
z^{\alpha _{k}}} )) $}\big)  \notag \\
&=&\Delta _{-q}\big(({f^{\alpha _{1}...\alpha _{q+1}}}_{,\alpha
_{q+1}}s+f^{\alpha _{1}...\alpha _{q+1}}s_{\alpha _{q+1}})\otimes (\underset{%
k=1}{\overset{q}{\otimes }}%
\mbox{$(\frac{\partial }{\partial z^{\alpha
_{k}}} ))$}\big)\hspace*{168pt}  \notag \\
&&\overset{Lem.\ref{Lemma 2.2}\,iv)}{=}{\normalsize -}\mbox{$\underset{i}{%
\sum }$}({f^{\alpha _{1}...\alpha _{q+1}}}_{,\alpha _{q+1}\overline{i}i}s+{%
f^{\alpha _{1}...\alpha _{q+1}}}_{,\alpha _{q+1}\overline{i}}s_{i}+{%
f^{\alpha _{1}...\alpha _{q+1}}}_{,\overline{i}i}s_{\alpha _{q+1}}%
{\normalsize )}\otimes (\overset{q}{\underset{k=1}{\otimes }}%
\mbox{$(\frac{\partial }{\partial z^{\alpha _{k}}} )$})\hspace*{56pt}  \notag
\\
&&\hspace*{40pt}-\mbox{$\underset{i}{\sum }$}({f^{\alpha _{1}...\alpha
_{q+1}}}_{,\overline{i}}s_{\alpha _{q+1}i}+{f^{\alpha _{1}...\alpha _{q+1}}}%
_{,i}s_{\alpha _{q+1}\overline{i}}+f^{\alpha _{1}...\alpha _{q+1}}s_{\alpha
_{q+1}\overline{i}i})\otimes (\overset{q}{\underset{k=1}{\otimes }}%
\mbox{$(\frac{\partial }{\partial z^{\alpha _{k}}} )$}).\hspace*{56pt} 
\notag
\end{eqnarray}%
\noindent The second term of (\ref{BBB}) is 
\begin{eqnarray}
&&I^{-(q+1)}{\normalsize \partial }^{-(q+1)}{\normalsize \Delta }_{-(q+1)}%
{\normalsize \big({f^{\alpha _{1}...\alpha _{q+1}}}s\underset{k=1}{\overset{%
q+1}{\otimes }}\mbox{$(\frac{\partial }{\partial z^{\alpha _{k}}} )$}}\big)%
\hspace*{180pt}  \label{2-4-2} \\
&&\overset{Lem.\ref{Lemma 2.2}\,iv)}{=}{\normalsize -I}^{-(q+1)}{\normalsize %
\partial }^{-(q+1)}\big(\mbox{$ \underset{i}{\sum }$}({f^{\alpha
_{1}...\alpha _{q+1}}}_{,\overline{i}i}s+{f^{\alpha _{1}...\alpha _{q+1}}}_{,%
\overline{i}}{\normalsize s}_{i})\otimes {(\overset{q+1}{\underset{k=1}{%
\otimes }}}\mbox{$(\frac{\partial }{\partial z^{\alpha _{k}}} )$})\big)%
\hspace*{92pt}  \notag \\
&=&{\normalsize -I}^{-(q+1)}\big(\mbox{$\underset{i,l}{\sum } $}(({f^{\alpha
_{1}...\alpha _{q+1}}}_{,\overline{i}il}{\normalsize s+{f^{\alpha
_{1}...\alpha _{q+1}}}_{,\overline{i}i}s}_{l}\hspace*{216pt}  \notag \\
&&\ \ \ \ \ +{f^{\alpha _{1}...\alpha _{q}}}_{,\overline{i}l}{\normalsize s}%
_{i}+{f^{\alpha _{1}...\alpha _{q+1}}}_{,\overline{i}}{\normalsize s}%
_{il})\otimes dz^{l})\otimes {(\overset{q+1}{\underset{k=1}{\otimes }}}%
\mbox{$(\frac{\partial }{\partial z^{\alpha _{k}}} )$})\big)\hspace*{160pt} 
\notag \\
&&\overset{(\ref{2-2-3})}{=}-\big(%
\mbox{$\underset{j,i,l}{\sum }
(\frac{\partial}{\partial z^{j}} $}\mathbin{\lrcorner}(({f^{\alpha
_{1}...\alpha _{q+1}}}_{,\overline{i}il}{\normalsize s+{f^{\alpha
_{1}...\alpha _{q+1}}}_{,\overline{i}i}s}_{l}\hspace*{176pt}  \notag \\
&&\ \ \ \ \ +{f^{\alpha _{1}...\alpha _{q}}}_{,\overline{i}l}{\normalsize s}%
_{i}+{f^{\alpha _{1}...\alpha _{q+1}}}_{,\overline{i}}{\normalsize s}%
_{il})\otimes dz^{l}))\delta _{q+1}^{j}\otimes {(\overset{q}{\underset{k=1}{%
\otimes }}}\mbox{$(\frac{\partial }{\partial z^{\alpha _{k}}} )$})\big)%
\hspace*{136pt}  \notag \\
&=&{\normalsize -}\mbox{$\underset{i}{\sum}$}\,({f^{\alpha _{1}...\alpha
_{q+1}}}_{,\overline{i}i\alpha _{q+1}}s+{f^{\alpha _{1}...\alpha _{q+1}}}_{,%
\overline{i}i}s_{\alpha _{q+1}}\hspace*{228pt}  \notag \\
&&\ \ \ \ \ +{f^{\alpha _{1}...\alpha _{q+1}}}_{,\overline{i}\alpha
_{q+1}}s_{i}+{f^{\alpha _{1}...\alpha _{q+1}}}_{,\overline{i}}s_{i\alpha
_{q+1}})\otimes (\overset{q}{\underset{k=1}{\otimes }}%
\mbox{$(\frac{\partial }{\partial
z^{\alpha _{k}}} )$}).\hspace*{160pt}  \notag
\end{eqnarray}%
\noindent Taking (\ref{2-4-1}) ${\normalsize -}$ (\ref{2-4-2}) on $M$ with $%
s_{\alpha _{q+1}\overline{i}i}=-Bg_{\alpha _{q+1}\overline{i}}s_{i}$ by
Lemma \ref{Lemma 2.1}, we have%
\begin{eqnarray*}
&&\big(\Delta _{-q}{\normalsize I}^{-(q+1)}{\normalsize \partial }^{-(q+1)}%
{\normalsize -I}^{-(q+1)}{\normalsize \partial }^{-(q+1)}{\normalsize \Delta 
}_{-(q+1)}{\normalsize \big)(f^{\alpha _{1}...\alpha _{q+1}}s}\underset{k=1}{%
\overset{q+1}{\otimes }}%
\mbox{$(\frac{\partial
}{\partial z^{\alpha _{k}}} )$})\hspace*{60pt} \\
&&\overset{}{=}\big({f^{\alpha _{1}...\alpha _{q+1}}}_{,\alpha
_{q+1}}Bs+f^{\alpha _{1}...\alpha _{q+1}}Bs_{\alpha _{q+1}}\big)\otimes (%
\overset{q}{\underset{k=1}{\otimes }}%
\mbox{$(\frac{\partial
}{\partial z^{\alpha _{k}}} )$}),
\end{eqnarray*}

\noindent which is, as similar in (\ref{2-4-1}), 
\begin{eqnarray*}
&&{\normalsize BI}^{-(q+1)}\big(\mbox{$\underset{i}{\sum}$}(({f^{\alpha
_{1}...\alpha _{q+1}}}_{,i}s+f^{\alpha _{1}...\alpha _{q+1}}s_{i})\otimes
dz^{i}){\normalsize \otimes }(\overset{q+1}{\underset{k=1}{\otimes }}%
\mbox{$(\frac{\partial }{\partial
z^{\alpha _{k}}} )$})\big)\hspace*{98pt} \\
&& ={\normalsize BI}^{-(q+1)}{\normalsize \partial }^{-(q+1)}\big(f^{\alpha
_{1}...\alpha _{q+1}}s{\normalsize \otimes }(\overset{q+1}{\underset{k=1}{%
\otimes }}\mbox{$(\frac{\partial }{\partial z^{\alpha _{k}}} )$})\big).%
\hspace*{200pt}
\end{eqnarray*}
\end{proof}

\begin{notation}
\label{2.6} 
Let $\odot ^{q}T^{\ast }$ be the symmetric part of $\otimes ^{q}T^{\ast }$.
We have $\underset{k=1}{\overset{q}{\odot }}dz^{\alpha _{k}}$ $=\mathcal{S}(%
\underset{k=1}{\overset{q}{\otimes }}dz^{\alpha _{\sigma (k)}}):=\frac{1}{q!}%
\underset{\sigma \in S_{q}}{\sum }(\underset{k=1}{\overset{q}{\otimes }}%
dz^{\alpha _{\sigma (k)}})$. Similar notations apply to $\odot ^{q}T$ of $%
T^{q}$.\newline
\end{notation}

For the next proposition we start with the symmetry-preserving property:

\begin{lemma}
\label{-q-q}


$i)$ For all $q\in \mathbb{N}$, and for $U\in \Omega ^{0,0}(N,L\otimes
(\odot ^{q}T))$ 
\begin{equation}
I^{-q}\partial ^{-q}(U)\in \Omega ^{0,0}(N,L\otimes (\odot ^{(q-1)}T)).%
\hspace*{180pt}
\end{equation}

$ii)$ For all $q\in \mathbb{N}$, and for $U\in \Omega ^{0,0}(N,L\otimes
(\odot ^{q}T^{\ast }))$ 
\begin{equation}
\overline{I}^{q}\overline{\partial }^{q}(U)\in \Omega ^{0,0}(N,L\otimes
(\odot ^{(q-1)}T^{\ast })).\hspace*{180pt}
\end{equation}%
%
%
%
%
%
%
%
%
%
%
%
%
%
%
%
%
%
%
%
%
%
%
%
%
%
%
%
%
%
%
%
%
%
%
%
%
%
%
%
%
%
%
%
%
%
%
%
%
%
%
%
%
%
%
%
%
%
%
%
%
%
%
%
%
%
%
%
%
%
%
%
%
%
%
%
%
%
%
%
%
%
%
%
%
%
%
%
%
%
%
%
%
%
%
%
%
%
%
%
%
%
%
%
%
%
%
%
%
%
%
%
%
%
%
%
%
%
%
%
%
%
\end{lemma}

\begin{proof}
To prove $i)$ we see\vspace*{8pt}\newline
\hspace*{60pt}$I^{-q}\partial ^{-q}\big((f^{\alpha _{1}...\alpha _{q}}s)(%
\underset{k=1}{\overset{q}{\otimes }}\frac{\partial }{\partial z^{\alpha
_{k}}})\big)=(f^{\alpha _{1}...\alpha _{q}}s)_{\alpha _{q}}(\underset{k=1}{%
\overset{q-1}{\otimes }}\frac{\partial }{\partial z^{\alpha _{k}}})$\vspace*{%
8pt}\newline
where the right hand side is symmetric with respect to $\alpha _{1},$...$%
,\alpha _{q-1}$. The $ii)$ is similar.
\end{proof}


\noindent \hspace*{12pt}With the inner product $<\ ,\,>$ (as given in (\ref%
{2-16a})) and Lemma \ref{-q-q}, we find the adjoints of $I^{-q}\partial
^{-q} $, $\overline{I}^{-q}\overline{\partial }^{-q}$, $I^{q}\partial ^{q}$
and $\overline{I}^{q}\overline{\partial }^{q}$ (on symmetric tensors) as
follows:

\begin{proposition}
\label{8} Assume that $N$ is compact K\"{a}hler. Let $\mathcal{S}$ be the
symmetrization operator as in Notation \ref{2.6}. We have\vspace*{-8pt}%
\newline


$i)$ For all $q\in \mathbb{N},$ the adjoint of $I^{\,-q}\partial ^{\,-q}:\
\Omega ^{0,0}(N,L\otimes (\odot ^{q}T))$ $\rightarrow \Omega
^{0,0}(N,L\otimes (\odot ^{(q-1)}T))$ is 
\begin{equation*}
{\scriptsize (I}^{-q}{\scriptsize \partial }^{-q}{\scriptsize )}^{\ast }%
{\scriptsize =-\mathcal{S}\overline{I}}^{\,-(q-1)}\overline{\partial }%
^{\,-(q-1)}:\ \Omega ^{0,0}(N,L\otimes (\odot ^{(q-1)}T))\rightarrow \Omega
^{0,0}(N,L\otimes (\odot ^{q}T)).
\end{equation*}

$ii)$ For all $q\in \{0\}\cup \mathbb{N},$ the adjoint of ${\scriptsize 
\mathcal{S}\overline{I}}^{\,-q}{\scriptsize \overline{\partial }^{\,-q}}:\
\Omega ^{0,0}(N,L\otimes (\odot ^{q}T))\rightarrow \Omega ^{0,0}(N,L\otimes
(\odot ^{(q+1)}T))$ is 
\begin{equation*}
{\scriptsize (\mathcal{S}\overline{I}}^{\,-q}{\scriptsize \overline{\partial 
}^{\,-q})^{\ast }}{\scriptsize =-I}^{-(q+1)}\partial ^{-(q+1)}:\ \Omega
^{0,0}(N,L\otimes (\odot ^{(q+1)}T))\rightarrow \Omega ^{0,0}(N,L\otimes
(\odot ^{q}T)).
\end{equation*}

$iii)$ For all $q\in \mathbb{N},$ the adjoint of $\overline{I}^{\,q}%
\overline{\partial }^{\,q}:\ \Omega ^{0,0}(N,L\otimes (\odot ^{q}T^{\ast }))$
$\rightarrow \Omega ^{0,0}(N,L\otimes (\odot ^{(q-1)}T^{\ast }))$ is 
\begin{equation*}
(\overline{I}^{q}{\scriptsize \overline{\partial }}^{q})^{\ast }{\scriptsize %
=-\mathcal{S}I}^{\,(q-1)}\partial ^{\,(q-1)}:\ \Omega ^{0,0}(N,L\otimes
(\odot ^{(q-1)}T^{\ast }))\rightarrow \Omega ^{0,0}(N,L\otimes (\odot
^{q}T^{\ast })).
\end{equation*}

$iv)$ For all $q\in \{0\}\cup \mathbb{N},$ the adjoint of ${\scriptsize 
\mathcal{S}I}^{\,q}{\scriptsize \partial ^{\,q}}:\ \Omega ^{0,0}(N,L\otimes
(\odot ^{q}T^{\ast }))\rightarrow \Omega ^{0,0}(N,L\otimes (\odot
^{(q+1)}T^{\ast }))$ is 
\begin{equation*}
{\scriptsize (\mathcal{S}I}^{\,q}{\scriptsize \partial ^{\,q})^{\ast }}%
{\scriptsize =-\overline{I}}^{(q+1)}\bar{\partial}^{(q+1)}:\ \Omega
^{0,0}(N,L\otimes (\odot ^{(q+1)}T^{\ast }))\rightarrow \Omega
^{0,0}(N,L\otimes (\odot ^{q}T^{\ast })).
\end{equation*}%
%
%
%
%
%
%
%
%
%
%
%
%
%
%
%
%
%
%
%
%
%
%
%
%
%
%
%
%
%
%
%
%
%
%
%
%
%
%
%
%
%
%
%
%
%
%
%
%
%
%
%
%
%
%
%
%
%
%
%
%
%
%
%
%
%
%
%
%
%
%
%
%
%
%
%
%
%
%
%
%
%
%
%
%
%
%
%
%
%
%
%
%
%
%
%
%
%
%
%
%
%
%
%
%
%
%
%
%
%
%
%
%
%
%
%
%
%
%
%
%
%
%
%
%
%
%
%
%
%
%
%
%
%
%
%
%
%
%
%
%
%
\end{proposition}

\begin{proof}
We content ourselves with proving $i)$ since $ii)$ is equivalent to $i)$,
and $iii)$, $iv)$ are similar to $i)$, $ii)$. By Lemma \ref{-q-q} the images
of $I^{-q}\partial ^{-q}$ on symmetric tensors are symmetric. Let 
\begin{equation*}
\mbox{$U = f^{A} s(\underset{k=1}{\overset{q-1}{\otimes }}\frac{\partial
}{\partial z^{\alpha _{k}}}),\hspace*{20pt} V=h^{B} s(
\overset{q}{\underset{l=1}{\otimes}}  \frac{\partial }{\partial z^{\beta
_{l}}})$}\hspace*{140pt}
\end{equation*}%
\noindent where $A=\alpha _{1}...\alpha _{q-1}$, $B=\beta _{1}...\beta _{q}$%
. To prove $i)$ we need to check that $<-\mathcal{S}\overline{I}^{\,-(q-1)}\,%
\overline{\partial }^{-(q-1)}\,U,\ V>=<U,\ I^{-q}\,\partial ^{-q}\,V>$ holds
if $U,$ $V$ are symmetric. We have, with $\nabla g_{\alpha \overline{\beta }%
}\equiv 0$

\begin{eqnarray*}
LHS &=&< -\overline{I}^{\,-(q-1)}(f^{A}s)_{\overline{j}}\otimes d\overline{z}%
^{\overline{j}}(\otimes _{k=1}^{q-1}\frac{\partial }{\partial z^{\alpha _{k}}%
}),\text{ }h^{B}s(\otimes _{l=1}^{q}\frac{\partial }{\partial z^{\beta _{l}}}%
) > \\
&=&-\int_{N}((f^{A}s)_{\overline{j}},\text{ }h^{B}s)g^{i\overline{j}}g_{i%
\overline{\beta }_{1}}g_{\alpha _{1}\overline{\beta }_{2}}\cdot \cdot \cdot 
\text{ }g_{\alpha _{q-1}\overline{\beta }_{q}}\text{dVol} \\
RHS &=&< f^{A}s(\otimes _{k=1}^{q-1}\frac{\partial }{\partial z^{\alpha _{k}}%
}),\text{ (}h^{B}s)_{\beta _{q}}(\otimes _{l=1}^{q-1}\frac{\partial }{%
\partial z^{\beta _{l}}}) > =-\int_{N}(f^{A}s,\text{(}h^{B}s)_{\beta
_{q}})g_{\alpha _{1}\overline{\beta }_{1}}\cdot \cdot \cdot \text{ }%
g_{\alpha _{q-1}\overline{\beta }_{q-1}}\text{dVol} \\
&=&-\int_{N}((f^{A}s)_{\overline{\beta }_{q}},h^{B}s)g_{\alpha _{1}\overline{%
\beta }_{1}}\cdot \cdot \cdot \text{ }g_{\alpha _{q-1}\overline{\beta }%
_{q-1}}\text{dVol \ \ (integration by parts)}
\end{eqnarray*}

\noindent which proves $i)$ by using the symmetric tensor condition.
\end{proof}

We have the following on $M;$ for similar formulas on $\mathbb{P}^{n}$ see
Propositions \ref{Prop 4.4} and \ref{Prop 4.5} combined.

\begin{proposition}
\label{9} We compute $\Delta ^{q}-\Delta _{q}$ for $q\in \mathbb{Z}$:


$i)$ For all $q\in \mathbb{N}$,%
\mbox{ on $\Omega ^{0,0}\left(M, L\otimes
\left( \odot^{q} T^{\ast}\right) \right) $} 
\begin{equation*}
\Delta ^{q}-\Delta _{q}=n(I^{(q-1)}\partial ^{(q-1)}\overline{I}^{q}%
\overline{\partial }^{q}-\overline{I}^{(q+1)}\overline{\partial }%
^{(q+1)}I^{q}\partial ^{q})=nB\cdot Id.\hspace*{180pt}
\end{equation*}%
%
%
%
%
%
%
%
%
%
%
%
%
%

$ii)$ $\Delta ^{0}-\Delta _{0}=I^{-1}\partial ^{-1}\overline{I}^{0}\overline{%
\partial }^{0}-\overline{I}^{1}\overline{\partial }^{1}I^{0}\partial
^{0}=nB\cdot Id\hspace*{4pt}%
\mbox{on\hspace*{4pt} $\Omega ^{0,0}\left( M, L
\right) $}.$

$iii)$ For all $q\in \mathbb{N}$, 
\mbox{on $\Omega
^{0,0}\left(M, L\otimes \left( \odot^{q} T\right) \right)$} 
\begin{equation*}
\Delta ^{-q}-\Delta _{-q}=n(I^{-(q+1)}\partial ^{-(q+1)}\overline{I}^{\,-q}\,%
\overline{\partial }^{\,-q}-\overline{I}^{\,-(q-1)}\,\overline{\partial }%
^{\,-(q-1)}I^{-q}\partial ^{-q})=nB\cdot Id.\hspace*{100pt}
\end{equation*}%
\noindent These operators separately may not preserve the symmetric parts;
only their differences do.
\end{proposition}


\begin{proof}
\noindent We will only prove $i)$ here. $ii)$ and $iii)$ will be proved in
the Appendix. The difference between $\Delta ^{q}$ and $\Delta _{q}$ equals $%
nB\cdot Id$ via Lemma \ref{Lemma 2.2} $i),$ $ii)$ and (\ref{4-7}) with
vanishing curvature. %
To prove the first equality of $i),$ we have (the following formulas (\ref%
{2-36a}), (\ref{2-36b}) and (\ref{2-36c}) hold true for a K\"{a}hler
manifold $N$)%
\begin{eqnarray}
{\normalsize I}^{q}{\normalsize \partial }^{q}{\normalsize (f_{\alpha
_{1}..\alpha _{q}}s}\underset{k=1}{\overset{q}{\otimes }}{\normalsize dz}%
^{\alpha _{k}}{\normalsize )} &=&{\normalsize I}^{q}[(({\normalsize %
f_{\alpha _{1}...\alpha _{q},i}s+f_{\alpha _{1}...\alpha _{q}}s}_{i})dz^{i})%
\underset{k=1}{\overset{q}{\otimes }}{\normalsize dz}^{\alpha _{k}}]
\label{2-36a} \\
&=&({\normalsize f_{\alpha _{1}...\alpha _{q},i}s+f_{\alpha _{1}...\alpha
_{q}}s}_{i})[dz^{i}\underset{k=1}{\overset{q}{\otimes }}{\normalsize dz}%
^{\alpha _{k}}].  \notag
\end{eqnarray}

\noindent This, together with the formula for ${\normalsize \overline{I}}%
^{\,q}\overline{\partial }^{\,q}$ used in (\ref{2-3-1}), gives%
\begin{equation}
{\normalsize I}^{q-1}{\normalsize \partial }^{q-1}{\normalsize \overline{I}}%
^{\,q}\overline{\partial }^{\,q}{\normalsize (f_{\alpha _{1}...\alpha _{q}}s}%
\underset{k=1}{\overset{q}{\otimes }}{\normalsize dz}^{\alpha _{k}}%
{\normalsize )=}\mbox{$\underset{i,j}{\sum }$}{\normalsize (f}_{\alpha
_{1}...\alpha _{q},\overline{j}i}{\normalsize s+f}_{\alpha _{1}...\alpha
_{q},\overline{j}}{\normalsize s}_{i}{\normalsize )}\text{ }{\normalsize g}^{%
\overline{j}\alpha _{q}}{\normalsize \,(dz}^{i}\otimes {\normalsize (}%
\underset{k=1}{\overset{q-1}{\otimes }}{\normalsize dz}^{\alpha _{k}}%
{\normalsize ))}  \label{2-36b}
\end{equation}

\noindent and 
\begin{equation}
{\normalsize \overline{I}}^{\,(q+1)}\overline{\partial }^{(q+1)}{\normalsize %
I}^{q}{\normalsize \partial }^{q}(f_{\alpha _{1}...\alpha _{q}}s\underset{k=1%
}{\overset{q}{\otimes }}{\normalsize dz}^{\alpha _{k}}{\normalsize )=}%
\mbox{$\underset{i,j}{\sum }$}{\normalsize (f}_{\alpha _{1}...\alpha _{q},i%
\overline{j}}{\normalsize s}+f_{\alpha _{1}...\alpha _{q},\overline{j}}%
{\normalsize s}_{i}{\normalsize +f_{\alpha _{1}...\alpha _{q}}s}_{i\overline{%
j}}{\normalsize )}\text{ }{\normalsize g}^{\overline{j}\alpha _{q}}%
{\normalsize \,(dz}^{i}\otimes {\normalsize (}\underset{k=1}{\overset{q-1}{%
\otimes }}{\normalsize dz}^{\alpha _{k}}{\normalsize )).}  \label{2-36c}
\end{equation}

\noindent Taking (\ref{2-36b}) -- (\ref{2-36c}) and using (\ref{4-7}) we have%
\begin{eqnarray}
&&\hspace*{10pt}{\normalsize (I}^{q-1}{\normalsize \partial }^{q-1}%
{\normalsize \overline{I}}^{\,q}\overline{\partial }^{\,q}{\normalsize -%
\overline{I}}^{\,(q+1)}\overline{\partial }^{(q+1)}{\normalsize I}^{q}%
{\normalsize \partial }^{q}{\normalsize )(f_{\alpha _{1}..\alpha _{q}}s}%
\underset{k=1}{\overset{q}{\otimes }}{\normalsize dz}^{\alpha _{k}}%
{\normalsize )}\hspace*{160pt}  \label{2-36} \\
&{\normalsize =}&{\normalsize -}\mbox{$\underset{i,j}{\sum }$}{\normalsize %
(f_{\alpha _{1}...\alpha _{q}}\otimes s}_{i\overline{j}}{\normalsize )\,g}^{%
\overline{j}\alpha _{q}}{\normalsize \,(dz}^{i}\otimes {\normalsize (}%
\underset{k=1}{\overset{q-1}{\otimes }}{\normalsize dz}^{\alpha _{k}}%
{\normalsize ))}  \notag \\
&\overset{Lem.\ref{Lemma 2.1}}{=}&Bf_{\alpha _{1}...\alpha _{q}}s(dz^{\alpha
_{q}}\otimes (\underset{k=1}{\overset{q-1}{\otimes }}{dz}^{\alpha _{k}}))%
\text{ on }M.  \notag
\end{eqnarray}%
\noindent %
%
%
%
%
%
%
%
%
%
%
%
%
%
%
%
%
%
%
%
%
%
%
%
%
%
%
%
%
%
%
%
%
%
%
%
%
%
%
%
%
%
%
%
%
%
%
%
%
%
%
%
%
%
%
%
%
Using the symmetric tensors $\Omega ^{0,0}(M,L\otimes (\odot ^{q}T^{\ast }))$
for the last line, one sees $i).$
\end{proof}

~~\newline
\noindent The following lemmas will be needed later :

\begin{lemma}
\label{Lemma1} \noindent On $\Omega ^{0,0}\left( N,L\right) $, \hspace*{4pt}$%
I^{-1}\partial ^{-1}\left( I^{-1}\partial ^{-1}\right) ^{\ast }=\Delta _{0}$%
, and $\overline{I}^{1}\overline{\partial }^{1}(\overline{I}^{1}\overline{%
\partial }^{1})^{\ast }=\Delta ^{0}$.
\end{lemma}

\begin{proof}
\noindent We will only prove the first equality. Choosing c.g.c. at $p$ and
noting $\nabla g_{i\overline{j}}$ $=$ $0$ we have, at $p$ (via Proposition $%
\ref{8}\,i)$ for $q=1$) 
\begin{align*}
I^{-1}\partial ^{-1}\left( I^{-1}\partial ^{-1}\right) ^{\ast }(fs)&
=-I^{-1}\partial ^{-1}\mathcal{S}\overline{I}^{\,-0}\overline{\partial }%
^{\,-0}(fs)=-I^{-1}\partial ^{-1}\mathcal{S}\big(%
\mbox{$\underset{i,j}{\sum
}$}f_{,\,\overline{j}}\,s\,g^{i\overline{j}}\otimes 
\mbox{$(
\frac{\partial }{\partial z_{i}}) $}\big) \\
& =-\mbox{$\underset{i}{\sum }$}(f_{,\overline{i}i}s+f_{,\overline{i}}s_{i})
\end{align*}%
\noindent while by $(\ref{partial q star})$ and $(\ref{Delta q})$, at $p$%
\begin{align*}
\Delta _{0}(fs)& =-\ast \partial ^{0}\ast \overline{\partial }^{0}(fs) \\
& =-\ast \big(\mbox{$\underset{j}{\sum }$}(f_{,\,\overline{j}j}s+f_{,\,%
\overline{j}}s_{j})\ (i^{n})\det (g)\,g^{j\overline{j}}(\underset{\alpha }{%
\mbox{$\bigwedge$}}dz^{\alpha }\wedge d\overline{z}^{\overline{\alpha }})%
\big)=-\mbox{$\underset{i}{\sum }$}(f_{,\overline{i}i}s+f_{,\overline{i}%
}s_{i}).\hspace*{40pt}
\end{align*}
\end{proof}

We remark that the annihilation and creation operators (\ref{ACs}) map into 
\textit{symmetric tensors (}for each $k);$ the image is $(-\mathcal{S}%
\overline{I}^{-(k-1)}\overline{\partial }^{-(k-1)})...(-\mathcal{S}\overline{%
I}^{0}\overline{\partial }^{0})s^{0}$ by Proposition \ref{8} $i).$ A
surprising point below is that the preceding insertion of $\mathcal{S}$ for
the composite turns out to be superfluous (not to be confused with the need
of the individually inserted $\mathcal{S}$ in Proposition \ref{8}).

\begin{lemma}
\label{lemmaA} For $k\in \mathbb{N}$ let $0\neq \eta ^{0}\in \Omega
^{0,0}\left( N,L\right) .$\newline

$i)$ Define 
\begin{equation*}
\eta ^{-k}:=(-1)^{k}\big(\overline{I}^{-(k-1)}\overline{\partial }^{-(k-1)}%
\big)\big(\overline{I}^{-(k-2)}\overline{\partial }^{-(k-2)}\big)...\big(%
\overline{I}^{0}\overline{\partial }^{0}\big)\eta ^{0}.\hspace*{100pt}%
\vspace*{-8pt}
\end{equation*}%
\begin{align*}
\mbox{Then}\hspace*{8pt}\eta ^{-k}& =\big(-\mathcal{S}\overline{I}^{-(k-1)}%
\overline{\partial }^{-(k-1)}\big)...\big(-\mathcal{S}\overline{I}^{0}%
\overline{\partial }^{0}\big)\eta ^{0} \\
& =\left( I^{-k}\partial ^{-k}\right) ^{\ast }\left( I^{-(k-1)}\partial
^{-(k-1)}\right) ^{\ast }...\left( I^{-1}\partial ^{-1}\right) ^{\ast }\eta
^{0}\in \Omega ^{0,0}(M,L\otimes (\odot ^{k}T)).\hspace*{100pt}
\end{align*}

$ii)$ Similarly, define 
\begin{equation*}
\eta ^{k}:=(-1)^{k}(I^{k-1}\partial ^{k-1})(I^{k-2}\partial
^{k-2})...(I^{0}\partial ^{0})\eta ^{0}.\hspace*{100pt}\vspace*{-8pt}
\end{equation*}%
\begin{eqnarray*}
\mbox{Then}\hspace*{8pt}\eta ^{k} &=&(-\mathcal{S}I^{k-1}\partial
^{k-1})...(-\mathcal{S}I^{0}\partial ^{0})\eta ^{0} \\
&=&(\overline{I}^{k}\overline{\partial }^{k})^{\ast }(\overline{I}^{k-1}%
\overline{\partial }^{k-1})^{\ast }...(\overline{I}^{1}\overline{\partial }%
^{1})^{\ast }\eta ^{0}\text{ \ \ \ \ \ \ \ \ \ \ \ \ \ }\in \Omega
^{0,0}(M,L\otimes (\odot ^{k}T^{\ast })).\hspace*{100pt}
\end{eqnarray*}%
%
%
%
%
%
%
%
%
%
%
%
%
%
%
%
%
%
%
%
%
%
%
%
%
%
%
%
%
%
%
%
%
%
%
%
%
%
%
%
%
%
%
%
%
%
%
%
%
%
%
%
%
%
%
%
%
%
%
%
%
%
%
%
%
%
%
%
%
%
%
%
%
%
%
%
%
%
%
%
%
%
%
%
%
%
%
%
%
%
%
%
%
%
%
%
%
%
%
%
%
In particular, $\eta ^{-k}$ and $\eta ^{k}$ are symmetric tensors.%
%
%
%
%
%
%
%
%
%
\end{lemma}

\begin{proof}
Let $\eta ^{0}=fs$ where $s$ is a (local) holomorphic section of $L$. Choose
c.g.c. at $p.$ For $k=1$

\begin{equation*}
-\overline{I}^{-0}\overline{\partial }^{-0}(fs)=-%
\mbox{$\underset{i,j}{\sum
}$}(fs)_{,\,\overline{j}}\,g^{i\overline{j}}(\frac{\partial }{\partial z^{i}}%
)\in \Omega ^{0,0}(M,L\otimes T)
\end{equation*}%
\newline
\noindent so $\eta ^{-1}=-\overline{I}^{-0}\overline{\partial }^{-0}\eta
^{0}=-\mathcal{S}\overline{I}^{-0}\overline{\partial }^{-0}\eta ^{0}\overset{%
Prop.\ref{8}\,i)}{=}(I^{-1}\partial ^{-1})^{\ast }\eta ^{0}\in \Omega
^{0,0}(M,L\otimes T).$

For $k=2$, since $-\eta ^{-1}=\overline{I}^{-0}\overline{\partial }%
^{-0}(fs)\in \Omega ^{0,0}(M,L\otimes T)$ (and $\nabla g_{i\overline{j}}$ $%
\equiv $ $0)$, 
\begin{align*}
\overline{I}^{\,-1}\overline{\partial }^{\,-1}\overline{I}^{\,-0}\overline{%
\partial }^{\,-0}(fs)& =\overline{I}^{\,-1}\overline{\partial }^{\,-1}%
\mathcal{S}\overline{I}^{\,-0}\overline{\partial }^{\,-0}(fs) \\
& =\mbox{$\underset{i,j,k,l}{\sum }$}(fs)_{,\,\overline{j}\,\overline{l}%
}\,g^{i\overline{j}}g^{k\overline{l}}\otimes (\frac{\partial }{\partial z^{k}%
}\otimes \frac{\partial }{\partial z^{i}})\overset{\text{at }p}{=}%
\mbox{$\underset{i,j}{\sum
}$}(fs)_{,\,\overline{j}\,\overline{i}}\otimes (\frac{\partial }{\partial
z^{i}}\otimes \frac{\partial }{\partial z^{j}})
\end{align*}%
\noindent is symmetric by $(fs)_{,\,\overline{j}\,\overline{i}}=(f)_{,\,%
\overline{j}\,\overline{i}}\ s=(f)_{,\,\overline{i}\,\overline{j}}\
s=(fs)_{,\,\overline{i}\,\overline{j}}$ (\ref{4-7}). One has 
\begin{align}
\eta ^{-2}=\overline{I}^{\,-1}\overline{\partial }^{\,-1}\mathcal{S}%
\overline{I}^{\,-0}\overline{\partial }^{\,-0}\eta ^{0}& =\mathcal{S}%
\overline{I}^{\,-1}\overline{\partial }^{\,-1}\mathcal{S}\overline{I}^{-0}%
\overline{\partial }^{-0}\eta ^{0}  \label{eta-2} \\
& \hspace*{-16pt}\overset{Prop.\ref{8}\,i)}{=}(I^{-2}\partial ^{-2})^{\ast
}(I^{-1}\partial ^{-1})^{\ast }\eta ^{0}\in \Omega ^{0,0}(M,L\otimes (\odot
^{2}T)).\hspace*{60pt}  \notag
\end{align}
\noindent Inductively, one can obtain $i)$ of the lemma for any $k\in 
\mathbb{N}$.%

The proof for $ii)$ of the lemma is similar. Replacing $\frac{\partial }{%
\partial z^{i}}$ by $dz^{i}$ and Proposition \ref{8} $i)$ by Proposition \ref%
{8} $iii),$ one sees that $\eta ^{1}=(\overline{I}^{1}\overline{\partial }%
^{1})^{\ast }\eta ^{0}$ $\in $ $\Omega ^{0,0}(M,L\otimes T^{\ast }).$ For $%
k=2,$ $\eta ^{2}:=(I^{1}\partial ^{1})(I^{0}\partial ^{0})\eta ^{0}$ is seen
to be%
\begin{equation*}
(fs)_{,jk}dz^{k}\otimes dz^{j}
\end{equation*}

\noindent and is symmetric by (\ref{4-7}). Parallel to the above (\ref{eta-2}%
) with Proposition \ref{8} $iii)$ replacing Proposition \ref{8} $i),$ one
reaches%
\begin{equation*}
\eta ^{2}=(I^{1}\partial ^{1})\mathcal{S}(I^{0}\partial ^{0})\eta ^{0}=(%
\overline{I}^{2}\overline{\partial }^{2})^{\ast }(\overline{I}^{1}\overline{%
\partial }^{1})^{\ast }\eta ^{0}\in \Omega ^{0,0}(M,L\otimes (\odot
^{2}T^{\ast })).
\end{equation*}

\noindent One similarly obtains $ii)$ of the lemma for any $k\in \mathbb{N}.$
\end{proof}

The following is another result on symmetric tensors, which will be needed
in the proof of Theorem \ref{TeoA} $iii).$

\begin{lemma}
\label{L-2-13} Let $N$ be an Abelian variety $M$ or $\mathbb{P}^{n}.$ Let $%
t^{-q}$ $\in $ $H^{0}(N,L\otimes (\odot ^{q}T))$ and $t^{-k}:=\left(
I^{-k-1}\partial ^{-k-1}\right) $ ... $\left( I^{-q}\partial ^{-q}\right)
t^{-q}$ ($0\leq k<q)$. Then it holds that for $k=q-1,$ $q-2,$ $\cdot \cdot
\cdot ,\ 0,$ both $i)$ $t^{-k}$ and $ii)$ $\overline{I}^{\,-k}\overline{%
\partial }^{-k}t^{-k}$ are symmetric tensors.
\end{lemma}

\begin{proof}
For $i),$ that $t^{-k}$ is symmetric follows from Lemma \ref{-q-q} $i)$ and
the assumption that $t^{-q}$ is symmetic. For $ii),$ write (for the case $N$
is $M)$%
\begin{eqnarray}
\overline{I}^{\,-k}\overline{\partial }^{-k}t^{-k} &=&\overline{I}^{\,-k}%
\overline{\partial }^{-k}I^{-(k+1)}\partial ^{-(k+1)}t^{-(k+1)}  \label{3-s}
\\
&\overset{Prop.\ref{9}\text{ }iii)}{=}&(I^{-(k+2)}\partial ^{-(k+2)}%
\overline{I}^{\,-(k+1)}\overline{\partial }^{-(k+1)}-B)t^{-(k+1)}.  \notag
\end{eqnarray}

\noindent For $k=q-1,$ $\overline{I}^{\,-(q-1)}\overline{\partial }%
^{-(q-1)}t^{-(q-1)}$ equals $-Bt^{-q}$ by assumption $\overline{\partial }%
^{-q}t^{-q}=0$; it is symmetric. 
For $k=q-2,$ by (\ref{3-s}) we have 
\begin{equation}
\overline{I}^{\,-(q-2)}\overline{\partial }^{-(q-2)}t^{-(q-2)}=I^{-q}%
\partial ^{-q}(\overline{I}^{\,-(q-1)}\overline{\partial }%
^{-(q-1)}t^{-(q-1)})-Bt^{-(q-1)}.  \label{3-s-1}
\end{equation}%
\noindent That $i)$ and $ii)$ of the lemma hold for $k=q-1$ means that $%
t^{-(q-1)}$ and $\overline{I}^{\,-(q-1)}\overline{\partial }%
^{-(q-1)}t^{-(q-1)}$ are symmetric. Then, by Lemma \ref{-q-q} $i),$ $%
I^{-q}\partial ^{-q}(\overline{I}^{\,-(q-1)}\overline{\partial }%
^{-(q-1)}t^{-(q-1)})$ is symmetric. Altogether, the LHS of (\ref{3-s-1}) is
symmetric. By repeating this process one proves $ii)$ for $k$ $=$ $q-3,$ $%
\cdot \cdot \cdot ,$ \noindent $0$. 
For $\mathbb{P}^{n}$ the above reasoning still works as long as the constant
\textquotedblleft $B$" in (\ref{3-s}) (from Proposition \ref{9} $iii))$ is
replaced by another constant via Proposition \ref{Prop 4.4} $iii).$
\end{proof}

\section{\textbf{Proofs of Theorem \protect\ref{TeoA}, Corollary \protect\ref%
{TeoB} and Theorem \protect\ref{T-main}\label{Sec3}}}


\noindent \hspace*{12pt} In this section, we assume that the K\"{a}hler
manifold $N$ is an Abelian variety $M$.

\begin{proof}
(of \textbf{Theorem \ref{TeoA} $ii);$ }see below for proofs of Theorem \ref%
{TeoA}\textbf{\ }$iii)$ and $i)$)\hspace*{8pt} For $k=1$, by Lemma\ \ref%
{Lemma1} 
\begin{equation*}
\left( I^{-1}\partial ^{-1}\right) s^{-1}=\left( I^{-1}\partial ^{-1}\right)
\left( I^{-1}\partial ^{-1}\right) ^{\ast }s^{0}=\Delta _{0}s^{0}\overset{%
\text{assumption}}{=}qBs^{0}\neq 0.\hspace*{100pt}
\end{equation*}%
\noindent So $s^{-1}\neq 0$. For $q>1$ and $k=2$, $\left( I^{-2}\partial
^{-2}\right) s^{-2}\in \Omega ^{0,0}(M,L\otimes T)$, 
\begin{align}
\left( I^{-2}\partial ^{-2}\right) s^{-2}& =\left( I^{-2}\partial
^{-2}\right) \left( I^{-2}\partial ^{-2}\right) ^{\ast }\left(
I^{-1}\partial ^{-1}\right) ^{\ast }s^{0}\overset{Prop.\ref{8}\text{ }i)}{%
\underset{Lem.\ref{lemmaA}\text{ }i)}{=}}-\big(I^{-2}\partial ^{-2}\big)\big(%
\overline{I}^{\,-1}\overline{\partial }^{\,-1}\big)s^{-1}  \label{3-a} \\
& \hspace*{-20pt}\overset{Prop.\ref{9}\,iii)}{=}-\big(\big(\overline{I}%
^{\,-0}\overline{\partial }^{\,-0}\big)\big(I^{-1}\partial ^{-1}\big)+B\big)%
s^{-1}\in \Omega ^{0,0}(M,L\otimes T)  \notag
\end{align}%
\noindent which is automatically symmetric in $T$ and thus equals

\begin{align}
-\big(\mathcal{S}\overline{I}^{\,-0}\overline{\partial }^{\,-0}\big)\big(%
I^{-1}\partial ^{-1}\big)s^{-1}& -Bs^{-1}\overset{Prop.\ref{8}\,i)}{=}\big(%
I^{\,-1}\partial ^{\,-1}\big)^{\ast }\big(I^{-1}\partial ^{-1}\big)%
s^{-1}-Bs^{-1}\hspace*{64pt}  \label{3-b} \\
& \hspace*{-60pt}\overset{k=1}{=}\big(I^{\,-1}\partial ^{\,-1}\big)^{\ast
}(qBs^{0})-Bs^{-1}=(q-1)Bs^{-1}\neq 0.  \notag
\end{align}%
\noindent Hence $s^{-2}\neq 0$.\ We are going to repeat this process. For
higher $k$ it is less straightforward: First, $s^{-k}$ equals $(-1)^{k}(%
\overline{I}^{\,-(k-1)}\overline{\partial }^{\,-(k-1)})...(\overline{I}%
^{\,-0}\overline{\partial }^{\,-0})\,s^{0}$ and is symmetric by Lemma \ref%
{lemmaA} $i)$. So ($I^{-k}\partial ^{-k})s^{-k}$ is symmetric by Lemma \ref%
{-q-q} $i).$ This and $Bs^{-(k-1)}$ being symmetric enable one to proceed to
the analogue of (\ref{3-b}) for $k.$ One arrives at 
\begin{equation}
\left( I^{-k}\partial ^{-k}\right) s^{-k}=(q+1-k)Bs^{-(k-1)}  \label{3-0}
\end{equation}%
\noindent $\neq 0\ (1\leq k\leq q),$\ proving that $s^{-k}\neq 0$. In
particular $s^{-q}\neq 0$.\newline
\noindent \hspace*{12pt} For the holomorphicity of $s^{-q}$, with $k=q+1$ in
(\ref{BBB}) and Lemma \ref{-q-q} $i)$%
\begin{equation}
\Delta _{-\left( k-1\right) }I^{-k}\partial ^{-k}-I^{-k}\partial ^{-k}\Delta
_{-k}=BI^{-k}\partial ^{-k}:\Omega ^{0,0}(M,L\otimes (\odot
^{k}T))\rightarrow \Omega ^{0,0}(M,L\otimes (\odot ^{k-1}T)).  \label{3-1}
\end{equation}%
\noindent Notice that $\Delta _{-k}$ preserves $\Omega ^{0,0}(M,L\otimes
(\odot ^{k}T))$ by Lemma $\ref{Lemma 2.2}\,iv)$. We take the adjoints of (%
\ref{3-1}) to get

\begin{equation}
\left( I^{-k}\partial ^{-k}\right) ^{\ast }\Delta _{-(k-1)}-\Delta
_{-k}\left( I^{-k}\partial ^{-k}\right) ^{\ast }=B\left( I^{-k}\partial
^{-k}\right) ^{\ast }\hspace*{120pt}  \label{STAR}
\end{equation}%
mapping $\Omega ^{0,0}(M,L\otimes (\odot ^{k-1}T))$ to $\Omega
^{0,0}(M,L\otimes (\odot ^{k}T))$ and enabling us to go through the
following recursive process.\newline
\noindent \hspace*{12pt} By using $(\ref{STAR})$ recursively, one has%
\begin{align}
\Delta _{-q}\text{ }s^{-q}& {\hspace*{8pt}=\hspace*{8pt}}\Delta _{-q}\left(
I^{-q}\partial ^{-q}\right) ^{\ast }s^{-\left( q-1\right) }  \notag
\label{holomor} \\
& \overset{(\ref{STAR})}{=}\left( I^{-q}\partial ^{-q}\right) ^{\ast }\Delta
_{-(q-1)}\text{ }s^{-\left( q-1\right) }-B\left( I^{-q}\partial ^{-q}\right)
^{\ast }s^{-\left( q-1\right) }=\left( I^{-q}\partial ^{-q}\right) ^{\ast
}\Delta _{-(q-1)}s^{-\left( q-1\right) }-B\text{ }s^{-q}  \notag \\
& {\hspace*{8pt}=\hspace*{8pt}}\left( I^{-q}\partial ^{-q}\right) ^{\ast
}\Delta _{-(q-1)}\left( I^{-(q-1)}\partial ^{-(q-1)}\right) ^{\ast
}s^{-\left( q-2\right) }-Bs^{-q} \\
& \overset{(\ref{STAR})}{=}\left( I^{-q}\partial ^{-q}\right) ^{\ast }\left(
I^{-(q-1)}\partial ^{-(q-1)}\right) ^{\ast }\Delta _{-(q-2)}s^{-\left(
q-2\right) }-Bs^{-q}-Bs^{-q}  \notag \\
& \hspace*{10pt}\vdots  \notag \\
& \overset{(\ref{STAR})}{=}\left( I^{-q}\partial ^{-q}\right) ^{\ast }\left(
I^{-\left( q-1\right) }\partial ^{-\left( q-1\right) }\right) ^{\ast
}...\left( I^{-1}\partial ^{-1}\right) ^{\ast }\Delta _{0}\text{ }%
s^{0}-qBs^{-q}  \notag \\
& {\hspace*{8pt}=\hspace*{8pt}}qBs^{-q}-qBs^{-q}=0  \notag
\end{align}

\noindent giving that $s^{-q}$ is holomorphic since $\Delta _{-q}=(\overline{%
\partial }^{\,-q})^{\ast }\overline{\partial }^{\,-q}$.\newline
\noindent \hspace*{12pt} The proof for $s^{-k}$ to be an eigensection with
eigenvalue $(q-k)B$ is similar to $(\ref{holomor})$. By using $(\ref{STAR})$
recursively, for $k=1,2,...q$ we have%
\begin{eqnarray}
\Delta _{-k}\text{ }s^{-k} &=&\Delta _{-k}\left( I^{-k}\partial ^{-k}\right)
^{\ast }s^{-\left( k-1\right) }  \label{3-4} \\
&=&\big(I^{-k}\partial ^{-k}\big)^{\ast }\Delta _{-(k-1)}\text{ }s^{-\left(
k-1\right) }-B\big(I^{-k}\partial ^{-k}\big)^{\ast }\text{ }s^{-\left(
k-1\right) }  \notag \\
&=&\big(I^{-k}\partial ^{-k}\big)^{\ast }...\big(I^{-1}\partial ^{-1}\big)%
^{\ast }\Delta _{0}\text{ }s^{0}-kBs^{-k}=(q-k)Bs^{-k}.\hspace*{40pt}  \notag
\end{eqnarray}
\end{proof}

\noindent \hspace*{12pt} As an application of the reasoning in the proof of
Theorem \ref{TeoA} $ii)$ above$,$ we have the following three results. For
the next theorem, define $s^{k}:=(\overline{I}^{k}\overline{\partial }%
^{k})^{\ast }...(\overline{I}^{1}\overline{\partial }^{1})^{\ast }s^{0}\in
\Omega ^{0,0}(M,L\otimes (\odot ^{k}T^{\ast }))$ for $k\in \mathbb{N},$
where $s^{0}$ $\in \Omega ^{0,0}\left( M,L\right) $ ($(\overline{I}^{k}%
\overline{\partial }^{k})^{\ast }...(\overline{I}^{1}\overline{\partial }%
^{1})^{\ast }$ is a \textquotedblleft creation operator") $.$ If we take the
adjoints of (\ref{-BBB}) with $k=q+1$ then a similar reasoning as (\ref{3-4}%
) yields the following result (see Theorem \ref{T-5-1} for $\mathbb{P}^{n}).$

\begin{theorem}
\label{q+n+k} \noindent Suppose that $0\neq s^{0}$ is an eigensection of $%
\Delta ^{0}$ with eigenvalue $(q+n)B>0$ for some $q\in \{0\}\cup \mathbb{N}$%
. Let $s^{k}$ be as above. Then $s^{k}$ is a nontrivial eigensection of $%
\Delta ^{k}$ with eigenvalue $(q+n+k)B>0$. (For the existence of $s^{0}$ for
every $q\in \{0\}\cup \mathbb{N}$, see Remark \ref{R-3-16} $b).$)
\end{theorem}

\begin{proof}
It remains to show $s^{k}\neq 0.$ Comparing (\ref{3-a}) and (\ref{3-b})
above and considering $(\overline{I}^{k}\overline{\partial }^{k})s^{k}$ in
place of $\left( I^{-k}\partial ^{-k}\right) s^{-k}$ we will be able to
reach $(\overline{I}^{k}\overline{\partial }^{k})s^{k}$ $=$ ($%
q+n+k-1)Bs^{k-1}$ similar to (\ref{3-0}) via the use of Propositions $\ref{8}
$ $iii), $ $\ref{9}\,i)$, Lemma $\ref{lemmaA}$ $ii)$ and Lemma \ref{-q-q} $%
ii)$\ replacing respectively that of Propositions $\ref{8}$ $i),$ $\ref{9}%
\,iii)$, Lemma $\ref{lemmaA}$ $i)$ and Lemma \ref{-q-q} $i)$\ in the
derivation of (\ref{3-0}). So inductively if $s^{k-1}\neq 0$ then $s^{k}\neq
0.$
\end{proof}

\begin{proposition}
\label{alleigen} The set $S_{1}$ of eigenvalues of $\Delta _{0}$ on $\Omega
^{0.0}(M,L)$ is contained in the set $S_{2}$ $=$ $\{qB$ \TEXTsymbol{\vert} $%
q\in \{0\}\cup \mathbb{N}\}.$ (In fact $S_{1}=S_{2}$ by Remark \ref{R-3-16} $%
a)$.)
\end{proposition}

\begin{proof}
\noindent Suppose that $s_{\lambda }\equiv s_{\lambda }^{0}$ is a nonzero
eigensection of $\Delta _{0}$ with the eigenvalue $\lambda \in \left(
qB,\left( q+1\right) B\right) $ for some $q\in \mathbb{N}\cup \{0\}.$ Let $%
s_{\lambda }^{-k}=\left( I^{-k}\partial ^{-k}\right) ^{\ast },\ ...\ ,\left(
I^{-1}\partial ^{-1}\right) ^{\ast }s_{\lambda }$ \ for $k\in \mathbb{N}$.
By a similar reasoning as in deducing (\ref{3-0}) we can reach (with $qB$
replaced by $\lambda $)%
\begin{equation}
\left( I^{-k}\partial ^{-k}\right) s_{\lambda }^{-k}=(\lambda
-(k-1)B)s_{\lambda }^{-(k-1)}  \label{3-2a}
\end{equation}%
\noindent implying inductively $s_{\lambda }^{-k}\neq 0$ for $1\leq k\leq
q+1 $. With $\Delta _{0}s_{\lambda }=\lambda s_{\lambda }$ we follow the
argument in deducing (\ref{3-4}) from (\ref{STAR}) with $s_{\lambda
},\lambda $ here replacing $s^{-0}$ (= $s^{0}$), $qB$ there respectively: 
\begin{equation}
\Delta _{-1}s_{\lambda }^{-1}=(\lambda -B)s_{\lambda }^{-1},\hspace*{10pt}%
\cdots ,\hspace*{10pt}\Delta _{-k}s_{\lambda }^{-k}=(\lambda -kB)s_{\lambda
}^{-k},\hspace*{20pt}k\in \mathbb{N}  \label{3-2b}
\end{equation}

\noindent and when $k=q+1,$ the eigensection $s_{\lambda }^{-k}$ of $\Delta
_{-k}$ is of negative eigenvalue, giving a contradiction since the operator $%
\Delta _{-k}$ is non-negative and $s_{\lambda }^{-(q+1)}\neq 0$ as already
mentioned.
\end{proof}

\begin{corollary}
\label{C-3-3} For $k\in \{0\}\cup \mathbb{N}$ the set $S_{1}^{k}$ of
eigenvalues of $\Delta ^{k}$ on $\Omega ^{0,0}(M,L\otimes (\odot ^{k}T^{\ast
}))$ in Theorem \ref{q+n+k} is contained in the set $S_{2}^{k}$ :$=$ $%
\{(q+n+k)B$ \TEXTsymbol{\vert} $q\in \{0\}\cup \mathbb{N}\}$. (In fact $%
S_{1}^{k}=S_{2}^{k}$ by Remark \ref{R-3-16} $b)$.)
\end{corollary}

\begin{proof}
Let $\lambda $ be an eigenvalue of $\Delta ^{k}$ with eigensection $%
t^{k}\neq 0.$ Define $t^{k-1}=\overline{I}^{k}\overline{\partial }^{k}t^{k},$
$t^{k-2}=(\overline{I}^{k-1}\overline{\partial }^{k-1})(\overline{I}^{k}%
\overline{\partial }^{k})t^{k},$ $\cdot \cdot \cdot $ , $t^{0}=(\overline{I}%
^{1}\overline{\partial }^{1})$ $\cdot \cdot \cdot $ $(\overline{I}^{k}%
\overline{\partial }^{k})t^{k}.$ Applying (\ref{-BBB}) in Proposition \ref%
{10.2PZ} to $t^{k}$ we obtain $\Delta ^{k-1}t^{k-1}$ $=$ $(\lambda
-B)t^{k-1} $ by $\Delta ^{k}t^{k}$ $=$ $\lambda t^{k},$ and repeatedly $%
\Delta ^{k-2}t^{k-2}$ $=$ $(\lambda -2B)t^{k-2}$, $\cdot \cdot \cdot $ , $%
\Delta ^{0}t^{0}$ $=$ $(\lambda -kB)t^{0}.$ By Proposition \ref{alleigen}
and Proposition \ref{9} $ii),$ any eigenvalue of $\Delta ^{0}$ is of the
form $(q+n)B$ for some $q\in \{0\}\cup \mathbb{N}$, which equals $\lambda
-kB.$
\end{proof}

\noindent \hspace*{12pt} A nontrivial holomorphic section in $\Omega
^{0,0}\left( M,L\otimes (\odot ^{q}T)\right) $ can give rise to an
eigensection of $\Delta _{0}$ as claimed by Theorem \ref{TeoA} $iii)$ which
we shall now show:

\begin{proof}
(of \textbf{Theorem \ref{TeoA} $iii)$}) Recall $t^{-q}$ $\in $ $%
H^{0}(M,L\otimes (\odot ^{q}T))$ and $t^{-\left( q-1\right) }$ $:=$ ($%
I^{-q}\partial ^{-q})t^{-q}$ is symmetric by Lemma \ref{-q-q} $i).$ $%
\mathcal{S}$ denotes the symmetrization operator.%
\begin{eqnarray}
\ \ \ \ \left( I^{-q}\partial ^{-q}\right) ^{\ast }t^{-\left( q-1\right) } &%
\overset{Prop.\ref{8}\,i)}{=}&-\big(\mathcal{S}\overline{I}^{\,-(q-1)}%
\overline{\partial }^{\,-(q-1)}\big)t^{-\left( q-1\right) }=-\big(\mathcal{S}%
\overline{I}^{\,-(q-1)}\overline{\partial }^{\,-(q-1)}I^{-q}\partial ^{-q}%
\big)t^{-q}  \label{3-5} \\
&\overset{Prop.\ref{9}\,iii)}{=}&-\mathcal{S}\big(I^{-\left( q+1\right)
}\partial ^{-\left( q+1\right) }\overline{I}^{\,-q}\overline{\partial }%
^{\,-q}-B\big)t^{-q}  \notag
\end{eqnarray}%
\noindent which equals $Bt^{-q}\neq 0$ since ${\scriptsize \overline{%
\partial }^{\,-q}}t^{-q}=0$ by holomorphicity, and in turn $t^{-(q-1)}\neq 0$%
. Similarly, we also obtain%
\begin{equation}
\left( I^{-(q-1)}\partial ^{-(q-1)}\right) ^{\ast }t^{-\left( q-2\right) }=-%
\mathcal{S}(I^{-q}\partial ^{-q}\overline{I}^{\,-(q-1)}\overline{\partial }%
^{\,-(q-1)}-B)t^{-(q-1)}.  \label{3-10a}
\end{equation}

\noindent If%
\begin{equation}
\overline{I}^{\,-(q-1)}\overline{\partial }^{\,-(q-1)}t^{-(q-1)}\text{ is
symmetric,}  \label{3-10b}
\end{equation}

\noindent which is proved in Lemma \ref{L-2-13}, then we write (\ref{3-10b})
as $\mathcal{S(}\overline{I}^{\,-(q-1)}\overline{\partial }%
^{-(q-1)}t^{-(q-1)})$ which by Proposition \ref{8} $i)$ equals $-\left(
I^{-q}\partial ^{-q}\right) ^{\ast }t^{-\left( q-1\right) }.$ Altogether, (%
\ref{3-10a}) becomes%
\begin{eqnarray}
&&\mathcal{S}I^{-q}\partial ^{-q}\left( I^{-q}\partial ^{-q}\right) ^{\ast
}t^{-\left( q-1\right) }+Bt^{-(q-1)}  \label{3-10c} \\
&\overset{(\ref{3-5})}{=}&BI^{-q}\partial
^{-q}t^{-q}+Bt^{-(q-1)}=2Bt^{-(q-1)}.  \notag
\end{eqnarray}%
\noindent Here we drop $\mathcal{S}$ since $I^{-q}\partial ^{-q}t^{-q}$ ($%
=t^{-(q-1)})$ is symmetric by Lemma \ref{-q-q} $i).$ For $t^{-k}:=\left(
I^{-k-1}\partial ^{-k-1}\right) $ ... $\left( I^{-q}\partial ^{-q}\right)
t^{-q}$ ($0\leq k<q)$ the argument similar to that from (\ref{3-10a}) to (%
\ref{3-10c}) yields%
\begin{equation}
\left( I^{-(k+1)}\partial ^{-(k+1)}\right) ^{\ast }t^{-k}=(q-k)Bt^{-(k+1)},
\label{3-8a}
\end{equation}%
\noindent giving $t^{-k}\neq 0$ inductively for $k=q-1,q-2,...,0$. In
particular $t^{0}\neq 0.$\newline
\noindent \hspace*{12pt} To show that $t^{0}$ is an eigensection of $\Delta
_{0}$, we apply Proposition \ref{10.1} recursively,%
\begin{eqnarray}
\Delta _{0}\text{ }t^{0} &=&\Delta _{0}\left( I^{-1}\partial
^{-1}I^{-2}\partial ^{-2}...I^{-q}\partial ^{-q}\right) t^{-q}  \label{3-5z}
\\
&=&I^{-1}\partial ^{-1}\Delta _{-1}I^{-2}\partial ^{-2}...I^{-q}\partial
^{-q}t^{-q}+BI^{-1}\partial ^{-1}I^{-2}\partial ^{-2}...I^{-q}\partial
^{-q}t^{-q}  \notag \\
&=&I^{-1}\partial ^{-1}I^{-2}\partial ^{-2}...I^{-q}\partial ^{-q}\Delta
_{-q}t^{-q}+qBI^{-1}\partial ^{-1}I^{-2}\partial ^{-2}...I^{-q}\partial
^{-q}t^{-q}  \notag \\
&=&I^{-1}\partial ^{-1}I^{-2}\partial ^{-2}...I^{-q}\partial ^{-q}\big((%
\overline{\partial }^{\,-q})^{\ast }\overline{\partial }^{\,-q}\big)%
t^{-q}+qBt^{0}=0+qBt^{0}  \notag
\end{eqnarray}

\noindent proving the desired. The similar recursion proves the case for $%
t^{-k}$.
\end{proof}

\begin{proof}
(of \textbf{Theorem \ref{TeoA} $i)$}$)$ For all $q\in \mathbb{N}$ we are
going to show that $(I^{-q}\partial ^{-q})^{\ast }...(I^{-1}\partial
^{-1})^{\ast }$ ($=A_{q})$ is a bijective operator from $E_{qB}$ $\subset $ $%
\Omega ^{0,0}(M,L)$ to ($E_{0}$ $=)$ $H^{0}(M,L\otimes (\odot ^{q}T))$. Let $%
s^{0}\in \Omega ^{0,0}(M,L)$ be an eigensection of $\Delta _{0}$ with
eigenvalue $qB$ and $s^{-k}:=(I^{-k}\partial ^{-k})^{\ast
}...(I^{-1}\partial ^{-1})^{\ast }s^{0}$ for $k=1,...,q$ as in Theorem \ref%
{TeoA}$.$ From (\ref{3-0}) we have%
\begin{equation*}
(I^{-1}\partial ^{-1})...(I^{-q}\partial ^{-q})s^{-q}=(qB)((q-1)B)\text{ }%
\cdot \cdot \cdot \text{ (}2B)Bs^{0}=q!B^{q}s^{0}
\end{equation*}

\noindent so $(q!B^{q})^{-1}A_{0}^{\#}A_{q}=Id$ on $E_{qB}$ (see (\ref{ACs})
for the notations $A_{q}$ and $A_{0}^{\#})$. For $t^{-q}$ $\in $ $%
H^{0}(M,L\otimes (\odot ^{q}T))$ and $t^{-k}:=A_{k}^{\#}t^{-q}$ $(=$ $%
I^{-(k+1)}\partial ^{-(k+1)}...I^{-q}\partial ^{-q}t^{-q}$), $0\leq k<q,$
the repeated use of (\ref{3-8a}) gives 
%
%
$\left( I^{-q}\partial ^{-q}\right) ^{\ast }...\left( I^{-1}\partial
^{-1)}\right) ^{\ast }t^{0}$ $=$ $q!B^{q}t^{-q}$ so $%
(q!B^{q})^{-1}A_{q}A_{0}^{\#}$ $=$ $Id$ on $H^{0}(M,L\otimes (\odot ^{q}T))$.

\end{proof}

Denoting by $E_{qB}$ the vector space of all eigensections of $\Delta _{0}$
with eigenvalues $qB$, we can calculate the dimension of $E_{qB}$.

\begin{proof}
(of \textbf{Corollary \ref{TeoB}}) By Theorem \ref{TeoA} $i)$ we have $\dim
_{\mathbb{C}}E_{qB}=\dim _{\mathbb{C}}H^{0}(M,L\otimes (\odot ^{q}T))$ for
any $q\in \{0\}\cup \mathbb{N}.$ Since $T$ is holomorphically trivial and $%
\mbox{Rank$(\odot^{q}T)=C^{n+q-1}_{q}$}$, we have 
\begin{equation}
\dim _{\mathbb{C}}E_{qB}=C_{q}^{\,n+q-1}\cdot h^{0}(L)=C_{q}^{\,n+q-1}\delta
_{1}...\delta _{n},
\end{equation}%
\noindent (which includes the case $q=0)$ where $h^{0}(L)=\delta
_{1}...\delta _{n}$ \cite[p.317]{GH84}.
\end{proof}

%
%
%

\begin{remark}
\label{R-3-16} $a)$ By Theorem \ref{TeoA} $i)$ and $\dim _{\mathbb{C}%
}E_{qB}>0$ by Corollary \ref{TeoB} together with Proposition \ref{alleigen}
we obtain that the set of eigenvalues of $\Delta _{0}$ on $\Omega
^{0,0}(M,L) $ is exactly $\{qB$ \TEXTsymbol{\vert} $q\in \{0\}\cup \mathbb{%
N\}}$.\ $b)$ The assumption of Theorem \ref{q+n+k} is met using $a)$ and
Proposition \ref{9} $ii).$ It follows from Theorem \ref{q+n+k}, Corollary %
\ref{C-3-3} and $a) $ that the set of eigenvalues of $\Delta ^{k}$ is
exactly $\{(q+n+k)B$ \TEXTsymbol{\vert} $q\in \{0\}\cup \mathbb{N}\}$.
\end{remark}

\begin{proof}
\textbf{(of Theorem \ref{T-main})} Denote the following direct image (where $%
\pi _{1}$ $:$ $M\times \hat{M}$ $\rightarrow $ $M$ and $\pi _{2}$ $:$ $%
M\times \hat{M}$ $\rightarrow $ $\hat{M}$ are the natural projections) by%
\begin{equation}
\mathbb{\hat{E}}_{qB}:=(\pi _{2})_{\ast }(P\otimes \pi _{1}^{\ast }(L\otimes
\odot ^{q}T))\rightarrow \hat{M}.  \label{EqBh}
\end{equation}%
We are going to equip $\mathbb{E}_{qB}$ with the holomorphic bundle
structure from $\mathbb{\hat{E}}_{qB}.$ The point is that the linear
isomorphism in Theorem \ref{TeoA} $i)$ can be treated in a smoothly varying
way along $\hat{M}$ by using the construction of a global metric on the
family (\cite[Section 5]{CCT24}). For the sake of clarity we give some
details.\ For any $\mu \in V$ ($M=V/\Lambda )$ let $\tau _{\mu }$ $:M$ $%
\rightarrow $ $M$ be the map defined by the translation by $[\mu ]$ $\in $ $%
M.$ Observe that for any $[\hat{\mu}]$ the fibre (see (\ref{EqBh}) for the
notation $\mathbb{\hat{E}}_{qB})$%
\begin{eqnarray*}
\mathbb{\hat{E}}_{qB}|_{[\hat{\mu}]} &=&H^{0}(M,P\otimes \pi _{1}^{\ast
}(L\otimes \odot ^{q}T)|_{M\times \{[\hat{\mu}]\}}) \\
&=&H^{0}(M,P_{[\hat{\mu}]}\otimes L\otimes \odot ^{q}T) \\
&=&H^{0}(M,\tau _{\mu }^{\ast }L\otimes \odot ^{q}T)\text{ as }P_{[\hat{\mu}%
]}=\tau _{\mu }^{\ast }L\otimes L^{\ast }\text{ \cite[(2.16)]{CCT24}.}
\end{eqnarray*}%
\noindent By Theorem \ref{TeoA} $i)$ (with $\tau _{\mu }^{\ast }L$ in place
of $L$) we have a linear isomorphism $A_{q}$ from $s^{0}$ $\in $ $%
E_{qB}(\tau _{\mu }^{\ast }L)$ $=$ $E_{qB}(\tilde{E}|_{M\times \{[\hat{\mu}%
]\}})$ ($\subset $ $\Omega ^{0,0}(M,\tau _{\mu }^{\ast }L))$ to a
holomorphic section $s^{-q}$ $\in $ $H^{0}(M,\tau _{\mu }^{\ast }L\otimes
\odot ^{q}T),$ giving a smoothly varying linear isomorphism between $\mathbb{%
E}_{qB}|_{[\hat{\mu}]}$ and $\mathbb{\hat{E}}_{qB}|_{[\hat{\mu}]}.$ Now by
Grauert's direct image theorem \cite[Corollary 12.9]{Har77} we learn that $%
\mathbb{\hat{E}}_{qB}$ $\rightarrow $ $\hat{M}$ is a holomorphic vector
bundle and hence via the isomorphism we can endow $\mathbb{E}_{qB}$ with the
structure of holomorphic vector bundle using that on $\mathbb{\hat{E}}_{qB}.$
The desired rank is the dimension of $H^{0}(M,\tau _{\mu }^{\ast }L\otimes
\odot ^{q}T)$ (= $\dim _{\mathbb{C}}E_{qB}$ by Theorem \ref{TeoA} $i))$,
which is $C_{q}^{\,n+q-1}\delta _{1}...\delta _{n}$ by Corollary \ref{TeoB}$%
. $

To compute the full curvature of $\mathbb{E}_{qB}$ we first observe that $%
H^{0}(M,\tau _{\mu }^{\ast }L\otimes \odot ^{q}T)$ is isomorphic to $%
H^{0}(M,\tau _{\mu }^{\ast }L)\otimes \mathbb{C}^{rank(\odot ^{q}T)}$ as $T$
is holomorphically trivial. So, by endowing $H^{0}(M,\tau _{\mu }^{\ast
}L\otimes \odot ^{q}T)$ with the metric $h_{q}$ given by the $L^{2}$-metric $%
h$ on $H^{0}(M,\tau _{\mu }^{\ast }L)$ (\cite[Theorem 1.2]{CCT24}) tensoring
the Euclidean metric on $\mathbb{C}^{rank(\odot ^{q}T)}$ we can then reach
the full curvature $\Theta (\mathbb{E}_{qB},h_{q})$ as given in (\ref{FC})
by resorting to \cite[Theorem 1.2]{CCT24}.
\end{proof}

\section{\textbf{Bochner-Kodaira type identities of the third order on $%
\mathbb{P}^{n}\label{Sec4}$}}

\noindent \hspace*{12pt} Using an analogous approach as in Section 3 we are
going to obtain similar results on $\mathbb{P}^{n}$. Let us begin with the
Fubini-Study form/metric $\omega _{FS}$ on $\mathbb{P}^{n}$ ($[1:z_{1}:\cdot
\cdot \cdot :z_{n}]$ denotes the standard (inhomogeneous) coordinates):%
\begin{equation}
{\normalsize g_{\alpha \bar{\beta}}(z)}{\normalsize =}\frac{2}{c}%
{\normalsize \frac{(\delta _{\alpha \beta }(1+\underset{\gamma }{\sum }%
z_{\gamma }\overline{z}_{\gamma })-\overline{z}_{\alpha }z_{\beta })}{(1+%
\underset{\gamma }{\sum }z_{\gamma }\overline{z}_{\gamma })^{2}}}\text{; \ }%
{\normalsize g_{\alpha \bar{\beta}}(0)=}\frac{2}{c}\delta _{\alpha \bar{\beta%
}}  \label{FS}
\end{equation}%
\noindent where $c=2$ but we retain it to keep track of the curvature
constant. We have 
\begin{equation}
{\normalsize \omega _{FS}}{\normalsize =ig_{\alpha \bar{\beta}}(z)dz^{\alpha
}\wedge d\overline{z}^{\overline{\beta }};}\text{ }\ \omega _{FS}(0)=i\,%
\frac{2}{c}\delta _{\alpha \bar{\beta}}\,dz^{\alpha }\wedge d\overline{z}^{%
\overline{\beta }}
\end{equation}

\noindent (consistent with \cite[p.146]{Kod86}, \cite[p.18]{Mok89} but
different from \cite[(5) on p.155]{Kob96}$).$

\vspace*{8pt} \noindent \hspace*{12pt} Let $L$ be a holomorphic line bundle $%
L\rightarrow \mathbb{P}^{n}$ with $\deg (L)=B\in \mathbb{N}$ and $c_{1}(L)$ $%
=$ $[\frac{B}{2\pi }\omega _{FS}]$. Equip $L$ with the Hermitian metric $%
h_{L}$ and $\nabla ^{L}$ the Chern connection on $L$ such that the curvature%
\begin{equation}
\Theta =\overline{\partial }\partial \log h_{L}=-iB\omega _{FS};\text{ }%
\Theta (z=0)=B\,\delta _{\alpha \beta }\,dz^{\alpha }\wedge d\overline{z}^{%
\overline{\beta }}.  \label{4-3}
\end{equation}

%
%

\noindent \hspace*{12pt} In the same way with the notation as in Section 2,
on $\mathbb{P}^{n}$ we have $L\otimes T^{q}$, $L\otimes T^{\ast q}$, $L\odot
T^{q}$, $L\odot T^{\ast q}$ and the spaces $\Omega ^{r,s}(\mathbb{P}^{n})$, $%
\Omega ^{r,s}(\mathbb{P}^{n},L\otimes T^{q})$, $\Omega ^{r,s}(\mathbb{P}%
^{n},L\otimes T^{\ast q})$. Furthermore, with the Levi-Civita connection $%
\nabla $ on $T$ (and $T^{\ast }$) and the Chern connection $\nabla ^{L}$ on $%
L$, we also have the connections $\nabla ^{q}$, $\nabla ^{-q}$, the
operators $\partial ^{q}$, $\partial ^{-q}$, $\overline{\partial }^{\,q}$, $%
\overline{\partial }^{\,-q}$, the contractions $I^{q}$, $I^{-q}$, $\overline{%
I}^{q}$, $\overline{I}^{-q}$, the (real) Hodge operator $\ast $ (see $(\ref%
{stardef})-(\ref{star3})$), the Laplacians $\Delta ^{q}$, $\Delta _{q}$, $%
\Delta ^{-q}$, $\Delta _{-q}$, and the inner product $<\phi ,\psi >$ for
bundle-valued $(r,s)$-forms $\phi ,\psi \in \Omega ^{r,s}(\mathbb{P}%
^{n},L\otimes T^{q}$or $L\otimes T^{\ast q})$ (the formulas up to (\ref%
{Delta q}) in Section 2 work for a K\"{a}hler manifold).\newline

Since $\mathbb{P}^{n}$ is curved, we often need a version of the commutation
relation; unfortunately its complex version is not quite available in the
standard textbook. Following the real version (see for instance \cite[%
pp.148, 217]{Taubes11})\footnote{%
In fact, our context is in line with that of \cite{Taubes11} but different
from \cite{DFN84} and \cite{Spivak99}; see our Notation \ref{notation}. The
formula (\ref{4-7}) does look the same as those of \cite[p.297]{DFN84} and 
\cite[p.5-8]{Spivak99}, but their second covariant derivative is taken on
tensors of one more rank resulting from the index of the first derivative
(that is, the two covariant derivatives are respectively taken on two
different bundles). In our context, the underlying bundle is unchanged when
the second covariant derivative is taken.} 
we get the complex version for K\"{a}hler manifolds:%
\begin{equation}
A_{i_{1}...i_{k},\overline{h}\eta }^{j_{1}...j_{l}}-A_{i_{1}...i_{k},\eta 
\overline{h}}^{j_{1}...j_{l}}=\mbox{$\overset{k}{\underset{r=1}{\sum}}%
\overset{n}{\underset{\nu=1}{\sum}}$}A_{i_{1}...i_{r-1}\nu
i_{r+1}...i_{k}}^{j_{1}...j_{l}}K_{i_{r}\overline{h}\eta }^{\nu }-%
\mbox{$\overset{l}{\underset{s=1}{\sum}}\overset{n}{\underset{\nu=1}{\sum}}$}%
A_{i_{1}...i_{k}}^{j_{1}...j_{s-1}\nu \jmath _{s+1}...j_{l}}K_{\nu \overline{%
h}\eta }^{j_{s}},  \label{4-7}
\end{equation}

\begin{equation}
A_{i_{1}...i_{k},\overline{h}\overline{\eta }%
}^{j_{1}...j_{l}}-A_{i_{1}...i_{k},\overline{\eta }\overline{h}%
}^{j_{1}...j_{l}}=A_{i_{1}...i_{k},h\eta
}^{j_{1}...j_{l}}-A_{i_{1}...i_{k},\eta h}^{j_{1}...j_{l}}=0,  \label{4-7a}
\end{equation}

\noindent where $K_{\overline{j_{s}}\nu \overline{h}\eta }$ ($=$ $g_{\mu 
\overline{j_{s}}}K_{\nu \overline{h}\eta }^{\mu })$ and $K_{i_{r}\bar{h}\eta
}^{\nu }$ are defined in \cite[pp.156-157, (13)-(15)]{Kob96}. 

Different authors may use different sign conventions for curvature in 
\textit{lower indices}, e.g. $R_{\alpha \bar{\alpha}\alpha \bar{\alpha}}$ is
positive on $\mathbb{P}^{n}$ (see \cite[p.31]{Mok89} by Mok) while in the
book \cite[p.169]{Kob96}$\ K_{\alpha \bar{\alpha}\alpha \bar{\alpha}}$ is
negative (but $R_{j\bar{k}l}^{i}$ $=$ $K_{j\overline{k}l}^{i}$ with both
authors). We follow Mok's convention for another reason that his book
contains curvature information for Hermitian symmetric spaces that
generalize the $\mathbb{P}^{n}$ case here. We have (the convention $c=2$
makes the coordinate fields \{$\frac{\partial }{\partial z^{j}}\}_{j}$
unitary at $z=0$)%
\begin{equation}
R_{i\overline{j}k\overline{l}}=\frac{c}{2}\left( g_{i\overline{j}}g_{k%
\overline{l}}+g_{i\overline{l}}g_{k\overline{j}}\right) ,\hspace*{6pt}R_{%
\overline{i}j\overline{k}l}=\overline{R_{i\overline{j}k\overline{l}}},\text{
\ }R_{\overline{i}j\overline{k}l}=R_{\overline{k}j\overline{i}l}\text{\ \ };%
\text{ \ \ \ }\nabla R_{i\overline{j}k\overline{l}}\equiv 0  \label{4-4a}
\end{equation}%
\noindent (cf. $K_{i\overline{j}k\overline{l}}=-\frac{c}{2}\left( g_{i%
\overline{j}}g_{k\overline{l}}+g_{i\overline{l}}g_{k\overline{j}}\right) $
in \cite[p.169]{Kob96}). \noindent In particular, at $z=0$%
\begin{equation}
\begin{cases}
R_{i\overline{i}i\overline{i}}=c & i=1,2...n. \\ 
R_{i\overline{i}j\overline{j}}=R_{i\overline{j}j\overline{i}}=c/2 & 
i,j=1,2...n,i\neq j.%
\end{cases}
\label{Rijkl}
\end{equation}%
\begin{equation}
\begin{cases}
R_{i\overline{j}k\overline{l}}=0 & \text{if at least three of }i,j,k,l\text{
are distinct.} \\ 
R_{i\bar{j}i\overline{l}}=0 & \text{for all }i,j,l\text{ with }j\neq i\text{
or }l\neq i\text{.}%
\end{cases}
\label{4-7-1}
\end{equation}

For computational convenience in what follows we adopt the non-covariant
expressions (at $z_{i}=0$), namely the curvature tensor is written only in
lower indices. More precisely, we are going to use (in Mok's sign convention
as mentioned above): \ 

\begin{lemma}
\label{L-4-1} $i)$ For $\mbox{$\underset{\alpha_{1}...\alpha_{q}}{\sum}$}%
f_{\alpha _{1}...\alpha _{q}}s\otimes (\overset{q}{\underset{k=1}{\otimes }}%
dz^{\alpha _{k}})\in \Omega ^{0,0}(\mathbb{P}^{n},L\otimes (\otimes
^{q}T^{\ast }))$, (\ref{4-7}) reduces to:%
\begin{align}
& \mbox{$\underset{\alpha_{1}...\alpha_{q}}{\sum}$}(f_{\alpha _{1}...\alpha
_{q},\overline{j}i}-f_{\alpha _{1}...\alpha _{q},i\overline{j}})s\otimes (%
\overset{q}{\underset{k=1}{\otimes }}dz^{\alpha _{k}})\hspace*{220pt}
\label{f Ricci 1} \\
=& -\mbox{$\underset{l,\alpha_{1}...\alpha_{q}}{\sum}$}\left( f_{l\alpha
_{2}..\alpha _{q}}R_{\overline{l}\alpha _{1}\overline{j}i}+f_{\alpha
_{1}l\alpha _{3}..\alpha _{q}}R_{\overline{l}\alpha _{2}\overline{j}%
i}...+f_{\alpha _{1}..\alpha _{q-1}l}R_{\overline{l}\alpha _{q}\overline{j}%
i}\right) s\otimes (\overset{q}{\underset{k=1}{\otimes }}dz^{\alpha _{k}}). 
\notag
\end{align}%
$ii)$ For $\mbox{$\underset{\alpha_{1}...\alpha_{q}}{\sum}$}f^{\alpha
_{1}...\alpha _{q}}s\otimes (\overset{q}{\underset{k=1}{\otimes }}\frac{%
\partial }{\partial z^{\alpha _{k}}})\in \Omega ^{0,0}(\mathbb{P}%
^{n},L\otimes (\otimes ^{q}T))$, (\ref{4-7}) reduces to: 
\begin{align}
& \mbox{$\underset{\alpha_{1}...\alpha_{q}}{\sum}$}({f^{\alpha _{1}...\alpha
_{q}}}_{,\overline{j}i}-{f^{\alpha _{1}...\alpha _{q}}}_{,i\overline{j}%
})s\otimes (\overset{q}{\underset{k=1}{\otimes }}\frac{\partial }{\partial
z^{\alpha _{k}}})\hspace*{220pt}  \label{f Ricci 2} \\
=& \mbox{$\underset{l,\alpha_{1}...\alpha_{q}}{\sum}$}\left( f^{l\alpha
_{2}..\alpha _{q}}R_{\overline{\alpha _{1}}l\overline{j}i}+f^{\alpha
_{1}l\alpha _{3}..\alpha _{q}}R_{\overline{\alpha _{2}}l\overline{j}%
i}...+f^{\alpha _{1}..\alpha _{q-1}l}R_{\overline{\alpha _{q}}l\overline{j}%
i}\right) s\otimes (\overset{q}{\underset{k=1}{\otimes }}\frac{\partial }{%
\partial z^{\alpha _{k}}}).  \notag
\end{align}
\end{lemma}

%
%
%
%

\noindent \hspace*{12pt} The following proposition is the $\mathbb{P}^{n}$%
-version of Proposition \ref{10.2PZ}; the RHS coefficient below is different
from that of $M$ due to the nonzero curvature on $\mathbb{P}^{n}$ $($with $%
c=2),$ and the tensors here need to be \textit{symmetric}. 
Being a tedious proof, it is presented in detail because the coefficients
involved should be precisely given for later use in eigenvalues.

\noindent \hspace*{12pt} As $\mathbb{P}^{n}$ is homogeneous, we may, without
loss of generality, only work at $p$ defined by $z_{i}=0$ where $g_{i%
\overline{j}}(0)$ $=$ $\delta _{i\overline{j}}$ and $dg_{i\overline{j}}(0)$ $%
=$ $0.$ This is used for the propositions in the rest of this section
without mention.

\begin{proposition}
\label{Prop 4.2} ($T^{\ast }$-valued) For all $q\in \{0\}\cup \mathbb{N}$,
we have, on $\Omega ^{0,0}(\mathbb{P}^{n},L\otimes (\odot ^{q+1}T^{\ast }))$%
\begin{equation}
{\scriptsize \Delta }^{q}{\scriptsize \overline{I}}^{(q+1)}\overline{%
\partial }^{(q+1)}{\scriptsize -\overline{I}}^{(q+1)}\overline{\partial }%
^{(q+1)}{\scriptsize \Delta }^{(q+1)}{\scriptsize =(-B+}\frac{{\scriptsize c}%
}{{\scriptsize 2}}{\scriptsize (2q+n+1))\,\overline{I}^{(q+1)}\overline{%
\partial }^{(q+1)}.}  \label{DIp-IpD}
\end{equation}%
with values in $\Omega ^{0,0}(\mathbb{P}^{n},L\otimes (\odot ^{q}T^{\ast
})). $
\end{proposition}

\begin{proof}
Using (\ref{2-3-1}) and (\ref{2-3-2}) (valid for K\"{a}hler manifolds as
remarked there) one sees%
\begin{eqnarray}
&&(\Delta ^{q}\overline{I}^{(q+1)}\overline{\partial }^{(q+1)}-\overline{I}%
^{(q+1)}\overline{\partial }^{(q+1)}\Delta ^{q+1})(%
\mbox{$ \underset{\tiny
\alpha_{1},..,\alpha_{q+1}}{\sum}  $}f_{\alpha _{1}...\alpha _{q+1}}s\otimes
(\overset{q+1}{\underset{k=1}{\otimes }}dz^{\alpha _{k}}))\hspace*{120pt}
\label{4-12} \\
&&\overset{at\,p}{=}\big{(}%
\mbox{$- \underset{\tiny
\alpha_{1},..,\alpha_{q+1}}{\sum} \underset{i}{\sum}$}(f_{\alpha
_{1}..\alpha _{q+1},\overline{\alpha _{q+1}}i}-f_{\alpha _{1}..\alpha
_{q+1},i\overline{\alpha _{q+1}}})_{,\overline{i}}s\otimes (\overset{q}{%
\underset{k=1}{\otimes }}dz^{\alpha _{k}})  \notag \\
&&+%
\mbox{$ \underset{\tiny \alpha_{1},..,\alpha_{q+1}}{\sum}
\underset{i}{\sum}$}f_{\alpha _{1}..\alpha _{q+1},\overline{i}}s_{i\overline{%
\alpha _{q+1}}}\otimes (\overset{q}{\underset{k=1}{\otimes }}dz^{\alpha
_{k}})\big{)}_{\mid _{z=0}};  \notag
\end{eqnarray}%
%
%
%
%
%
%
%
%
%
%
%
%
%

\noindent the right-most double summation of (\ref{4-12})\ can be rewritten
as (with Lemma \ref{Lemma 2.1})%
\begin{equation}
-B\overline{I}^{(q+1)}\overline{\partial }^{(q+1)}(%
\mbox{$ \underset{\tiny
\alpha_{1},..,\alpha_{q+1}}{\sum}$}f_{\alpha _{1}..\alpha _{q+1}}s\otimes (%
\overset{q+1}{\underset{k=1}{\otimes }}dz^{\alpha _{k}}))_{\mid _{z=0}}.
\label{eq 4.9}
\end{equation}%
\noindent \hspace*{12pt} By Lemma \ref{L-4-1} for the commutation relation
let us write%
\begin{align}
& \mbox{the first double summation of (\ref{4-12})}\hspace*{240pt}  \notag \\
=& 
\mbox{$ \underset{\tiny \alpha_{1},..,\alpha_{q+1}}{\sum}
\underset{i,l}{\sum}$}(f_{l\alpha _{2}..\alpha _{q+1}}R_{\overline{l}\alpha
_{1}\overline{\alpha _{q+1}}i})_{,\overline{i}}s\otimes (\overset{q}{%
\underset{k=1}{\otimes }}dz^{\alpha _{k}})_{\mid _{z=0}}  \label{eq 4.10} \\
& 
\mbox{$+ \underset{\tiny \alpha_{1},..,\alpha_{q+1}}{\sum}
\underset{i,l}{\sum}$}(f_{\alpha _{1}l..\alpha _{q+1}}R_{\overline{l}\alpha
_{2}\overline{\alpha _{q+1}}i})_{,\overline{i}}s\otimes (\overset{q}{%
\underset{k=1}{\otimes }}dz^{\alpha _{k}})_{\mid _{z=0}}  \notag \\
& \hspace*{20pt}\vdots  \notag \\
& 
\mbox{$+ \underset{\tiny \alpha_{1},..,\alpha_{q+1}}{\sum}
\underset{i,l}{\sum}$}(f_{\alpha _{1}\alpha _{2}..\alpha _{q}l}R_{\overline{l%
}\alpha _{q+1}\overline{\alpha _{q+1}}i})_{,\overline{i}}s\otimes (\overset{q%
}{\underset{k=1}{\otimes }}dz^{\alpha _{k}})_{\mid _{z=0}}.  \notag
\end{align}%
\noindent \hspace*{12pt} Due to the curvature (\ref{Rijkl}) the calculation
of the first line of (\ref{eq 4.10}) is somewhat tedious and we explained it
as follows. See also the proof of Proposition \ref{P-7-1} for supplementary
explanations. We are going to reach (\ref{4-15}) below. To get the first
equality in (\ref{4-15}), we separate into two cases: Case $1)$ $l=\alpha
_{1}$ and Case $2)$ $l\neq \alpha _{1}.$ In Case $1)$ $i$ must be $\alpha
_{q+1}$. Then we have two subcases Case $1)_{a}:$ $\alpha _{1}\neq \alpha
_{q+1}$ (then $R_{\overline{l}\alpha _{1}\overline{\alpha _{q+1}}i}$ $=$ $R_{%
\overline{\alpha _{1}}\alpha _{1}\overline{\alpha _{q+1}}\alpha _{q+1}}$ $=$ 
$\frac{c}{2}$) and Case $1)_{b}:$ $\alpha _{1}=\alpha _{q+1}$ (then $R_{%
\overline{l}\alpha _{1}\overline{\alpha _{q+1}}i}$ $=$ $R_{\overline{\alpha
_{1}}\alpha _{1}\overline{\alpha _{1}}\alpha _{1}}$ $=$ $c).$ In Case $2),$
by using the two vanishings in (\ref{4-7-1}), it is not difficult to see
that $l$ must be $i$ ($\neq \alpha _{1}$ since $l\neq \alpha _{1})$ and then 
$\alpha _{q+1}=\alpha _{1},$ so $R_{\overline{l}\alpha _{1}\overline{\alpha
_{q+1}}i}$ $=$ $R_{\overline{i}\alpha _{1}\overline{\alpha _{1}}i}$ $=$ $R_{%
\overline{\alpha _{1}}\alpha _{1}\overline{i}i}$ $=$ $\frac{c}{2}.$ With $c=%
\frac{c}{2}+\frac{c}{2}$ we write \textit{half}-Case $1)_{b}$ to mean $c$ in
Case $1)_{b}$ is replaced by $\frac{c}{2}.$ Combining Case $2)$ ($i\neq
\alpha _{1},$ $\alpha _{q+1}=\alpha _{1})$ and \textit{half}-Case $1)_{b}$ ($%
i=\alpha _{1},$ $\alpha _{q+1}=\alpha _{1})$ gives the second term of the
second line in (\ref{4-15}). Similarly, combining Case $1)_{a}$ and \textit{%
half}-Case $1)$$_{b}$ gives the corresponding first term. Note $\nabla R_{i%
\overline{j}k\overline{l}}\equiv 0$ (\ref{4-4a}) and $f_{\alpha _{1}\alpha
_{2}..\alpha _{q+1}}$ is symmetric in $\alpha _{1},$ $\alpha _{2},$ $\cdot
\cdot \cdot ,$ $\alpha _{q+1},$ which are needed for the second equality in (%
\ref{4-15}). (The net result that is neat appears lucky.)%
\begin{eqnarray}
&&%
\mbox{$ \underset{\tiny \alpha_{1},..,\alpha_{q+1}}{\sum}
\underset{i,l}{\sum}$}(f_{l\alpha _{2}..\alpha _{q+1}}R_{\overline{l}\alpha
_{1}\overline{\alpha _{q+1}}i})_{,\overline{i}}s\otimes (\overset{q}{%
\underset{k=1}{\otimes }}dz^{\alpha _{k}})_{\mid _{z=0}}  \label{4-15} \\
&=&\displaystyle%
\mbox{$  \frac{c}{2} \underset{\tiny
\alpha_{1},..,\alpha_{q+1}}{\sum} $}(f_{\alpha _{1}\alpha _{2}..\alpha
_{q+1}})_{,\overline{\alpha _{q+1}}}s\otimes (\overset{q}{\underset{k=1}{%
\otimes }}dz^{\alpha _{k}})_{\mid _{z=0}}%
\mbox{$  + \frac{c}{2}
\underset{\tiny i, \alpha_{2},..,\alpha_{q},\alpha_{1}}{\sum} $}(f_{i\alpha
_{2}..\alpha _{q}\alpha _{1}})_{,\overline{i}}s\otimes (\overset{q}{\underset%
{k=1}{\otimes }}dz^{\alpha _{k}})_{\mid _{z=0}}  \notag \\
&=&\mbox{$ c\underset{\tiny \alpha_{1},..,\alpha_{q+1}}{\sum} $}(f_{\alpha
_{1}\alpha _{2}..\alpha _{q+1}})_{,\overline{\alpha _{q+1}}}s\otimes (%
\overset{q}{\underset{k=1}{\otimes }}dz^{\alpha _{k}})_{\mid _{z=0}}.  \notag
\end{eqnarray}%
\noindent \hspace*{12pt} Now that every double summation of $(\ref{eq 4.10})$
is seen to be the same as (\ref{4-15})\ (via the symmetric tensor condition)
except the last term. This last term reads the first equality in (\ref{4-18}%
) below, using $R_{\overline{l}\alpha _{q+1}\overline{\alpha _{q+1}}i}$%
\textbf{\ }$\neq $ $0$ only when $l=i.$ For the second equality below, one
separates into two cases: $i=\alpha _{q+1}$ (getting the coefficient $c=%
\frac{c}{2}\cdot 2)$ and $i\neq \alpha _{q+1}$ (getting the coefficient $%
(n-1)$ times $\frac{c}{2}).$ Thus, the last term of $(\ref{eq 4.10})$ is%
\begin{eqnarray}
&&%
\mbox{$ \underset{\tiny \alpha_{1},..,\alpha_{q+1}}{\sum}
\underset{i,l}{\sum}$}(f_{\alpha _{1}..\alpha _{q}l}R_{\overline{l}\alpha
_{q+1}\overline{\alpha _{q+1}}i})_{,\overline{i}}s\otimes (\overset{q}{%
\underset{k=1}{\otimes }}dz^{\alpha _{k}}))_{\mid _{z=0}}\hspace*{160pt}
\label{4-18} \\
&=&%
\mbox{$ \underset{\tiny \alpha_{1},..,\alpha_{q+1}}{\sum}
\underset{i}{\sum}$}(f_{\alpha _{1}..\alpha _{q}i}R_{\overline{i}\alpha
_{q+1}\overline{\alpha _{q+1}}i})_{,\overline{i}}s\otimes (\overset{q}{%
\underset{k=1}{\otimes }}dz^{\alpha _{k}}))_{\mid _{z=0}}  \notag \\
&=&%
\mbox{$  \frac{c}{2} (2+(n-1)) \underset{\tiny
\alpha_{1},..,\alpha_{q+1}}{\sum}$}(f_{\alpha _{1}..\alpha _{q}\alpha
_{q+1}})_{,\overline{\alpha _{q+1}}}s\otimes (\overset{q}{\underset{k=1}{%
\otimes }}dz^{\alpha _{k}}))_{\mid _{z=0}}.  \notag
\end{eqnarray}%
\noindent The first double summation of (\ref{4-12}) is finally obtained by
substituting $(\ref{4-15})$ ($q$ times) and $(\ref{4-18})$ into $(\ref{eq
4.10})$, resulting in $%
\mbox{$ \frac{c}{2}(2q+n+1)\underset{\tiny
\alpha_{1},..,\alpha_{q+1}}{\sum}
f_{\alpha_{1}...\alpha_{q+1},\overline{\alpha_{q+1}}}s\otimes(\overset{q}{\underset{k=1}{\otimes}}dz^{\alpha_{k}})$}
$; this can be written as (cf. $(\ref{eq 4.9})$)%
\begin{equation}
\mbox{$ \frac{c}{2}(2q+n+1)\,
\overline{I}^{q+1}\overline{\partial}^{q+1}\underset{\tiny
\alpha_{1},..,\alpha_{q+1}}{\sum}
f_{\alpha_{1}...\alpha_{q+1}}s\otimes(\overset{q+1}{\underset{k=1}{\otimes}}dz^{\alpha_{k}}).$}
\label{eq 4.11}
\end{equation}%
\noindent Combining the two double summations $(\ref{eq 4.9})$ and $(\ref{eq
4.11})$ completes (\ref{4-12}) and yields (\ref{DIp-IpD}).
\end{proof}


\noindent \hspace*{12pt} The following is a T-valued analogue of the above
Proposition \ref{Prop 4.2} on $\mathbb{P}^{n},$ parallel to Proposition \ref%
{10.1}. The precise coefficient is important, and its proof will be given in
the Appendix.

\begin{proposition}
\label{Prop 4.3} ($T$-valued) For all $q\in \{0\}\cup \mathbb{N}$, we have,
on $\Omega ^{0,0}(\mathbb{P}^{n},L\otimes \odot ^{q+1}T)$%
\begin{equation}
{\scriptsize \Delta }_{-q}{\scriptsize I}^{-(q+1)}{\scriptsize \partial }%
^{-(q+1)}{\scriptsize -I}^{-(q+1)}{\scriptsize \partial }^{-(q+1)}%
{\scriptsize \Delta }_{-(q+1)}=(B+\frac{c}{2}(2q+n+1))\,{\scriptsize I}%
^{-(q+1)}{\scriptsize \partial }^{-(q+1)}.\hspace*{60pt}  \label{B+cq+c}
\end{equation}%
with values in $\Omega ^{0,0}(\mathbb{P}^{n},L\otimes \odot ^{q}T).$
\end{proposition}

The next two propositions are parallel to, but not exact analogues of,
Proposition \ref{9}. The identity $\Delta ^{k}-\Delta
_{k}=n(I^{(k-1)}\partial ^{(k-1)}\overline{I}^{k}\overline{\partial }^{k}-%
\overline{I}^{(k+1)}\overline{\partial }^{(k+1)}I^{k}\partial ^{k})$ for $%
k\in \mathbb{N}$ on an Abelian variety$,$ does not hold for $\mathbb{P}^{n}$%
. We calculate them separately. Their proofs will be given in the Appendix.
The parameter $c$ below takes the value $2.$

\begin{proposition}
\label{Prop 4.4} (canonical commutation relations for creation-annihilation
operators, cf. (\ref{ACs}))


$i)$ For all $q\in \mathbb{N}$, $I^{(q-1)}\partial ^{(q-1)}\overline{I}^{q}%
\overline{\partial }^{q}-\overline{I}^{(q+1)}\overline{\partial }%
^{(q+1)}I^{q}\partial ^{q}=[B-\frac{c}{2}(2q+n-1)]\cdot Id$ $%
\mbox{on\hspace*{4pt} $\Omega
^{0,0}\left(\mathbb{P}^{n}, L\otimes \left( \odot^{q} T^{\ast}\right)
\right) $}.$

$ii)$ $I^{-1}\partial ^{-1}\overline{I}^{0}\overline{\partial }^{0}-%
\overline{I}^{1}\overline{\partial }^{1}I^{0}\partial ^{0}=nB\cdot Id%
\hspace*{4pt}%
\mbox{on\hspace*{4pt} $\Omega ^{0,0}\left(\mathbb{P}^{n}, L
\right) $}.$

$iii)$ For all $q\in \mathbb{N}$, $I^{-(q+1)}\partial ^{-(q+1)}\overline{I}%
^{\,-q}\,\overline{\partial }^{\,-q}-\overline{I}^{\,-(q-1)}\,\overline{%
\partial }^{\,-(q-1)}I^{-q}\partial ^{-q}$ $=$ $[B+\frac{c}{2}(2q+n-1)]\cdot
Id$

\noindent on $\Omega^{0,0}\left(\mathbb{P}^{n}, L\otimes \left( \odot^{q}
T\right) \right)$. 
\end{proposition}

\begin{proposition}
\label{Prop 4.5}


$i)$ For all $q\in \mathbb{N}$, $\Delta ^{q}-\Delta _{q}=[-\frac{c}{2}%
(n+1)q+nB]\cdot Id$ on $\Omega ^{0,0}(\mathbb{P}^{n},L\otimes (\otimes
^{q}T^{\ast })).$

$ii)$ $\Delta ^{0}-\Delta _{0}=nB\cdot Id\hspace*{4pt}%
\mbox{on\hspace*{4pt} $\Omega
^{0,0}\left(\mathbb{P}^{n}, L \right) $}.$

$iii)$ For all $q\in \mathbb{N}$, $\Delta ^{-q}-\Delta _{-q}=[\frac{c}{2}%
(n+1)q+nB]\cdot Id$ on $\Omega ^{0,0}(\mathbb{P}^{n},L\otimes (\otimes
^{q}T)).$%

\noindent The same formulas hold true on symmetric tensors $\Omega ^{0,0}(%
\mathbb{P}^{n},L\otimes (\odot ^{q}T^{\ast }))$ for $i)$ and $\Omega ^{0,0}(%
\mathbb{P}^{n},L\otimes (\odot ^{q}T))$ for $iii)$ respectively.
\end{proposition}

\begin{remark}
\label{R-4-6}
For complex hyperbolic spaces, all the Bochner-Kodaira type identities in
this section are basically true as the calculations are virtually the same
(where $c$ is then set to be $-2)$. 
\end{remark}

\section{\textbf{Proof of Theorem \protect\ref{TeoC} and Corollary \protect
\ref{TeoD}\label{Sec5}}}

The Bochner-Kodaira type identities established in the last section will be
used in this section.

\begin{proof}
(of \textbf{Theorem \ref{TeoC}} $ii)$) Although the proof is in a similar
spirit to that of Theorem \ref{TeoA} $ii)$ in Section \ref{Sec3}, the
clarity on the coefficients is desirable and we make the argument parallel
to that of Section \ref{Sec3} in order to keep track of these quantities.

\noindent \hspace*{12pt} For $k=1$, by Lemma \ref{Lemma1} , 
\begin{equation*}
\left( I^{-1}\partial ^{-1}\right) s^{-1}=\left( I^{-1}\partial ^{-1}\right)
\left( I^{-1}\partial ^{-1}\right) ^{\ast }s^{0}\overset{Lem.\ref{Lemma1}}{=}%
\Delta _{0}s^{0}\overset{\text{assumption}}{=}q(B+\frac{c}{2}(n+q))s^{0}\neq
0.\hspace*{10pt}
\end{equation*}%
\noindent We have $q(B+\frac{c}{2}(n+q))>qB>0$ as $B,q,c>0$, which implies $%
s^{-1}\neq 0$.

\noindent \hspace*{12pt} For $q>1$ and $k=2$, $\left( I^{-2}\partial
^{-2}\right) s^{-2}\in \Omega ^{0,0}(\mathbb{P}^{n},L\otimes T)$ by Lemma $%
\ref{-q-q}$, 
\begin{align*}
\left( I^{-2}\partial ^{-2}\right) s^{-2}& =\left( I^{-2}\partial
^{-2}\right) \left( I^{-2}\partial ^{-2}\right) ^{\ast }s^{-1}\overset{Prop.%
\ref{8}\,i)}{\underset{Lem.\ref{lemmaA}}{=}}-\big(I^{-2}\partial ^{-2}\big)%
\big(\overline{I}^{\,-1}\overline{\partial }^{\,-1}\big)s^{-1} \\
& \hspace*{-20pt}\overset{Prop.\ref{Prop 4.4}\text{ }iii)}{=}-\big(\big(%
\overline{I}^{\,-0}\overline{\partial }^{\,-0}\big)\big(I^{-1}\partial ^{-1}%
\big)+(B+c)+\frac{c}{2}(n-1)\big)s^{-1}\in \Omega ^{0,0}(\mathbb{P}%
^{n},L\otimes T)
\end{align*}%
\noindent which equals 
\begin{align*}
& -[\big(\mathcal{S}\overline{I}^{\,-0}\overline{\partial }^{\,-0}\big)\big(%
I^{-1}\partial ^{-1}\big)+(B+c)+\frac{c}{2}(n-1)]s^{-1} \\
& \overset{Prop.\ref{8}\,i)}{=}[\big(I^{\,-1}\partial ^{\,-1}\big)^{\ast }%
\big(I^{-1}\partial ^{-1}\big)-(B+c)-\frac{c}{2}(n-1)]s^{-1} \\
\overset{k=1}{=}& \big(I^{\,-1}\partial ^{\,-1}\big)^{\ast }((qB+\frac{c}{2}%
(q^{2}+q)+\frac{c}{2}(n-1)q)s^{0})-(B+c)s^{-1}-\frac{c}{2}(n-1)s^{-1} \\
& =\big{(}(q-1)B+\frac{c}{2}((q^{2}-1)+(q-1))+\frac{c}{2}(n-1)(q-1)\big{)}%
s^{-1}\neq 0.
\end{align*}%
\noindent Hence $s^{-2}\neq 0$. By repeating this process (less
straightforward yet similar to that for (\ref{3-0})), for $s^{-k}=(-1)^{k}(%
\overline{I}^{\,-(k-1)}\overline{\partial }^{\,-(k-1)})...(\overline{I}%
^{\,-0}\overline{\partial }^{\,-0})\,s^{0}\in \Omega ^{0,0}$\ $\left( 
\mathbb{P}^{n},L\otimes \left( \odot ^{k}T\right) \right) $, one has 
\begin{equation}
\left( I^{-k}\partial ^{-k}\right) s^{-k}=\Big{(}%
\begin{array}{c}
(q-(k-1))B+\frac{c}{2}(q^{2}-(k-1)^{2})+\frac{c}{2}(q-(k-1)) \\ 
+\frac{c}{2}(n-1)(q-(k-1))%
\end{array}%
\Big{)}s^{-(k-1)}\neq 0  \label{5-0}
\end{equation}%
\noindent for $q\geq k\geq 1,$ $c>0,$ $B>0$, proving inductively that $%
s^{-k}\neq 0$. In particular $s^{-q}\neq 0$.\newline
\noindent \hspace*{12pt} For the holomorphicity of $s^{-q}$, we follow (\ref%
{STAR}) with (\ref{B+cq+c}) in place of (\ref{3-1}) to get 
\begin{equation}
\left( I^{-k}\partial ^{-k}\right) ^{\ast }\Delta _{-(k-1)}-\Delta
_{-k}\left( I^{-k}\partial ^{-k}\right) ^{\ast }=(B+ck+\frac{c}{2}%
(n-1))\left( I^{-k}\partial ^{-k}\right) ^{\ast }\hspace*{120pt}  \label{5-1}
\end{equation}%
mapping $\Omega ^{0,0}(\mathbb{P}^{n},L\otimes (\odot ^{k-1}T))$ to $\Omega
^{0,0}(\mathbb{P}^{n},L\otimes (\odot ^{k}T))$. By using $(\ref{5-1})$
recursively, 
\begin{align}
\Delta _{-q}\text{ }s^{-q}& {\hspace*{8pt}=\hspace*{8pt}}\Delta _{-q}\left(
I^{-q}\partial ^{-q}\right) ^{\ast }s^{-\left( q-1\right) }=\cdots (%
\mbox{similar to $(\ref{holomor})$})  \notag  \label{eq5.3} \\
& {\hspace*{8pt}=\hspace*{8pt}}\left( I^{-q}\partial ^{-q}\right) ^{\ast
}\left( I^{-\left( q-1\right) }\partial ^{-\left( q-1\right) }\right) ^{\ast
}...\left( I^{-1}\partial ^{-1}\right) ^{\ast }\Delta _{0}\text{ }s^{0}-(qB+%
\frac{c}{2}q(n+q))s^{-q} \\
& {\hspace*{8pt}=\hspace*{8pt}}(qB+\frac{c}{2}q(n+q))s^{-q}-(qB+\frac{c}{2}%
q(n+q))s^{-q}=0  \notag
\end{align}%
\noindent giving $s^{-q}$ is holomorphic as $\Delta _{-q}=(\overline{%
\partial }^{\,-q})^{\ast }\overline{\partial }^{\,-q}$.\newline
\noindent \hspace*{12pt} The proof for $s^{-k}$ is similar to $(\ref{eq5.3})$%
. By using $(\ref{5-1})$ recursively, for $k=1,2,...q$ we have 
\begin{eqnarray}
\Delta _{-k}\text{ }s^{-k} &=&\Delta _{-k}\left( I^{-k}\partial ^{-k}\right)
^{\ast }s^{-\left( k-1\right) }=\cdots (\mbox{similar to
(\ref{3-4})})  \notag \\
&=&\big{(}(q-k)B+\frac{c}{2}(q^{2}-k^{2})+\frac{c}{2}(q-k)+\frac{c}{2}%
(n-1)(q-k)\big{)}s^{-k}.\hspace*{40pt}
\end{eqnarray}
\end{proof}


\noindent \hspace*{12pt}The following theorem and Corollary are parallel to
Theorem \ref{q+n+k} and Proposition \ref{alleigen} respectively, but see
Remark \ref{R-5-2} below for the difference. Similar to Theorem \ref{q+n+k},
we define $s^{k}:=(\overline{I}^{k}\overline{\partial }^{k})^{\ast }...(%
\overline{I}^{1}\overline{\partial }^{1})^{\ast }s^{0}\in \Omega ^{0,0}(%
\mathbb{P}^{n},L\otimes (\odot ^{k}T^{\ast }))$ for $k\in \mathbb{N},$ where 
$s^{0}\in \Omega ^{0,0}\left( \mathbb{P}^{n},L\right) .$

\begin{theorem}
\label{T-5-1}\noindent Suppose that $0\neq s^{0}\in \Omega ^{0,0}\left( 
\mathbb{P}^{n},L\right) $ is an eigensection of $\Delta ^{0}$ with
eigenvalue $(q+n)B+\frac{cq(n+q)}{2}$ for some $q\in \{0\}\cup \mathbb{N}$.
Let $s^{k}$ be as above. Then $s^{k}$ is a nontrivial eigensection of $%
\Delta ^{k}$ with eigenvalue $(q+n+k)[B+\frac{c}{2}(q-k)]\geq $ $0$ for $%
1\leq k\leq $ $B+q$ (where $c=2).$ Moreover, for $k=B+q,$ $s^{k}$ is an
anti-holomorphic section. Further, $s^{B+q+1}=0$. (For the existence of $%
s^{0}$ for every $q\in \{0\}\cup \mathbb{N},$ see Remark \ref{R-5-13} $b)$;
for a correspondence between anti-holomorphic sections and holomorphic ones,
compare (7-7-1) in the proof of Theorem \ref{T-7-1}.) 
\end{theorem}

\begin{proof}
We use the adjoint of (\ref{DIp-IpD}) in place of that of (\ref{-BBB}): 
\begin{equation}
{\scriptsize (\overline{I}}^{k}\overline{\partial }^{k})^{\ast }{\scriptsize %
\Delta }^{(k-1)}-{\scriptsize \Delta }^{k}({\scriptsize \overline{I}}^{k}%
\overline{\partial }^{k})^{\ast }{\scriptsize =[-B+\frac{c}{2}(n+2k-1)]\,}(%
\overline{I}^{(k)}\overline{\partial }^{k})^{\ast }  \label{5-4a}
\end{equation}%
\noindent to get the stated eigenvalue for the eigensection $s^{k};$ see the
lines above Theorem \ref{q+n+k}. For $1\leq k\leq B+q,$ $0\neq s^{k}$
follows the argument in the proof of Theorem \ref{q+n+k} where Proposition %
\ref{9} $i)$ is now replaced by Proposition \ref{Prop 4.4} $i)$: It is
slightly tedious but straightforward to derive an analogue $\overline{I}^{k}%
\overline{\partial }^{k}s^{k}=f(k)s^{k-1}$ where $f(k)=(q+n+k-1)B+\frac{c}{2}%
[q(n+q)-(k-1)(n+k-1)]$ ($c=2$) is quadratic in $k$ and is positive for $%
1\leq k\leq B+q$.


When $k=B+q,$ the stated eigenvalue formula gives $0.$ Then one has $0$ $=$ $%
\Delta ^{k}s^{k}$ $=$ ($\partial ^{k})^{\ast }\partial ^{k}s^{k}$ so $%
\partial ^{k}s^{k}$ $=$ $0,$ i.e. $s^{B+q}$ is anti-holomorphic. This and $%
s^{B+q+1}=(\overline{I}^{B+q+1}\overline{\partial }^{B+q+1})^{\ast }s^{B+q}$
give $s^{B+q+1}=0$ by Proposition \ref{8} $iii).$
\end{proof}

\begin{remark}
\label{R-5-2} On an Abelian variety (with positive $L$) all the eigenvalues
of $\Delta ^{k}$ are positive for any $k\in \{0\}\cup \mathbb{N}$ (Theorem %
\ref{q+n+k} and Corollary \ref{C-3-3}). But on $\mathbb{P}^{n}$ (with
positive $L$) $\Delta ^{k}$ has a zero eigenvalue as long as $k\geq B$ by
the above result.
\end{remark}

\begin{proposition}
\label{alleigen2} The set $S$ of eigenvalues of $\Delta _{0}$ on $\Omega
^{0.0}(\mathbb{P}^{n},L)$ is contained in the set $S^{\prime }$ $=$ $\{qB+%
\frac{c}{2}q(n+q)$ \TEXTsymbol{\vert} $q\in \{0\}\cup \mathbb{N}\}.$ (In
fact $S=S^{\prime }$ by Remark \ref{R-5-13} $a)$.)
\end{proposition}

\begin{proof}
We adopt the similar notation $s_{\lambda }\equiv s_{\lambda }^{0},$ $%
s_{\lambda }^{-k}$ as that in the proof of Proposition \ref{alleigen}, where
the eigenvalue $\lambda \in \left( qB+\frac{c}{2}q(n+q),\left( q+1\right) B+%
\frac{c}{2}(q+1)(n+q+1)\right) $ for some $q\in \mathbb{N}\cup \{0\}.$ We
follow the proof of (\ref{3-2a}) and (\ref{3-0}) except that Proposition \ref%
{Prop 4.4} $iii)$ replaces Proposition \ref{9} $iii)$ to obtain an analogue
of (\ref{3-2a})%
\begin{equation*}
(I^{-k}\partial ^{-k})s_{\lambda }^{-k}=[\lambda -(k-1)(B+\frac{c}{2}%
(n+k-1)]s_{\lambda }^{-(k-1)}\text{ \ \ }(c=2)
\end{equation*}

\noindent whose coefficient is larger than $0$ for $k\leq q+1,$ giving $%
s_{\lambda }^{-k}\neq 0$ and $s_{\lambda }^{-(q+1)}\neq 0$.\ Similar to (\ref%
{5-4a}) but with (\ref{B+cq+c}) in place of (\ref{DIp-IpD}), we obtain%
\begin{equation}
({\scriptsize I}^{-k}{\scriptsize \partial }^{-k})^{\ast }{\scriptsize %
\Delta }_{-(k-1)}{\scriptsize -\Delta _{-k}(I}^{-k}{\scriptsize \partial }%
^{-k})^{\ast }=(B+\frac{c}{2}(n+2k-1))\,({\scriptsize I}^{-k}{\scriptsize %
\partial }^{-k})^{\ast }.  \label{5-4b}
\end{equation}%
\noindent Starting with $\Delta _{0}s_{\lambda }^{0}=\lambda \cdot
s_{\lambda }^{0}$ and applying (\ref{5-4b}) to $s_{\lambda }^{-(k-1)}$
recursively we have%
\begin{equation}
\Delta _{-1}s_{\lambda }^{-1}=(\lambda -B-\frac{c}{2}(n+1))s_{\lambda
}^{-1},\cdots ,\text{ }\Delta _{-k}s_{\lambda }^{-k}=(\lambda -kB-\frac{c}{2}%
k(n+k))s_{\lambda }^{-k},\hspace*{20pt}k\in \mathbb{N}.  \label{5-4c}
\end{equation}

\noindent As similar to (\ref{3-2b}) the last equality gives a contradiction
if $k>q.$
\end{proof}

We leave the formulation of an analogue (on $\mathbb{P}^{n})$ of Corollary %
\ref{C-3-3} (on $M)$ to the interested reader.

\begin{proof}
(of \textbf{Theorem \ref{TeoC}}\thinspace\ $iii)$) Recall the notations $%
0\neq t^{-q}$ $\in $ $H^{0}(\mathbb{P}^{n},L\otimes (\odot ^{q}T))$, $%
t^{-k}:=A_{k}^{\#}t^{-q}$ $(=$ $(I^{-(k+1)}\partial
^{-(k+1)})...(I^{-q}\partial ^{-q})t^{-q}$), $0\leq k<q,$ and $N_{k,l}$ $:=$ 
$(k-l)B$ $+$ $\frac{c}{2}[(k^{2}-l^{2})+n(k-l)]$ ($c=2$) $>0$ if $k>l\geq 0.$

%
%
%
For $0\leq k<q$ we are going to show $t^{-k}\neq 0.$ It is similar to (\ref%
{3-8a}) except that Proposition \ref{Prop 4.4} $iii)$ replaces Proposition %
\ref{9} $iii)$ and Lemma \ref{L-2-13} is used on $\mathbb{P}^{n}$. Below, we
mainly focus on the precise coefficients. First, we follow the proof of (\ref%
{3-5}) to get%
\begin{eqnarray}
\left( I^{-q}\partial ^{-q}\right) ^{\ast }t^{-\left( q-1\right) } &=&-%
\mathcal{S(}I^{-(q+1)}\partial ^{-(q+1)}\overline{I}^{-q}\overline{\partial }%
^{-q}-(B+cq+\frac{c}{2}(n-1)))t^{-q}  \label{5-b} \\
&=&N_{q,q-1}t^{-q}\text{ \ since }\overline{\partial }^{-q}t^{-q}=0,  \notag
\end{eqnarray}%
\noindent so $t^{-\left( q-1\right) }\neq 0.$ The reasoning similar to (\ref%
{3-10b}) and (\ref{3-10c}) gives 
\begin{eqnarray*}
\left( I^{-(q-1)}\partial ^{-(q-1)}\right) ^{\ast }t^{-\left( q-2\right) }
&=&-\mathcal{S}(I^{-q}\partial ^{-q}\overline{I}^{\,-(q-1)}\overline{%
\partial }^{\,-(q-1)}-N_{q-1,q-2})t^{-(q-1)} \\
&\overset{Lem.\ref{L-2-13}}{=}&I^{-q}\partial ^{-q}\left( I^{-q}\partial
^{-q}\right) ^{\ast }t^{-\left( q-1\right) }+N_{q-1,q-2}t^{-(q-1)} \\
&\overset{(\ref{5-b})}{=}&N_{q,q-1}(I^{-q}\partial
^{-q})t^{-q}+N_{q-1,q-2}t^{-(q-1)} \\
&=&(N_{q,q-1}+N_{q-1,q-2})t^{-(q-1)}.
\end{eqnarray*}%
\noindent Recursively, for $0\leq k<q$ we have 
\begin{equation}
\left( I^{-(k+1)}\partial ^{-(k+1)}\right) ^{\ast }t^{-k}=\big{(}%
\sum_{l=0}^{q-1-k}N_{q-l,q-l-1}\big{)}t^{-(k+1)}  \label{5-ca}
\end{equation}%
giving $t^{-k}\neq 0$ inductively for $k=q-1,q-2,...,0$.
%
%
%

%
\noindent \hspace*{12pt} To show that $t^{0}$ is an eigensection of $\Delta
_{0}$ we follow (\ref{3-5z}) closely to keep track of the coefficients, and
use Proposition \ref{Prop 4.3} instead of Proposition \ref{10.1}:%
\begin{eqnarray}
\ \ \ \Delta _{0}\text{ }t^{0} &=&\Delta _{0}\left( I^{-1}\partial
^{-1}I^{-2}\partial ^{-2}...I^{-q}\partial ^{-q}\right) t^{-q}  \label{Eigen}
\\
&=&I^{-1}\partial ^{-1}\Delta _{-1}I^{-2}\partial ^{-2}...I^{-q}\partial
^{-q}t^{-q}+(B+c+\frac{c}{2}(n-1))I^{-1}\partial ^{-1}I^{-2}\partial
^{-2}...I^{-q}\partial ^{-q}t^{-q}  \notag \\
&=&I^{-1}\partial ^{-1}I^{-2}\partial ^{-2}...I^{-q}\partial ^{-q}\Delta
_{-q}t^{-q}+(qB+\frac{c}{2}q(n+q))I^{-1}\partial ^{-1}I^{-2}\partial
^{-2}...I^{-q}\partial ^{-q}t^{-q}  \notag \\
&=&I^{-1}\partial ^{-1}I^{-2}\partial ^{-2}...I^{-q}\partial ^{-q}\big((%
\overline{\partial }^{\,-q})^{\ast }\overline{\partial }^{\,-q}\big)%
t^{-q}+(qB+\frac{c}{2}q(n+q))t^{0}  \notag \\
&=&0+(qB+\frac{c}{2}q(n+q))t^{0}.  \notag
\end{eqnarray}

\noindent The similar recursion proves the case for $t^{-k}$.%
\end{proof}

\begin{proof}
(of \textbf{Theorem \ref{TeoC}}\thinspace\ $i)$) As before, $N_{k,l}$
denotes $(k-l)B$ $+$ $(k^{2}-l^{2})+n(k-l)$ (where $c=2$) with the $N_{q,0}$%
-eigenspace of $\Delta _{0}$ being $E_{0,q,B,c}.$ To show that $%
(I^{-q}\partial ^{-q})^{\ast }...(I^{-1}\partial ^{-1})^{\ast }$ (=$A_{q})$ $%
:$ $E_{0,q,B,c}$ $\rightarrow $ $(E_{q,q,B,c}$ $=)$ $H^{0}(\mathbb{P}%
^{n},L\otimes (\odot ^{k}T))$ is bijective, we let $0\neq s^{0}\in
E_{0,q,B,c}$ ($\subset $ $\Omega ^{0,0}(\mathbb{P}^{n},L)$ and $%
s^{-k}:=(I^{-k}\partial ^{-k})^{\ast }...(I^{-1}\partial ^{-1})^{\ast }s^{0}$
for $k=1,...,q$ as in Theorem \ref{TeoC} $ii).$ We follow the proof of
Theorem \ref{TeoA} $i)$ by using (\ref{5-0}) in place of (\ref{3-0}), and get

%
%
%
\begin{equation}
(\prod\limits_{k=1}^{q}N_{q,k-1})^{-1}A_{0}^{\#}A_{q}=Id\text{ \ \ \ on }%
E_{0,q,B,c}  \label{5-a}
\end{equation}

\noindent (noting that $B>0,$ $c=2$, so $N_{q,k-1}>0$ for $q\geq k\geq 1).$
Let $t^{-q}$ $\in $ $H^{0}(\mathbb{P}^{n},L\otimes (\odot ^{q}T))$ and set $%
t^{-k}:=A_{k}^{\#}t^{-q}$ $(=$ $(I^{-(k+1)}\partial
^{-(k+1)})...(I^{-q}\partial ^{-q})t^{-q}$), $0\leq k<q.$ We obtain, via (%
\ref{5-ca})
%
\begin{equation*}
\left( I^{-q}\partial ^{-q}\right) ^{\ast }...\left( I^{-1}\partial
^{-1)}\right) ^{\ast
}t^{0}=\prod\limits_{k=0}^{q-1}(\sum_{l=0}^{k}N_{q-l,q-l-1})t^{-q}
\end{equation*}%
\noindent i.e. 
\begin{equation}
\prod\limits_{k=0}^{q-1}(\sum_{l=0}^{k}N_{q-l,q-l-1})^{-1}A_{q}A_{0}^{\#}=Id%
\text{ \ \ \ on }H^{0}(\mathbb{P}^{n},L\otimes (\odot ^{q}T)).  \label{5-d}
\end{equation}%
\noindent It is easily seen that $N_{q,q-1}$ $+$ $N_{q-1,q-2}$ $+$ $\cdot
\cdot \cdot $ $+$ $N_{k+1,k}$ $=$ $N_{q,k}$ for $k=0,$ $\cdot \cdot \cdot $, 
$q-1,$ giving $\sum_{m=0}^{l}N_{q-m,q-m-1}$ $=$ $N_{q,q-l-1}$, so that $%
\prod\limits_{l=0}^{q-1}(\sum_{m=0}^{l}N_{q-m,q-m-1})^{-1}$ $=$ ($%
N_{q,q-1}\cdot N_{q,q-2}$ $\cdot \cdot \cdot $ $N_{q,0})^{-1}$, i.e. the two
associated constants in (\ref{5-d}) and (\ref{5-a}) coincide.
\end{proof}

\begin{proof}
(of \textbf{Corollary \ref{TeoD}}) The assertion $\dim E_{0,q,B,c}=$ $h^{0}(%
\mathbb{P}^{n},L\otimes (\odot ^{q}T))$ follows from Theorem \ref{TeoC} $i)$
above. This is nonzero: For $q=0,$ $h^{0}(\mathbb{P}^{n},L)$ $>$ $0$ is
obvious. Let $V$ be a nontrivial holomorphic vector field on $\mathbb{P}^{n}$
and $0\neq s$ $\in $ $H^{0}(\mathbb{P}^{n},L).$ Then $0\neq s\otimes
(\otimes ^{q}V)$ is symmetric and lies in $H^{0}(\mathbb{P}^{n},L\otimes
(\odot ^{q}T)).$
\end{proof}

\begin{remark}
\label{R-5-13} $a)$ By Corollary \ref{TeoD} and Proposition \ref{alleigen2}
we obtain that the set of eigenvalues of $\Delta _{0}$ on $\Omega ^{0,0}(%
\mathbb{P}^{n},L)$ is exactly $\{qB+\frac{c}{2}q(n+q)$ \TEXTsymbol{\vert} $%
q\in \{0\}\cup \mathbb{N}\}.$ $b)$ The assumption for the existence of $%
s^{0} $ in Theorem \ref{T-5-1} is met for every $q$ $\in $ $\{0\}\cup 
\mathbb{N}$ by using $a)$ and Proposition \ref{Prop 4.5} $ii).$
\end{remark}

\section{\textbf{Proof of Theorem \protect\ref{TeoE}\label{Sec6}}}

\noindent \hspace*{12pt} To calculate $h^{0}(\mathbb{P}^{n},L\otimes (\odot
^{q}T))$, we need a vanishing theorem from \cite{Mani97}:

\begin{theorem}[Vanishing Theorem \protect\cite{Mani97}]
\label{l+n-p} Let $E$ be a holomorphic vector bundle of rank $e$, and $L$ a
(holomorphic) line bundle on a smooth projective complex variety $N$ of
dimension $n$. Suppose that $E$ is ample and $L$ nef, or that $E$ is nef and 
$L$ is ample. Then for any sequences of integers $k_{1}$,...,$k_{l}$ and $%
j_{1}$,...$j_{m}$, the Dolbeault cohomology groups%
\begin{equation}
H^{p,i}(N,(\odot ^{k_{1}}E)\otimes ...\otimes (\odot ^{k_{l}}E)\otimes
(\wedge ^{j_{1}}E)\otimes ...\otimes (\wedge ^{j_{m}}E)\otimes
(det(E)^{l+n-p})\otimes L)=0
\end{equation}%
as soon as $p+i>n+\sum_{s=1}^{m}(e-j_{s})$.
\end{theorem}

\noindent \hspace*{12pt} The above vanishing theorem enables us to compute $%
h^{0}$ using the Hirzebruch-Riemann-Roch Theorem:

\begin{lemma}
\label{h0} Let $L$ be a positive line bundle on $\mathbb{P}^{n}$. Then%
\begin{equation}
h^{0}(\mathbb{P}^{n},L\otimes (\odot ^{q}T))(=\dim H^{0}(\mathbb{P}%
^{n},L\otimes (\odot ^{q}T)))=\int_{\mathbb{P}^{n}}\text{\mbox{td}}(T)\cdot 
\text{\mbox{ch}}(L\otimes (\odot ^{q}T)).  \label{6-2}
\end{equation}
\end{lemma}

\begin{proof}
In Theorem \ref{l+n-p}, we let $N=\mathbb{P}^{n}$, $E=T$, $p=n$, $l=1$, and $%
m=0$. Since $T$ and $L$ (positive line bundle) are both ample and the needed
condition $p+i>n+\sum_{s=1}^{m}(e-j_{s})$ reduces to $n+i>n$ which is always
true for all $i\in \mathbb{N}$, the formula in Theorem \ref{l+n-p} becomes 
\begin{equation}
H^{n,i}(\mathbb{P}^{n},(\odot ^{q}T)\otimes \det (T)\otimes L)=0,\hspace*{%
20pt}\forall i\in \mathbb{N}.\hspace*{60pt}
\end{equation}%
\noindent By the Dolbeault theorem 
\begin{equation*}
H^{n,i}(\mathbb{P}^{n},(\odot ^{q}T)\otimes \det (T)\otimes L)=H^{i}(\mathbb{%
P}^{n},(\odot ^{q}T)\otimes \det (T)\otimes L\otimes K_{\mathbb{P}^{n}}),%
\hspace*{20pt}\forall i\in \mathbb{N}
\end{equation*}%
\noindent giving that, since det$(T)\otimes K_{\mathbb{P}^{n}}$ is trivial, 
\begin{equation*}
H^{i}(\mathbb{P}^{n},L\otimes (\odot ^{q}T))=0,\hspace*{10pt}\forall i\in 
\mathbb{N}.
\end{equation*}%
\noindent This vanishing and the HRR Theorem (cf. \cite[Theorem 5.1.1, p.232]%
{Huy05}) yield (\ref{6-2}).
\end{proof}

%
%

The following is well known (cf. \cite[p.166]{GH84}):

\begin{lemma}
\label{Lkl0h0} 
\begin{equation}
C_{n}^{k+n}=h^{0}(\mathbb{P}^{n},\mathcal{O}(k))=\int_{\mathbb{P}^{n}}\text{%
\mbox{td}}(T)\cdot \text{\mbox{ch}}(\mathcal{O}(k)).\hspace*{20pt}
\end{equation}
\end{lemma}

\noindent \hspace*{12pt} The main calculation for $\int_{\mathbb{P}^{n}}$%
\mbox{td}$(\cdot )$\mbox{ch}$(\mathcal{\cdot \cdot })$ of Lemma \ref{h0}
starts with the splitting principle.\ For $T\rightarrow \mathbb{P}^{n}$
letting $\pi :F(T)\rightarrow \mathbb{P}^{n}$ be the projection from the
split manifold $F(T)$ to $\mathbb{P}^{n}$ (\cite[pp.273-279]{BoTu82}, \cite[%
pp.54-56]{Ful84}) one splits the vector bundle $\pi ^{\ast }T\rightarrow
F(T) $ as a direct sum of line bundles $\pi ^{\ast }T=L_{1}\oplus ...\oplus
L_{n}$ on $F(T)$, and 
\begin{equation}
\pi ^{\ast }(\odot ^{q}T)=\underset{\substack{k_{1}+...+k_{n}=q,\\ k_{i}\in 
\mathbb{N}\cup\{0\}}}{\oplus }{\large ((L_{1})^{k_{1}}\otimes ...\otimes
(L_{n})^{k_{n}}).}
\end{equation}%
Let $c_{t}(T)$ denote the Chern polynomial of $T\rightarrow \mathbb{P}^{n}.$
We have $\pi ^{\ast }(c_{t}(T))=c_{t}(\pi ^{\ast }T)=\overset{n}{\underset{%
j=1}{\prod }}(1+\pi ^{\ast }(\lambda _{j}\omega )t)$ with $c_{1}(L_{j})=\pi
^{\ast }(\lambda _{j}\omega )$, $\omega =c_{1}(\mathcal{O}(1))$ where $%
\lambda _{j}$ are the roots of $%
x^{n}+(-1)^{1}C_{1}^{n+1}x^{n-1}+...+(-1)^{n}C_{n}^{n+1}=0$ \cite[Remark
3.2.3 on p.54]{Ful84}\footnote{%
In fact we can solve $\lambda _{j}$ explicitly. By the change of variabes $x=%
\frac{-1}{y}$ we are reduced to solving $y\neq 0$ for $0$ $=$ $1$ $+$ $%
C_{1}^{n+1}y$ $+$ $C_{2}^{n+1}y^{2}$ $+$ $\cdot \cdot \cdot $ $+$ $%
C_{n}^{n+1}y^{n}$ $=$ $(1+y)^{n+1}$ $-$ $y^{n+1}$ $=$ $y^{n+1}((\frac{1}{y}%
+1)^{n+1}-1).$ We obtain $\frac{1}{y}+1$ $=$ $e^{i\frac{2\pi j}{n+1}},$ $%
j=1, $ $\cdot \cdot \cdot $ $,n,$ and then $x$ $=$ $1-e^{i\frac{2\pi j}{n+1}%
} $ ($= $ $\lambda _{j}).$} using $c(T)$ $=$ $(1+\omega )^{n+1}$ \cite[p.409]%
{GH84}. Thus $\underset{j}{\sum }\lambda _{j}=C_{1}^{n+1}$, $\underset{i<j}{%
\sum }\lambda _{i}\lambda _{j}=C_{2}^{n+1}$,..., and $\lambda _{1}...\lambda
_{n}=C_{n}^{n+1}$.\newline

\noindent \hspace*{12pt} For the Chern character form of $\odot ^{q}T$: 
\begin{align*}
\pi ^{\ast }(\mbox{ch}(\odot ^{q}T))& =\mbox{ch}(\pi ^{\ast }(\odot ^{q}T))=%
\mbox{ch}{\Large (}\underset{\substack{k_{1}+...+k_{n}=q,\\ k_{i}\in \mathbb{%
N}\cup\{0\}}}{\oplus }{\large ((L_{1})^{k_{1}}\otimes ...\otimes
(L_{n})^{k_{n}})}{\Large )} \\
& =\underset{\substack{k_{1}+...+k_{n}=q,\\ k_{i}\in \mathbb{N}\cup\{0\}}}{%
\sum }((\mbox{ch}(L_{1})^{k_{1}}\otimes ...\otimes \mbox{ch}%
(L_{n})^{k_{n}}))=\underset{\substack{k_{1}+...+k_{n}=q,\\ k_{i}\in \mathbb{N%
}\cup\{0\}}}{\sum }e^{\underset{j}{\sum }k_{j}\pi ^{\ast }(\lambda
_{j}\,\omega )} \\
& =\pi ^{\ast }{\Large (}\underset{\substack{k_{1}+...+k_{n}=q,\\ k_{i}\in 
\mathbb{N}\cup\{0\}}}{\sum }e^{\underset{j}{\sum }(k_{j}\lambda
_{j})\,\omega }{\Large )},
\end{align*}

\noindent giving ($\pi ^{\ast }$ being injective)%
\begin{equation}
\mbox{ch}(\odot ^{q}T)=\underset{\substack{k_{1}+...+k_{n}=q,\\ k_{i}\in 
\mathbb{N}\cup\{0\}}}{\sum }\,e^{(\underset{j}{\sum }k_{j}\lambda
_{j})\,\omega }.\hspace*{100pt}  \label{6-7}
\end{equation}%
\ 

\begin{proof}
(of \textbf{Theorem \ref{TeoE}}) First we claim that%
\begin{equation}
\int_{\mathbb{P}^{n}}\mbox{td}(T)\cdot e^{z\,\omega }=%
\begin{pmatrix}
z+n \\ 
n%
\end{pmatrix}%
=\frac{(z+n)(z+n-1)..(z+1)}{n(n-1)...1},\hspace*{20pt}\forall z\in \mathbb{C}
\label{6-8}
\end{equation}%
where $%
\begin{pmatrix}
a \\ 
n%
\end{pmatrix}%
$ is the binomial coefficient for $a\in \mathbb{C}$, $n\in \mathbb{N}$. Both
sides of (\ref{6-8}) are polynomials in $z$ of degree $n$: $P_{1}(z)=\int_{%
\mathbb{P}^{n}}\mbox{td}(T)\cdot e^{z\,\omega }=\int_{\mathbb{P}^{n}}%
\mbox{td}(T)\cdot (1+z\omega +...+z^{n}\omega ^{n})$, $P_{2}(z)=%
\begin{pmatrix}
z+n \\ 
n%
\end{pmatrix}%
$. For $z=k,$ this is the $L=\mathcal{O}(k)$ case in Lemma \ref{Lkl0h0}
above; that is $P_{1}(k)=P_{2}(k),\hspace*{6pt}\forall k\in \mathbb{N}$.
This implies $P_{1}(z)=P_{2}(z)$ $(\forall z\in \mathbb{C)}$, proving (\ref%
{6-8}). Using (\ref{6-7}) for $L\otimes (\odot ^{q}T),$ $B=\deg (L)$ and
setting $y_{k_{1},\cdot \cdot \cdot ,k_{n}}$ $:=$ $\underset{1\leq j\leq n}{%
\sum }k_{j}\lambda _{j},$ $k_{j}\in \mathbb{N}\cup \{0\},$ we have%
\begin{equation}
\mbox{ch}(L\otimes (\odot ^{q}T))=\underset{\substack{k_{1}+...+k_{n}=q,\\
k_{i}\in \mathbb{N}\cup\{0\}}}{\sum }\,e^{(B+y_{k_{1},\cdot \cdot \cdot
,k_{n}})\,\omega }.\hspace*{100pt}  \label{6-12}
\end{equation}%
\noindent Now (\ref{E-1}) of Theorem \ref{TeoE} follows by (\ref{6-12}) and
using (\ref{6-8}) for evaluation.
\end{proof}

\begin{example}
\label{E-6-1} Let us verify the dimension formula (\ref{E-1}) in two
different ways. We first do it using explicit numerical data. This method
cannot be easily generalizable to higher-dimensional cases. For $n=2$ and $T$
$(=$ holomorphic tangent bundle of $\mathbb{P}^{2}),$ via \cite[p.56]{Ful84}
and \cite[p.409]{GH84} 
\begin{eqnarray}
&&\mbox{td}(T)  \label{6-13} \\
&=&1+\frac{1}{2}c_{1}(T)+\frac{1}{12}(c_{1}(T)^{2}+c_{2}(T))\text{ }  \notag
\\
&=&1+\frac{3}{2}\omega +\omega ^{2}\text{ on }\mathbb{P}^{2}  \notag
\end{eqnarray}%
\noindent where $c_{1}(T)=3\omega $ and $c_{2}(T)=3\omega ^{2}$ ($\omega
=c_{1}(\mathcal{O}(1))).$ For computational purpose we write $T$
\textquotedblleft $=$" $L_{1}\oplus L_{2}$ where $L_{i}$ ($i=1,2)$ are
(hypothetical) line bundles. Then%
\begin{eqnarray}
\mbox{ch}(\odot ^{q}T) &=&\sum_{l=0}^{q}e^{(q-l)c_{1}(L_{1})+lc_{1}(L_{2})}
\label{6-13a} \\
&=&\sum_{l=0}^{q}(1+(q-l)c_{1}(L_{1})+lc_{1}(L_{2})+\frac{1}{2!}%
[(q-l)c_{1}(L_{1})+lc_{1}(L_{2})]^{2})  \notag \\
&=&\cdot \cdot \cdot =(q+1)+\frac{q(q+1)}{2}(c_{1}(L_{1})+c_{1}(L_{2}))+ 
\notag \\
&&+\frac{q(q+1)(2q+1)}{12}((c_{1}(L_{1})+c_{1}(L_{2}))^{2}-\frac{q(q+1)(q+2)%
}{6}c_{1}(L_{1})c_{1}(L_{2})  \notag \\
&=&(q+1)+\frac{3q(q+1)}{2}\omega +\frac{q(q+1)(4q-1)}{4}\omega ^{2}  \notag
\end{eqnarray}%
\noindent where we have used $c_{1}(L_{1})+c_{1}(L_{2})$ $=$ $c_{1}(T)$ $=$ $%
3\omega $ and $c_{1}(L_{1})c_{1}(L_{2})$ $=$ $c_{2}(T)$ $=$ $3\omega ^{2}.$ 
With $\mbox{ch}(L)=e^{c_{1}(L)}=e^{B\omega }=1+B\omega +\frac{1}{2}%
B^{2}\omega ^{2},$ (\ref{6-13}), (\ref{6-13a}) and $\mbox{ch}(L\otimes \odot
^{q}T)$ $=$ $\mbox{ch}(L)\mbox{ch}(\odot ^{q}T)$ we get%
\begin{eqnarray}
h^{0}(\mathbb{P}^{2},L\otimes (\odot ^{q}T)) &=&\int_{\mathbb{P}^{2}}%
\mbox{td}(T)\mbox{ch}(L\otimes \odot ^{q}T)  \label{6-13b} \\
&=&\frac{q+1}{2}(B^{2}+3(q+1)B+2(q+1)^{2}).  \notag
\end{eqnarray}

The RHS of Theorem \ref{TeoE} for $n=2$ is going to coincide with (\ref%
{6-13b})$:$ By Theorem \ref{TeoE} we compute%
\begin{eqnarray}
&&\sum_{k_{1}+k_{2}=q}\binom{k_{1}\lambda _{1}+k_{2}\lambda _{2}+2+B}{2}
\label{6-15} \\
&=&\frac{1}{2}\sum_{k_{1}=0}^{q}[(k_{1}(\lambda _{1}-\lambda _{2})+q\lambda
_{2})^{2}+(3+2B)(k_{1}(\lambda _{1}-\lambda _{2})+q\lambda _{2})  \notag \\
&&+(2+B)(1+B)]  \notag
\end{eqnarray}%
\noindent Using $\sum_{k_{1}=0}^{q}k_{1}^{2}$ $=$ $q(q+1)(2q+1)/6,$ $%
\sum_{k_{1}=0}^{q}k_{1}$ $=$ $q(q+1)/2$ and $\lambda _{1}$ $=$ $\frac{3-i%
\sqrt{3}}{2}$, $\lambda _{2}$ $=$ $\frac{3+i\sqrt{3}}{2}$ we can express the
RHS of (\ref{6-15}) in terms of $q$ and $B$ (this is straightforward
although slightly tedious)$.$ The final answer turns out to be the same as (%
\ref{6-13b}).
\end{example}

\section{The Grassmannian case\label{Sec7}}

As mentioned in the Introduction, we anticipate that the basic principles of
this paper can work over other K\"{a}hler manifolds such as Hermitian
symmetric spaces of compact type. The complete treatment requires a
systematic formulation based on the general theory of Hermitian symmetric
spaces, which goes beyond the scope of this paper. Let us be content with
indicating a few points and using the Grassmannian as an illustrative
example. Our main reference is the monograph by Ngaiming Mok \cite{Mok89}.

The discussion below is focused on Grassmannians. Many of our operations
work on K\"{a}hler manifolds. In particular, one wishes the Bochner-Kodaira
type identities could hold true on the Grassmannians $G(\mu ,\nu )$. A
principal one of such identities is Proposition \ref{Prop 4.2} or
Proposition \ref{Prop 4.3}. With this tool, our main result (Theorem \ref%
{T-7-1}) gives a complete analogue for the second lowest eigenvalue case.
See Remark \ref{R-7-2} for higher-eigenvalue cases.

However, there are substantial differences between Grassmannians and
projective spaces. Some key features are the following. For any tangent
vector $X$ of $G(\mu ,\nu )$ write $\mathcal{N}_{X}$ for the null-space of
the Hermitian bilinear form $H_{X}(V,W):=R_{X\overline{X}V\overline{W}}$ at $%
0$ \cite[p.84]{Mok89}. These null-spaces $\mathcal{N}_{X}$ are in general
nontrivial on $G(\mu ,\nu ),$ and play a very important role in various
rigidity theorems developed by Ngaiming Mok. The Bochner-Kodaira type
identities cannot survive intact. More precisely, let us first define linear
operators $\mathcal{E}_{(q+1)}$ and $\mathcal{F}_{(q+1)}:$ $\Omega
^{0,0}(G(\mu ,\nu ),L\otimes (\odot ^{q+1}T^{\ast }))$ $\rightarrow $ $%
\Omega ^{0,0}(G(\mu ,\nu ),L\otimes (\odot ^{q}T^{\ast }))$ for $q\in 
\mathbb{N}$ by (the double-index notation in $\mathcal{F}_{(q+1)}$ below
being referred to lines after (\ref{7-1-3}))\ 
\begin{eqnarray}
&&\mathcal{E}_{(q+1)}(\sum_{\alpha _{1}...\alpha _{q+1}}f_{\alpha _{1}\alpha
_{2}...\alpha _{q+1}}s\otimes (\overset{q+1}{\underset{k=1}{\otimes }}%
dz^{\alpha _{k}}))  \label{7-z} \\
{:=} &&\sum_{\alpha _{1}...\alpha _{q}}\sum_{j=1}^{q}\sum_{\alpha _{q+1}\in 
\mathcal{N}_{\alpha _{j}}}(f_{\alpha _{1}\alpha _{2}...\alpha _{q+1}})_{,%
\overline{\alpha _{q+1}}}\}s\otimes (\overset{q}{\underset{k=1}{\otimes }}%
dz^{\alpha _{k}})  \notag
\end{eqnarray}

\begin{eqnarray}
&&\mathcal{F}_{(q+1)}(\sum_{\alpha _{1}...\alpha _{q+1}}f_{\alpha _{1}\alpha
_{2}...\alpha _{q+1}}s\otimes (\overset{q+1}{\underset{k=1}{\otimes }}%
dz^{\alpha _{k}}))  \label{7-y} \\
{:=} &&\sum_{(rt)\text{ }\alpha _{1}...\alpha _{q-1}}\sum_{(ik)\in \mathcal{N%
}_{(rt)}}f_{\alpha _{1}\alpha _{2}...\alpha _{q-1}(it)(rk),\overline{ik}%
}s\otimes dz^{(rt)}\otimes (\overset{q-1}{\underset{k=1}{\otimes }}%
dz^{\alpha _{k}}).  \notag
\end{eqnarray}

\begin{proposition}
\label{P-7-1} (BK identities of Proposition \ref{Prop 4.2} revisited with
correction) For all $q\in \{0\}\cup \mathbb{N}$, we have, on $\Omega
^{0,0}(G(\mu ,\nu ),L\otimes (\odot ^{q+1}T^{\ast }))$%
\begin{equation}
\Delta ^{q}\overline{I}^{(q+1)}\overline{\partial }^{(q+1)}-\overline{I}%
^{(q+1)}\overline{\partial }^{(q+1)}\Delta ^{(q+1)}=(-B+\frac{c}{2}(2q+\mu
+\nu ))\overline{I}^{(q+1)}\overline{\partial }^{(q+1)}-c\mathcal{E}%
_{(q+1)}+cq\mathcal{F}_{(q+1)}\text{ for }q\in \mathbb{N},  \label{7-m}
\end{equation}%
\begin{equation}
{\scriptsize \Delta }^{0}{\scriptsize \overline{I}}^{1}\overline{\partial }%
^{1}{\scriptsize -\overline{I}}^{1}\overline{\partial }^{1}{\scriptsize %
\Delta }^{1}={\scriptsize (-B+}\frac{{\scriptsize c}}{{\scriptsize 2}}%
{\scriptsize (\mu +\nu )\,)\overline{I}^{1}\overline{\partial }^{1}}\text{
for }q=0.  \label{7-ma}
\end{equation}%
\noindent The constant ${\scriptsize c}$ above equals $2$ throughout this
section.
\end{proposition}

\begin{proof}
Except the computation of (\ref{eq 4.10}), the original proof of Proposition %
\ref{Prop 4.2} can be carried over to the Grassmannian$,$ which we shall now
elaborate$.$ The complete details are hardly straightforward due to the
aforementioned null-space $\mathcal{N}_{X}.$ We follow the terminology there
and divide the proof into three steps.

\textbf{Step 1: }Let us briefly review the case of projective spaces. For $%
R_{\overline{l}\alpha _{1}\overline{\alpha _{q+1}}i}$ in the first line of (%
\ref{eq 4.10}), we have Case $1):$ $l=\alpha _{1}$ and then $i$ must be $%
\alpha _{q+1}.$ That is to say,%
\begin{equation}
(\text{with }l=\alpha _{1})\text{ }\alpha _{q+1}\neq i\text{ }\implies \text{
}R_{\overline{l}\alpha _{1}\overline{\alpha _{q+1}}i}=0.  \label{7-1-1}
\end{equation}

\noindent Lemma \ref{L-7-1} below shows that (\ref{7-1-1}) holds true also
on $G(\mu ,\nu )$. For Case $2):$ $l\neq \alpha _{1}$ we used the vanishing
results for (\ref{eq 4.10}):%
\begin{equation}
(\text{with }l\neq \alpha _{1})\text{ }l\neq i\text{ and }\alpha
_{q+1}=\alpha _{1}\text{ }\implies \text{ }R_{\overline{l}\alpha _{1}%
\overline{\alpha _{q+1}}i}=0,  \label{7-1-2}
\end{equation}%
\begin{equation}
(\text{with }l\neq \alpha _{1})\text{ }l=i\text{ and }\alpha _{q+1}\neq
\alpha _{1}\text{ }\implies \text{ }R_{\overline{l}\alpha _{1}\overline{%
\alpha _{q+1}}i}=0.  \label{7-1-3}
\end{equation}

\noindent If $R_{I\overline{J}K\overline{L}}$ or $R_{\overline{I}J\overline{K%
}L}$ denotes the curvature of $G(\mu ,\nu ),$ where $I=(ii^{\prime })$ ($%
1\leq i\leq \mu ,$ $1\leq i^{\prime }\leq \nu ),$ $J=(jj^{\prime }),\cdots ,$
by identification of tangent vectors with complex $\mu \times \nu $ matrices
(see \cite[p.84]{Mok89}), then Lemma \ref{L-7-1} below also verifies these
vanishing results on $G(\mu ,\nu ).$ The step 1 is finished.
\end{proof}

\begin{lemma}
\label{L-7-1} As analogous to (\ref{7-1-1})-(\ref{7-1-3}) used in the case
of projective spaces, we have on $G(\mu ,\nu ):$%
\begin{equation}
I=J\text{ and }K\neq L\text{ }\implies \text{ }R_{\overline{I}J\overline{K}%
L}=0,  \label{7-1-5}
\end{equation}%
\begin{equation}
I\neq J\text{, }I\neq L\text{ and }K=J\text{ }\implies \text{ }R_{\overline{I%
}J\overline{K}L}=0,  \label{7-1-6}
\end{equation}%
\begin{equation}
I\neq J\text{, }I=L\text{ and }K\neq J\text{ }\implies \text{ }R_{\overline{I%
}J\overline{K}L}=0.  \label{7-1-7}
\end{equation}%
%
%
%
%
%
%
%
%
\end{lemma}

\begin{proof}
The curvature of $G(\mu ,\nu )$ is the following (see \cite[p.84]{Mok89}
where the dual pair of $D_{\nu ,\mu }^{I}$ and $G(\mu ,\nu )$ \cite[p.79]%
{Mok89} are of curvature opposite in sign\footnote{%
Via the Borel embedding theorem \cite[p.51]{Mok89}, $D_{\nu ,\mu }^{I}$ is
naturally and equidimensionally embedded in $G(\mu ,\nu )$ so that the
tangent vectors (at the origin) of both spaces can be thus identified.} \cite%
[proof of Proposition 2 in p.46]{Mok89}): 
\begin{equation}
\overline{R_{\overline{I}J\overline{K}L}}=R_{I\overline{J}K\overline{L}%
}=R_{ii^{\prime },\overline{jj^{\prime }},kk^{\prime },\overline{ll^{\prime }%
}}=\delta _{ij}\delta _{kl}\delta _{i^{\prime }l^{\prime }}\delta
_{j^{\prime }k^{\prime }}+\delta _{il}\delta _{kj}\delta _{i^{\prime
}j^{\prime }}\delta _{k^{\prime }l^{\prime }},\text{ }R_{\overline{I}J%
\overline{K}L}=R_{\overline{K}J\overline{I}L}.  \label{7-0}
\end{equation}%
\noindent By using (\ref{7-0}), the verification of (\ref{7-1-5}) is
straightforward. Observe that $R_{\overline{K}J\overline{\overline{I}}L}=R_{%
\overline{I}J\overline{K}L}$ (resp. $R_{\overline{I}L\overline{K}J}=R_{%
\overline{I}J\overline{K}L})$ and using (\ref{7-1-5}) implies (\ref{7-1-6})
(resp. (\ref{7-1-7})).
\end{proof}

\begin{remark}
\label{R-7-z} The first identity of (\ref{4-7-1}) does not hold on a general
Grassmannian. For instance, for $I=(11),$ $J=(12),$ $K=(22)$ and $L=(21)$ we
have $R_{I\overline{J}K\overline{L}}=1$ on $G(2,2)$ while the second
identity of (\ref{4-7-1}) still holds on $G(\mu ,\nu ):$%
\begin{equation}
J\neq I\text{ or }L\neq I\text{ }\implies \text{ }R_{\overline{I}J\overline{I%
}L}=0.  \label{7-1-8}
\end{equation}%
By the proof below working for a general Grassmannian, there is a second
proof for the case of projective spaces without using the first identity of (%
\ref{4-7-1}).
\end{remark}

\begin{proof}
(of Proposition \ref{P-7-1} continued) \textbf{Step 2} ($c\mathcal{E}_{(q+1)}
$\textit{\ term of (\ref{7-z})})\textbf{:} We first work for $q\geq 1$ on $%
G(\mu ,\nu ).$ We closely follow (\ref{eq 4.10}) and the terminology after
it. Similar to the first line of (\ref{eq 4.10}), we have the following
cases. For Case $1)$ $l=\alpha _{1}$ and then $i=\alpha _{q+1}$ by the
vanishing (\ref{7-1-5})$,$ we have two subcases Case $1)_{a}$ $\alpha
_{1}\neq \alpha _{q+1}$ $\in $ $\mathcal{N}_{\alpha _{1}}^{\bot }$ (then $R_{%
\overline{l}\alpha _{1}\overline{\alpha _{q+1}}i}$ $=$ $R_{\overline{\alpha
_{1}}\alpha _{1}\overline{\alpha _{q+1}}\alpha _{q+1}}$ $=$ $\frac{c}{2})$
and Case $1)_{b}$ $\alpha _{1}=\alpha _{q+1}$ (then $R_{\overline{l}\alpha
_{1}\overline{\alpha _{q+1}}i}$ $=$ $R_{\overline{\alpha _{1}}\alpha _{1}%
\overline{\alpha _{1}}\alpha _{1}}$ $=$ $c).$ Note $R_{\overline{\alpha _{1}}%
\alpha _{1}\overline{\alpha _{q+1}}\alpha _{q+1}}=0$ for $\alpha _{q+1}\in 
\mathcal{N}_{\alpha _{1}}.$ If not Case $1)$, which means $l\neq \alpha _{1}$%
, then we distinguish Case $2)$ $i=l$ from Case $3)$ $i\neq l.$ In the case
of projective spaces, Case $3)$ cannot occur.

We first discuss Case $2)$ $l\neq \alpha _{1}$ and $i=l,$ and leave Case $3)$
to Step 3$.$ Then by the vanishing \ 

(\ref{7-1-7}) we also have $\alpha _{1}=\alpha _{q+1},$ then $R_{\overline{l}%
\alpha _{1}\overline{\alpha _{1}}l}$ $=$ $\frac{c}{2}$ for $l(=i)\in 
\mathcal{N}_{\alpha _{1}}^{\perp }\backslash \{\alpha _{1}\}.$ We combine
Case $2)$ and \textit{half}-Case $1)_{b}$ (so that $i$ $\in $ $\mathcal{N}%
_{\alpha _{1}}^{\perp }$ with $i=\alpha _{1}$ allowed$)$ to get (cf. the
second term of the second line of (\ref{4-15}))%
\begin{equation}
\frac{c}{2}\sum_{i\in \mathcal{N}_{\alpha _{1}}^{\bot },\alpha _{2}...\alpha
_{q}\alpha _{1}}(f_{i\alpha _{2}...\alpha _{q}\alpha _{1}})_{,\overline{i}%
}s\otimes (\overset{q}{\underset{k=1}{\otimes }}dz^{\alpha _{k}})_{\mid
_{z=0}}  \label{7-2-0}
\end{equation}

\noindent and similarly combine Case $1)_{a}$ and \textit{half}-Case $1)_{b}$
to give (cf. the first term of the second line of (\ref{4-15}))%
\begin{equation}
\frac{c}{2}\sum_{\alpha _{q+1}\in \mathcal{N}_{\alpha _{1}}^{\bot },\alpha
_{1}...\alpha _{q}}(f_{\alpha _{1}\alpha _{2}...\alpha _{q+1}})_{,\overline{%
\alpha _{q+1}}}s\otimes (\overset{q}{\underset{k=1}{\otimes }}dz^{\alpha
_{k}})_{\mid _{z=0}}.  \label{7-2}
\end{equation}

\noindent Note that if $\mathcal{N}_{\alpha _{1}}=\{0\},$ then $\mathcal{N}%
_{\alpha _{1}}^{\bot }$ contains everything and the above notation $i\in 
\mathcal{N}_{\alpha _{1}}^{\bot }$ recovers the notation (see (\ref{4-15})
for $i$-index) for projective spaces. Altogether, we obtain an analogue of (%
\ref{4-15}), for \textit{symmetric tensors }$f_{\bullet }$ 
\begin{equation}
c\sum_{\alpha _{q+1}\in \mathcal{N}_{\alpha _{1}}^{\bot },\alpha
_{1}...\alpha _{q}}(f_{\alpha _{1}\alpha _{2}...\alpha _{q+1}})_{,\overline{%
\alpha _{q+1}}}s\otimes (\overset{q}{\underset{k=1}{\otimes }}dz^{\alpha
_{k}})_{\mid _{z=0}},  \label{7-2-1}
\end{equation}

\noindent in which the notation $\alpha _{q+1}\in \mathcal{N}_{\alpha
_{1}}^{\bot }$ when $\alpha _{q+1}=\alpha _{1}$ is understood as giving no
restriction on $\alpha _{q+1}$ (because any tangent vector $v$ lies in $%
\mathcal{N}_{v}^{\bot }$ always). (\ref{7-2-1}) gives the first line of (\ref%
{eq 4.10}) (modulo the contribution from Case $3)$ to be discussed in Step 3
below).

F\textbf{or the analogous sum of the first }$q$\textbf{\ lines }of (\ref{eq
4.10}), replacing $\alpha _{q+1}\in \mathcal{N}_{\alpha _{1}}^{\bot }$ in (%
\ref{7-2-1}) by $\alpha _{q+1}\in \mathcal{N}_{\alpha _{j}}^{\bot }$ and
summing over $j=1,\cdots ,q$ ($j=q+1$ giving the last line of (\ref{eq 4.10}%
) will be discussed soon), we have then\footnote{%
In fact, the double sums in each of the first $q$ lines of (\ref{eq 4.10})
are the same by $f_{\bullet }$ $\in \odot ^{q+1}T^{\ast }$ and the
notational change of dummy variables.}%
\begin{equation}
c\sum_{j=1}^{q}\sum_{\alpha _{q+1}\in \mathcal{N}_{\alpha _{j}}^{\bot
},\alpha _{1}...\alpha _{q}}(f_{\alpha _{1}\alpha _{2}...\alpha _{q+1}})_{,%
\overline{\alpha _{q+1}}}s\otimes (\overset{q}{\underset{k=1}{\otimes }}%
dz^{\alpha _{k}})_{\mid _{z=0}}.  \label{7-2-2}
\end{equation}

\noindent Thinking $\alpha _{q+1}\in \mathcal{N}_{\alpha _{j}}^{\bot }$ as
\textquotedblleft the full range of $\alpha _{q+1}$ throwing away those in $%
\mathcal{N}_{\alpha _{j}}$", we rewrite $(\ref{7-2-2})$ as ($f_{\alpha
_{1}\alpha _{2}..\alpha _{q+1}}$ being symmetric)%
\begin{eqnarray}
(\ref{7-2-2}) &=&\{qc\sum_{\alpha _{q+1}\alpha _{1}...\alpha _{q}}(f_{\alpha
_{1}\alpha _{2}...\alpha _{q+1}})_{,\overline{\alpha _{q+1}}}  \label{7-2-a}
\\
&&-c\sum_{\alpha _{1}...\alpha _{q}}\sum_{j=1}^{q}\sum_{\alpha _{q+1}\in 
\mathcal{N}_{\alpha _{j}}}(f_{\alpha _{1}\alpha _{2}...\alpha _{q+1}})_{,%
\overline{\alpha _{q+1}}}\}s\otimes (\overset{q}{\underset{k=1}{\otimes }}%
dz^{\alpha _{k}})_{\mid _{z=0}}.  \notag
\end{eqnarray}

To get an analogue of (\ref{4-18}) \textbf{for the last line} of (\ref{eq
4.10}), we first recall that $l=i$ \ As remarked above (\ref{4-18}), we
similarly separate into the two cases for $i$: $i$ equals $\alpha _{q+1}$
and the analogue of that last line is $c$ $\sum_{\alpha _{1}...\alpha
_{q+1}}(f_{\alpha _{1}\alpha _{2}...\alpha _{q+1}})_{,\overline{\alpha _{q+1}%
}}s\otimes (\overset{q}{\underset{k=1}{\otimes }}dz^{\alpha _{k}})_{\mid
_{z=0}})$ (where $R_{\overline{l}\alpha _{q+1}\overline{\alpha _{q+1}}i}=c$
by $l=i=\alpha _{q+1})$; $i$ does not equal $\alpha _{q+1}$ and the analogue
is $\frac{c}{2}$ $\sum_{\alpha _{q+1}\alpha _{1}...\alpha _{q}}\sum_{i\in 
\mathcal{N}_{\alpha _{q+1}}^{\perp }\backslash \{\alpha _{q+1}\}}(f_{\alpha
_{1}\alpha _{2}...\alpha _{q}i})_{,\overline{i}}$ $s\otimes (\overset{q}{%
\underset{k=1}{\otimes }}dz^{\alpha _{k}})_{\mid _{z=0}}$ (where $R_{%
\overline{l}\alpha _{q+1}\overline{\alpha _{q+1}}i}=\frac{c}{2}$ or $0$ by $%
l=i\neq \alpha _{q+1}$). These put together yield the last line of (\ref{eq
4.10}) or an analogue of (\ref{4-18}): 
\begin{equation}
\{c\sum_{\alpha _{1}...\alpha _{q+1}}(f_{\alpha _{1}\alpha _{2}...\alpha
_{q+1}})_{,\overline{\alpha _{q+1}}}+\frac{c}{2}\sum_{\alpha _{q+1}\alpha
_{1}...\alpha _{q},}\sum_{i\in \mathcal{N}_{\alpha _{q+1}}^{\perp
}\backslash \{\alpha _{q+1}\}}(f_{\alpha _{1}\alpha _{2}...\alpha _{q}i})_{,%
\overline{i}}\}s\otimes (\overset{q}{\underset{k=1}{\otimes }}dz^{\alpha
_{k}})_{\mid _{z=0}}.  \label{7-14a}
\end{equation}

For the Grassmannian, we rewrite (\ref{7-14a}) in the following way (\ref%
{7-2-b}). In view of the main index structure $\sum_{i\in \mathcal{N}%
_{\alpha _{q+1}}^{\perp }\backslash \{\alpha _{q+1}\}}(\cdot )$ $=$ $%
\sum_{i\in \mathcal{N}_{\alpha _{q+1}}^{\perp }}(\cdot )$ $-$ $%
\sum_{i=\alpha _{q+1}}(\cdot )$ half of the first term of (\ref{7-14a}) just
cancels $\frac{c}{2}\sum_{\alpha _{q+1}\alpha _{1}...\alpha
_{q}}\sum_{i=\alpha _{q+1}}(\cdot ),$ resulting in the following 
\begin{equation}
(\ref{7-14a})=\frac{c}{2}\sum_{\alpha _{1}...\alpha _{q}\alpha
_{q+1}}(f_{\alpha _{1}\alpha _{2}...\alpha _{q+1}})_{,\overline{\alpha _{q+1}%
}}+\frac{c}{2}\sum_{\alpha _{q+1}}\sum_{i\in \mathcal{N}_{\alpha
_{q+1}}^{\perp }\alpha _{1}...\alpha _{q}}(f_{\alpha _{1}\alpha
_{2}...\alpha _{q}i})_{,\overline{i}}  \label{7-2-b}
\end{equation}%
\noindent (times $s\otimes (\overset{q}{\underset{k=1}{\otimes }}dz^{\alpha
_{k}})_{\mid _{z=0}}).$ Moreover, by the same decomposition as in (\ref%
{7-2-a}) we reduce the second term in the RHS of (\ref{7-2-b}) to (by
summing over $\alpha _{q+1}$ for the first term below$)$%
\begin{equation}
\frac{c}{2}\mu \nu \sum_{\alpha _{1}...\alpha _{q}i}(f_{\alpha _{1}\alpha
_{2}...\alpha _{q}i})_{,\overline{i}}-\frac{c}{2}\sum_{\alpha
_{q+1}}\sum_{i\in \mathcal{N}_{\alpha _{q+1}}\alpha _{1}...\alpha
_{q}}(f_{\alpha _{1}\alpha _{2}...\alpha _{q}i})_{,\overline{i}}.
\label{7-2-c}
\end{equation}

\noindent Noting the index structure $\sum_{(i,\alpha _{q+1}),i\in \mathcal{N%
}_{\alpha _{q+1}}}(\cdot )=\sum_{(i,\alpha _{q+1}),\alpha _{q+1}\in \mathcal{%
N}_{i}}(\cdot )=\sum_{i}\sum_{\alpha _{q+1}\in \mathcal{N}_{i}}(\cdot ),$ we
reduce the second term of (\ref{7-2-c}) (by fixing an $i$ in $%
\sum_{i}\sum_{\alpha _{q+1}\in \mathcal{N}_{i}}(\cdot )$ and summing over $%
\alpha _{q+1}\in \mathcal{N}_{i}$ using $\dim \mathcal{N}_{i}=(\mu -1)(\nu
-1))$ to%
\begin{equation}
\sum_{\alpha _{q+1}}\sum_{i\in \mathcal{N}_{\alpha _{q+1}},\alpha
_{1}...\alpha _{q}}(f_{\alpha _{1}\alpha _{2}...\alpha _{q}i})_{,\overline{i}%
}=(\mu -1)(\nu -1)\sum_{\alpha _{1}...\alpha _{q}i}(f_{\alpha _{1}\alpha
_{2}...i})_{,\overline{i}}  \label{7-2-c2}
\end{equation}

\noindent and then rewrite (\ref{7-2-c}) as (the original dummy variables $i$
are denoted by $\alpha _{q+1}$ below)%
\begin{equation}
(\ref{7-2-c})=\frac{c}{2}\mu \nu \sum_{\alpha _{1}...\alpha _{q}\alpha
_{q+1}}(f_{\alpha _{1}\alpha _{2}...\alpha _{q}\alpha _{q+1}})_{,\overline{%
\alpha _{q+1}}}-\frac{c}{2}(\mu -1)(\nu -1)\sum_{\alpha _{1}...\alpha
_{q}\alpha _{q+1}}(f_{\alpha _{1}\alpha _{2}...\alpha _{q+1}})_{,\overline{%
\alpha _{q+1}}}.  \label{7-2-c3}
\end{equation}%
\noindent Recalling (\ref{7-2-c3}) as the second term of the RHS of (\ref%
{7-2-b}), we rewrite%
\begin{equation}
(\ref{7-2-b})=\frac{c}{2}\sum_{\alpha _{1}...\alpha _{q}\alpha
_{q+1}}(f_{\alpha _{1}\alpha _{2}...\alpha _{q+1}})_{,\overline{\alpha _{q+1}%
}}+(\ref{7-2-c3}),  \label{7-21a}
\end{equation}%
\noindent giving an analogue of the last line of (\ref{eq 4.10}). The
analogue of summing over the first $q$ lines and over the last line of (\ref%
{eq 4.10}) is, as aforementioned,%
\begin{eqnarray}
(\ref{7-2-a})+(\ref{7-21a}) &=&\{[qc+\frac{c}{2}+\frac{c}{2}\mu \nu -\frac{c%
}{2}(\mu -1)(\nu -1)]\sum_{\alpha _{1}...\alpha _{q}\alpha _{q+1}}(f_{\alpha
_{1}\alpha _{2}...\alpha _{q+1}})_{,\overline{\alpha _{q+1}}}  \label{7-2-d}
\\
&&-c\sum_{\alpha _{1}...\alpha _{q}}\sum_{j=1}^{q}\sum_{\alpha _{q+1}\in 
\mathcal{N}_{\alpha _{j}}}(f_{\alpha _{1}\alpha _{2}...\alpha _{q+1}})_{,%
\overline{\alpha _{q+1}}}\}s\otimes (\overset{q}{\underset{k=1}{\otimes }}%
dz^{\alpha _{k}})_{\mid _{z=0}}.  \notag
\end{eqnarray}

\noindent Combining (\ref{eq 4.9}) and (\ref{7-2-d}), modulo the
contribution from Case $3)$ below$,$ gives (\ref{7-m}), a modified version
of (\ref{4-12}) with the modification being the triple summation term of (%
\ref{7-2-d}), which is precisely the term $c\mathcal{E}_{(q+1)}(\cdot )$ of (%
\ref{7-z}).

\textbf{Step 3 }($cq\mathcal{F}_{(q+1)}$\textit{\ term of (\ref{7-z})})%
\textbf{: }We now turn to Case $3)$ $l\neq \alpha _{1}$ and $i\neq l.$ It
follows that 
\begin{equation}
\alpha _{q+1}\neq \alpha _{1}  \label{7-3-z}
\end{equation}%
\noindent (otherwise $R_{\overline{l}\alpha _{1}\overline{\alpha _{q+1}}%
i}=R_{\overline{l}\alpha _{1}\overline{\alpha _{1}}i}=0$ by (\ref{7-1-6}) in
the first line of (\ref{eq 4.10}))$.$ For the curvature contribution of Case 
$3)$ to the first line of (\ref{eq 4.10}), let $\alpha _{1}=(rr^{\prime }),$ 
$1\leq r\leq \mu ,$ $1\leq r^{\prime }\leq \nu ,$ be given. Write $%
l=(jj^{\prime }).$ The assumption $l\neq \alpha _{1}$ means Case $3)_{a}$ $%
j\neq r$ or Case $3)_{b}$ $j^{\prime }\neq r^{\prime };$ each contribution
is exclusive: If $j\neq r$ and $j^{\prime }\neq r^{\prime },$ then the
curvature (writing $\alpha _{q+1}=(kk^{\prime }),$ $i=(mm^{\prime }))$ 
\begin{equation}
R_{\overline{l}\alpha _{1}\overline{\alpha _{q+1}}i}=R_{\overline{jj^{\prime
}},rr^{\prime },\overline{kk^{\prime }},mm^{\prime }}\overset{(\ref{7-0})}{=}%
\delta _{jr}\delta _{km}\delta _{j^{\prime }m^{\prime }}\delta _{r^{\prime
}k^{\prime }}+\delta _{jm}\delta _{kr}\delta _{j^{\prime }r^{\prime }}\delta
_{k^{\prime }m^{\prime }}  \label{7-3-a}
\end{equation}

\noindent vanishes as $\delta _{jr}=\delta _{j^{\prime }r^{\prime }}=0.$
Assuming Case $3)_{a}$ (resp. Case $3)_{b})$ now, we have then $j^{\prime
}=r^{\prime }$ (resp. $j=r$) (otherwise the curvature vanishes)$.$ So the
curvature from (\ref{7-3-a}) is 
\begin{equation}
R_{\overline{jj^{\prime }},rr^{\prime },\overline{kk^{\prime }},mm^{\prime
}}=\delta _{jm}\delta _{kr}\delta _{k^{\prime }m^{\prime }}\text{ (resp. }R_{%
\overline{jj^{\prime }},rr^{\prime },\overline{kk^{\prime }},mm^{\prime
}}=\delta _{km}\delta _{j^{\prime }m^{\prime }}\delta _{r^{\prime }k^{\prime
}}).  \label{7-3-b}
\end{equation}%
\noindent The condition ($\ref{7-3-z})$ means $k\neq r$ or $k^{\prime }\neq
r^{\prime }.$ For Case $3)_{a}$ (resp. Case $3)_{b}),$ if $k\neq r$ (resp. $%
k^{\prime }\neq r^{\prime }),$ then the curvature vanishes by (\ref{7-3-b})
(respectively). So we take $k=r$ (resp. $k^{\prime }=r^{\prime }),$ and then 
$k^{\prime }\neq r^{\prime }$ (resp. $k\neq r)$ must hold$.$ The curvature (%
\ref{7-3-b}) vanishes unless%
\begin{equation}
m=j\text{ and }m^{\prime }=k^{\prime }\ (\text{resp}.\ m=k\text{ and }%
m^{\prime }=j^{\prime }).  \label{7-27-a}
\end{equation}%
\noindent Note that this satisfies the assumption ($jj^{\prime })=l$ $\neq $ 
$i=(mm^{\prime })$ as $j^{\prime }=r^{\prime }$ $\neq $ $k^{\prime
}=m^{\prime }$ (resp. $j=r\neq k=m).$ Plugging (\ref{7-27-a}) into the
curvature in the first line of (\ref{eq 4.10}) with $l=(jj^{\prime }),$ $%
i=(mm^{\prime }),$ $\alpha _{1}=(rr^{\prime }),$ $\alpha _{q+1}=(kk^{\prime
}),$ yields 
\begin{equation}
\sum_{(rr^{\prime }),\alpha _{2}\cdots \alpha _{q}}\sum_{(jk^{\prime })\in 
\mathcal{N}_{(rr^{\prime })}}f_{(jr^{\prime })\alpha _{2}\cdots \alpha
_{q}(rk^{\prime }),\overline{jk^{\prime }}}s\otimes dz^{(rr^{\prime
})}\otimes (\overset{q}{\underset{k=2}{\otimes }}dz^{\alpha _{k}})_{\mid
_{z=0}}  \label{7-3-b1}
\end{equation}%
\begin{equation}
\text{(resp. }\sum_{(rr^{\prime }),\alpha _{2}\cdots \alpha
_{q}}\sum_{(kj^{\prime })\in \mathcal{N}_{(rr^{\prime })}}f_{(rj^{\prime
})\alpha _{2}\cdots \alpha _{q}(kr^{\prime }),\overline{kj^{\prime }}%
}s\otimes dz^{(rr^{\prime })}\otimes (\overset{q}{\underset{k=2}{\otimes }}%
dz^{\alpha _{k}})_{\mid _{z=0}}).  \label{7-3-b2}
\end{equation}%
\noindent Together the contribution from Case $3)$ to the first line of (\ref%
{eq 4.10}) for $G(\mu ,\nu )$ is, by the definition of $\mathcal{F}_{(q+1)}$
in (\ref{7-y}), for \textit{symmetric tensors }$f_{\bullet }$ (which implies 
$(\ref{7-3-b1})=(\ref{7-3-b2})$ by setting dummy variables ($jk^{\prime
})\equiv (kj^{\prime })$ and so $f_{(jr^{\prime })\alpha _{2}\cdots \alpha
_{q}(rk^{\prime })}$ $=$ $f_{(kr^{\prime })\alpha _{2}\cdots \alpha
_{q}(rj^{\prime })}$ which is $f_{(rj^{\prime })\alpha _{2}\cdots \alpha
_{q}(kr^{\prime })}$ by symmetry$)$%
\begin{equation}
c\mathcal{F}_{(q+1)}(\sum_{\alpha _{1}\cdots \alpha _{q+1}}f_{\alpha
_{1}\alpha _{2}\cdots \alpha _{q+1}}s\otimes (\overset{q+1}{\underset{k=1}{%
\otimes }}dz^{\alpha _{k}})),\text{ \ \ }c=2.  \label{7-3-c}
\end{equation}%
\noindent To the analogous sum of the first $q$\ lines\textbf{\ }of (\ref{eq
4.10}), the total contribution from Case $3)$ is $q$ times (\ref{7-3-c});
compare the footnote seated above (\ref{7-2-2}). Note that the last line of (%
\ref{eq 4.10}) receives no contribution from Case $3)$ since in this case $%
l\neq i$ thus $R_{\overline{l}\alpha _{q+1}\overline{\alpha _{q+1}}i}=0$.
Step 3 is now completed.

\textbf{Step 4: }The\textbf{\ }original proof for the projective space
contains (\ref{4-12}) and (\ref{eq 4.9}), which are valid for general K\"{a}%
hler manifolds. The above Steps $2$ and $3$ give the subsequent treatment
analogous to that of (\ref{eq 4.10}). From Step 2 we get (\ref{7-m}) without
the last term on the RHS; Step 3 shows that this last term equals $qc%
\mathcal{F}_{(q+1)}.$ Together, we obtain the complete formula of (\ref{7-m}%
).

We turn now to $q=0$ for (\ref{7-ma}) of this proposition$.$ For the sake of
clarity, let us treat it independently. Starting from (\ref{4-12}) (for $%
q=0),$ we have 
\begin{eqnarray}
&&({\scriptsize \Delta }^{0}{\scriptsize \overline{I}}^{1}\overline{\partial 
}^{1}{\scriptsize -\overline{I}}^{1}\overline{\partial }^{1}{\scriptsize %
\Delta }^{1})(\sum_{\alpha _{1}}f_{\alpha _{1}}s\otimes dz^{\alpha _{1}})
\label{7-q-1} \\
&=&-\sum_{\alpha _{1},i}(f_{\alpha _{1},\overline{\alpha _{1}}i}-f_{\alpha
_{1},i\overline{\alpha _{1}}})_{,\overline{i}}s+\sum_{\alpha
_{1},i}f_{\alpha _{1},\overline{i}}s_{i\overline{\alpha _{1}}}  \notag \\
&&\overset{(\ref{f Ricci 1})}{=}\sum_{\alpha _{1},i,l}(f_{l}R_{\overline{l}%
\alpha _{1}\overline{\alpha _{1}}i})_{,\overline{i}}s-B\sum_{\alpha
_{1}}f_{\alpha _{1},\overline{\alpha _{1}}}s.  \notag
\end{eqnarray}

\noindent Note that $R_{\overline{l}\alpha _{1}\overline{\alpha }_{1}i}%
\overset{(\ref{7-0})}{=}R_{\overline{\alpha }_{1}\alpha _{1}\overline{l}%
i}=R_{\overline{\alpha }_{1}\alpha _{1}\overline{i}i}\neq 0$ \ for $l=i\in 
\mathcal{N}_{\alpha _{1}}^{\bot }$ while $R_{\overline{\alpha }_{1}\alpha
_{1}\overline{l}i}=0$ if $l\neq i$ by (\ref{7-1-5}). This yields (using the
division into\textbf{\ }$\mathbf{i=\alpha }_{1}$ and $\mathbf{i\in N}%
_{\alpha _{1}}^{\bot }\mathbf{\backslash \{\alpha }_{1}\mathbf{\}}$ for the
second equality below)%
\begin{eqnarray}
\sum_{i,\alpha _{1},l}(f_{l}R_{\overline{l}\alpha _{1}\overline{\alpha }%
_{1}i})_{,\overline{i}}s &=&\sum_{i,\alpha _{1}}(f_{i}R_{\overline{\alpha }%
_{1}\alpha _{1}\overline{i}i})_{,\overline{i}}s  \label{7-q-2} \\
&=&c\sum_{\alpha _{1}}f_{\alpha _{1},\overline{\alpha _{1}}}s+\frac{c}{2}%
\sum_{\alpha _{1}}\sum_{i\in \mathcal{N}_{\alpha _{1}}^{\bot }\backslash
\{\alpha _{1}\}}f_{i,\overline{i}}s.  \notag
\end{eqnarray}

\noindent To continue with (\ref{7-q-2}), we have%
\begin{equation}
\sum_{\alpha _{1}}\sum_{i\in \mathcal{N}_{\alpha _{1}}^{\bot }\backslash
\{\alpha _{1}\}}f_{i,\overline{i}}s=\sum_{\alpha _{1}}(\sum_{i}f_{i,%
\overline{i}}s-\sum_{i\in \mathcal{N}_{\alpha _{1}}}f_{i,\overline{i}%
}s-\sum_{i=\alpha _{1}}f_{i,\overline{i}}s),  \label{7-q-4}
\end{equation}%
\noindent where $\sum_{\alpha _{1}}\sum_{i\in \mathcal{N}_{\alpha _{1}}}f_{i,%
\overline{i}}s$ reads as, by rearranging the index as in deducing (\ref%
{7-2-c2}), 
\begin{equation}
\sum_{\alpha _{1}}\sum_{i\in \mathcal{N}_{\alpha _{1}}}f_{i,\overline{i}%
}s=(\mu -1)(\nu -1)\sum_{i}f_{i,\overline{i}}s.  \label{7-q-3}
\end{equation}

\noindent and $\sum_{\alpha _{1}}\sum_{i=\alpha _{1}}f_{i,\overline{i}%
}s=\sum_{\alpha _{1}}f_{\alpha _{1},\overline{\alpha _{1}}}s=\sum_{i}f_{i,%
\overline{i}}s.$ In sum%
\begin{equation}
(\ref{7-q-4})=[\mu \nu -(\mu -1)(\nu -1)-1]\sum_{i}f_{i,\overline{i}}s=(\mu
+\nu -2)\sum_{i}f_{i,\overline{i}}s.  \label{7-q-5}
\end{equation}%
\noindent Substituting (\ref{7-q-5}) into the second term of the RHS of (\ref%
{7-q-2}) yields%
\begin{equation*}
\sum_{i,\alpha _{1},l}(f_{l}R_{\overline{l}\alpha _{1}\overline{\alpha }%
_{1}i})_{,\overline{i}}s=\frac{c}{2}(\mu +\nu )\sum_{\alpha _{1}}f_{\alpha
_{1},\overline{\alpha _{1}}}s.
\end{equation*}

\noindent This together with (\ref{7-q-1}) proves the $q=0$ case in (\ref%
{7-m}), giving that $\mathcal{E}_{(q+1)}$ and $\mathcal{F}_{(q+1)}$ are
vacuous for $q=0.$
\end{proof}

\begin{proposition}
\label{P-7-2} Suppose that $0\neq s^{1}\in \Omega ^{0,0}(G(\mu ,\nu
),L\otimes T^{\ast })$ is an anti-holomorphic section. Then $i)$ $s^{0}:=%
\overline{I}^{1}\overline{\partial }^{1}s^{1}$ $\in $ $\Omega ^{0,0}(G(\mu
,\nu ),L)$ is an eigensection of $\Delta ^{0}$ with eigenvalue ${\scriptsize %
-B+}\frac{{\scriptsize c}}{{\scriptsize 2}}{\scriptsize (\mu +\nu ).}$ $ii)$
When ${\scriptsize -B+}\frac{{\scriptsize c}}{{\scriptsize 2}}{\scriptsize %
(\mu +\nu )}$ ${\scriptsize \neq }$ ${\scriptsize 0,}$ $s^{0}$ must be
nontrivial.
\end{proposition}

\begin{proof}
(of Proposition \ref{P-7-2} $i))$ We compute%
\begin{eqnarray*}
\Delta ^{0}s^{0} &=&\Delta ^{0}\overline{I}^{1}\overline{\partial }^{1}s^{1}%
\overset{(\ref{7-ma})}{=}\overline{I}^{1}\overline{\partial }^{1}\Delta
^{1}s^{1}+(-B+\frac{c}{2}(\mu +\nu ))\overline{I}^{1}\overline{\partial }%
^{1}s^{1} \\
&=&0+(-B+\frac{c}{2}(\mu +\nu ))s^{0}
\end{eqnarray*}

\noindent as $\Delta ^{1}s^{1}\overset{(\ref{Delta q})}{=}(\partial
^{1})^{\ast }\partial ^{1}s^{1}=0$ since $s^{1}$ is anti-holomorphic by
assumption. We have shown $i).$
\end{proof}

To show that $s^{0}\neq 0$ if $s^{1}\neq 0,$ we need an analogue of
Proposition \ref{Prop 4.4} $i)$ on $G(\mu ,\nu ).$ As the general formula
may involve some extra terms in a similar spirit to $\mathcal{E}_{(q+1)},$ $%
\mathcal{F}_{(q+1)}$ in Proposition \ref{P-7-1}, we will only prove a
special case where no extra terms occur for our purpose.

\begin{lemma}
\label{L-7-2} It holds that $I^{0}\partial ^{0}\overline{I}^{1}\overline{%
\partial }^{1}-\overline{I}^{2}\overline{\partial }^{2}I^{1}\partial ^{1}=[B-%
\frac{c}{2}(\mu +\nu )]\cdot Id$ on $\Omega ^{0,0}(G(\mu ,\nu ),L\otimes
T^{\ast }).$
\end{lemma}

\begin{proof}
By (\ref{A1}) (which works for K\"{a}hler manifolds with (C.C.)), we have%
\begin{eqnarray}
&&(I^{0}\partial ^{0}\overline{I}^{1}\overline{\partial }^{1}-\overline{I}%
^{2}\overline{\partial }^{2}I^{1}\partial ^{1})(\sum_{\alpha _{1}}f_{\alpha
_{1}}s\otimes dz^{\alpha _{1}})  \label{7-2e} \\
&=&-\sum_{i,l,\alpha _{1}}f_{l}R_{\overline{l}\alpha _{1}\overline{\alpha }%
_{1}i}sdz^{i}+B\sum_{\alpha _{1}}f_{\alpha _{1}}s\otimes dz^{\alpha _{1}}. 
\notag
\end{eqnarray}

\noindent The computation of $\sum_{i,l,\alpha _{1}}f_{l}R_{\overline{l}%
\alpha _{1}\overline{\alpha }_{1}i}sdz^{i}$ in (\ref{7-2e}) is almost
identical to that of $\sum_{i,\alpha _{1},l}(f_{l}R_{\overline{l}\alpha _{1}%
\overline{\alpha }_{1}i})_{,\overline{i}}s$ in (\ref{7-q-2}), and reads%
\begin{equation*}
\sum_{i,l,\alpha _{1}}f_{l}R_{\overline{l}\alpha _{1}\overline{\alpha }%
_{1}i}sdz^{i}=\frac{c}{2}(\mu +\nu )\sum_{\alpha _{1}}f_{\alpha
_{1}}s\otimes dz^{\alpha _{1}}.
\end{equation*}%
\noindent This together with (\ref{7-2e}) proves the lemma.

\end{proof}

\begin{lemma}
\label{L-7-3} Suppose that $0\neq s^{1}\in \Omega ^{0,0}(G(\mu ,\nu
),L\otimes T^{\ast })$ is an anti-holomorphic section and ${\scriptsize -B+}%
\frac{{\scriptsize c}}{{\scriptsize 2}}{\scriptsize (\mu +\nu )\neq 0}$.
Then $s^{0}:=\overline{I}^{1}\overline{\partial }^{1}s^{1}$ $\in $ $\Omega
^{0,0}(G(\mu ,\nu ),L)$ is not identically zero.
\end{lemma}

\begin{proof}
We compute%
\begin{eqnarray}
(\overline{I}^{1}\overline{\partial }^{1})^{\ast }s^{0} &=&(\overline{I}^{1}%
\overline{\partial }^{1})^{\ast }\overline{I}^{1}\overline{\partial }%
^{1}s^{1}\overset{\text{Prop.\ref{8} }iii)}{=}-\mathcal{S}I^{0}\partial ^{0}%
\overline{I}^{1}\overline{\partial }^{1}s^{1}  \label{7-11} \\
&&\overset{Lem.\ref{L-7-2}}{=}-\mathcal{S}\overline{I}^{2}\overline{\partial 
}^{2}I^{1}\partial ^{1}s^{1}-\mathcal{S}(B-\frac{c}{2}(\mu +\nu ))s^{1} 
\notag \\
&=&(-B+\frac{c}{2}(\mu +\nu ))s^{1}\text{ (as }\partial ^{1}s^{1}=0\text{ by
assumption).}  \notag
\end{eqnarray}

\noindent The conclusion follows.
\end{proof}

\begin{proof}
(of Proposition \ref{P-7-2} $ii))$ From Lemma \ref{L-7-3}, $ii)$ follows.
\end{proof}

We are going to identify $E_{B,\mu ,\nu }^{L}$ which is the space of the
eigensections of $\Delta ^{0}$ on $\Omega ^{0,0}(G(\mu ,\nu ),L)$ associated
to the eigenvalue $-B+\frac{c}{2}(\mu +\nu ),$ with $E_{anti}^{L\otimes
T^{\ast }}$ which is the space of anti-holomorphic sections in $\Omega
^{0,0}(G(\mu ,\nu ),L\otimes T^{\ast }).$ Define the map $\Psi
:E_{anti}^{L\otimes T^{\ast }}\rightarrow E_{B,\mu ,\nu }^{L}$ by $\Psi
(s^{1})=\overline{I}^{1}\overline{\partial }^{1}s^{1}$. The definition is
justified by Proposition \ref{P-7-2} $i).$

\begin{proposition}
\label{P-7-3} Suppose $E_{B,\mu ,\nu }^{L}\neq \{0\}.$ $i)$ The map $\Psi $
is a linear isomorphism. $ii)$ Assume $-B+\frac{c}{2}(\mu +\nu )$ 
\TEXTsymbol{>}0. Then it is the lowest positive eigenvalue of $\Delta ^{0},$
i.e. there does not exist an eigenvalue $\lambda $ of $\Delta ^{0}$ such
that $0<\lambda <-B+\frac{c}{2}(\mu +\nu ).$
\end{proposition}

\begin{proof}
The proof is basically the same as in previous sections. We give the details
simply to make sure that no extra terms (cf. (\ref{7-m})) would arise here.
For $i),$ we claim that $(\overline{I}^{1}\overline{\partial }^{1})^{\ast }$
is the inverse $\Psi ^{-1}$ up to constants. Note that (\ref{7-11}) gives
one direction. For the other direction, letting $t^{0}\in E_{B,\mu ,\nu
}^{L} $ we have%
\begin{equation*}
\Psi \circ (\overline{I}^{1}\overline{\partial }^{1})^{\ast }t^{0}=(%
\overline{I}^{1}\overline{\partial }^{1})(\overline{I}^{1}\overline{\partial 
}^{1})^{\ast }t^{0}\overset{Lem.\ref{Lemma1}}{=}\Delta ^{0}t^{0}=(-B+\frac{c%
}{2}(\mu +\nu ))t^{0}.
\end{equation*}

\noindent For $ii),$ suppose otherwise. Let $0\neq s_{\lambda }^{0}$ be an
eigensection of $\Delta ^{0}$ with eigenvalue $\lambda ,$ $0<\lambda <-B+%
\frac{c}{2}(\mu +\nu ).$ Let $s_{\lambda }^{1}:=(\overline{I}^{1}\overline{%
\partial }^{1})^{\ast }s_{\lambda }^{0}.$ Note that $s_{\lambda }^{1}\neq 0$
since ($\overline{I}^{1}\overline{\partial }^{1})s_{\lambda }^{1}=(\overline{%
I}^{1}\overline{\partial }^{1})(\overline{I}^{1}\overline{\partial }%
^{1})^{\ast }s_{\lambda }^{0}\overset{Lem.\ref{Lemma1}}{=}\Delta
^{0}s_{\lambda }^{0}=\lambda s_{\lambda }^{0}\neq 0.$ Compute (using the
adjoint of (\ref{7-ma}), i.e. $q=0$ case, for the second equality below)%
\begin{eqnarray}
\Delta ^{1}s_{\lambda }^{1} &=&\Delta ^{1}(\overline{I}^{1}\overline{%
\partial }^{1})^{\ast }s_{\lambda }^{0}  \label{7-11a} \\
&=&(\overline{I}^{1}\overline{\partial }^{1})^{\ast }\Delta ^{0}s_{\lambda
}^{0}-(-B+\frac{c}{2}(\mu +\nu ))(\overline{I}^{1}\overline{\partial }%
^{1})^{\ast }s_{\lambda }^{0}  \notag \\
&=&\{\lambda -(-B+\frac{c}{2}(\mu +\nu ))\}(\overline{I}^{1}\overline{%
\partial }^{1})^{\ast }s_{\lambda }^{0}=\{\lambda -(-B+\frac{c}{2}(\mu +\nu
))\}s_{\lambda }^{1},  \notag
\end{eqnarray}

\noindent contradicting the nonnegativity of $\Delta ^{1}$ as $\lambda -(-B+%
\frac{c}{2}(\mu +\nu ))$ $<$ $0$ and $s_{\lambda }^{1}\neq 0.$
\end{proof}

The eigenvalue $-B+\frac{c}{2}(\mu +\nu )$ above could be the \textbf{second
lowest one}, which follows from the canonical embeddings of Hermitian
symmetric spaces of compact type. The following theorem is basically a
restatement and summing up all the previous results in this section. Let $%
E_{anti}^{L}$ $\subset $ $\Omega ^{0,0}(G(\mu ,\nu ),L)$ denote the space of
all anti-holomorphic sections of $L$ over $G(\mu ,\nu ).$

\begin{theorem}
\label{T-7-1} \textbf{(main result)} Suppose $B(=\deg L)<0.$ Then $i)$ Both $%
E_{B,\mu ,\nu }^{L}$ and $E_{anti}^{L}$ are nontrivial; $E_{B,\mu ,\nu }^{L}$
is naturally isomorphic to $E_{anti}^{L\otimes T^{\ast }}.$ $ii)$ $0$ and $%
-B+\frac{c}{2}(\mu +\nu )$ are the first and second eigenvalues of $\Delta
^{0}$ on $\Omega ^{0,0}(G(\mu ,\nu ),L)$ with the corresponding spaces of
eigensections being $E_{anti}^{L}$ and $E_{B,\mu ,\nu }^{L}$ respectively.
\end{theorem}

\begin{proof}
We need to transform anti-holomorphic sections to holomorphic ones (and vice
versa). The following is standard but perhaps not so well-known (cf. \cite[%
p.23]{Mok89} for an explanation):

\medskip

\textbf{(7-7-1) }\textit{For a holomorphic vector bundle }$\mathcal{E}$%
\textit{\ over a complex manifold }$N$\textit{, denote by }$E_{hol}^{%
\mathcal{E}}(=H^{0}(N,\mathcal{E}))$\textit{\ (resp. }$E_{anti}^{\mathcal{E}%
^{\ast }})$\textit{\ the space of }$\mathcal{E}$\textit{-valued holomorphic
sections (resp. }$\mathcal{E}^{\ast }$\textit{-valued anti-holomorphic
sections). The natural conjugate-linear map }$\Phi :$\textit{\ }$\mathcal{E}%
\rightarrow \mathcal{E}^{\ast }$\textit{\ (via a Hermitian metric on }$%
\mathcal{E}$\textit{) is in fact an isomorphism mapping }$E_{hol}^{\mathcal{E%
}}$\textit{\ onto }$E_{anti}^{\mathcal{E}^{\ast }}.$

\medskip

For $i),$ by Proposition \ref{P-7-3} $i),$ $E_{B,\mu ,\nu }^{L}$is linearly
isomorphic to $E_{anti}^{L\otimes T^{\ast }}$ and in turn is
conjugate-linearly isomorphic to $H^{0}(G(\mu ,\nu ),L^{\ast }\otimes T).$
Then it is enough to show $h^{0}(G(\mu ,\nu ),L^{\ast }\otimes T)$ $>$ $0$
and $h^{0}(G(\mu ,\nu ),L^{\ast })$ $>$ 0$.$ As $L^{\ast }$ is positive and
hence gives rise to (one of) canonical embeddings of $G(\mu ,\nu )$ (\cite[%
p.216]{Mok89}), $L^{\ast }$ is \textit{very ample. }In particular,\textit{\ }%
$h^{0}(G(\mu ,\nu ),L^{\ast })$ $>$ 0 (hence $E_{anti}^{L}$ is nontrivial).
The remaining proof for $h^{0}(G(\mu ,\nu ),L^{\ast }\otimes T)$ $>$ $0$ is
similar to that for $\mathbb{P}^{n}$ (cf. proof of Corollary \ref{TeoD} in
Section \ref{Sec5}). For $ii),$ since $\dim E_{anti}^{L}$ $=$ $h^{0}(G(\mu
,\nu ),L^{\ast })$ $\neq $ $0,$ $0$ is the first eigenvalue of $\Delta ^{0}.$
That $-B+\frac{c}{2}(\mu +\nu )$ is the second eigenvalue follows from
Proposition \ref{P-7-3} $ii)$.

%
%
\end{proof}

\begin{remark}
\label{R-7-1} $a)$ One may compute $\dim E_{B,\mu ,\nu }^{L}$ which equals $%
h^{0}(G(\mu ,\nu ),L^{\ast }\otimes T)$ using similar ideas as in Theorem $%
E. $ One key step there is the total Chern class of $G(\mu ,\nu ),$ whose
explicit expression can be found in \cite[p.521 in Section 16]{BH58}. $b)$
If $\mu $ or $\nu $ $\geq 2$ and $B>0,$ say $B=1,$ $2,$ then it can be
checked that $E_{B,\mu ,\nu }^{L}\neq \{0\}$ and $E_{anti}^{L}$ $=$ $\{0\}.$
In these cases, $-B+\frac{c}{2}(\mu +\nu )$ is the lowest eigenvalue.
\end{remark}

\begin{remark}
\label{R-7-2} (the third lowest eigenvalue case) Following the similar line
of ideas and notations in Theorem \ref{T-7-1}, we ask the question whether $%
s^{0}:=\overline{I}^{1}\overline{\partial }^{1}\overline{I}^{2}\overline{%
\partial }^{2}s^{2}$ with $s^{2}$ being anti-holomorphic is an eigensection
of $\Delta ^{0}$ (with the third lowest eigenvalue). As a matter of fact, we
are encountering an extra term $\overline{I}^{1}\overline{\partial }^{1}(-%
\mathcal{E}_{(2)}+\mathcal{F}_{(2)})s^{2}$ due to (\ref{7-m}) when computing 
$\Delta ^{0}s^{0}.$ Whether or not this extra term vanishes is unclear to us
at this stage.
\end{remark}

\section{Appendix}

We are going to give proofs of Propositions \ref{Prop 4.3}, \ref{Prop 4.4}, %
\ref{Prop 4.5} and then Proposition \ref{9} $ii),$ $iii)$.

\begin{proof}
(of Proposition \ref{Prop 4.3}) 
By (\ref{2-4-1}) and (\ref{2-4-2}) (which works for K\"{a}hler manifolds) we
reorganize terms and obtain (via (\ref{4-7a}) and Lemma \ref{Lemma 2.1})%
\begin{eqnarray}
&&({\scriptsize \Delta }_{-q}{\scriptsize I}^{-(q+1)}{\scriptsize \partial }%
^{-(q+1)}{\scriptsize -I}^{-(q+1)}{\scriptsize \partial }^{-(q+1)}%
{\scriptsize \Delta }_{-(q+1)})(%
\mbox{$ \underset{\tiny
\alpha_{1},..,\alpha_{q+1}}{\sum}  $}f^{\alpha _{1}...\alpha _{q+1}}s\otimes
(\overset{q+1}{\underset{k=1}{\otimes }}%
\mbox{$\frac{\partial}{\partial
z^{\alpha_{k}}}$}))  \label{7-3} \\
&=&%
\mbox{$ \underset{\tiny
\alpha_{1},..,\alpha_{q+1}}{\sum} \underset{j}{\sum}$}\big{(}({f^{\alpha
_{1}..\alpha _{q+1}}}_{,\overline{j}\alpha _{q+1}}-{f^{\alpha _{1}..\alpha
_{q+1}}}_{,\alpha _{q+1}\overline{j}})s\big{)}_{,j}\otimes (\overset{q}{%
\underset{k=1}{\otimes }}\mbox{$\frac{\partial}{\partial
z^{\alpha_{k}}}$})  \notag \\
&&+B\mbox{$ \underset{\tiny \alpha_{1},..,\alpha_{q+1}}{\sum}$}(f^{\alpha
_{1}..\alpha _{q+1}}s)_{,\alpha _{q+1}}\otimes (\overset{q}{\underset{k=1}{%
\otimes }}\mbox{$\frac{\partial}{\partial z^{\alpha_{k}}}$})\text{ \ \ \ \
at }z=0.  \notag
\end{eqnarray}%
\noindent \hspace*{12pt} Applying $(\ref{f Ricci 2})$ and $\nabla R_{i%
\overline{j}k\overline{l}}\equiv 0$ (\ref{4-4a}) to the first term of the
RHS of (\ref{7-3}) we then follow almost the same reasoning as in the proof
of Proposition \ref{Prop 4.2}; see (\ref{eq 4.10}) and the paragraph after
it (with $dz^{\alpha _{k}}$ replaced by $\frac{\partial }{\partial z^{\alpha
_{k}}}$ and so on, via (\ref{7-3})). Under this replacement, the use of
Lemma \ref{L-4-1} $ii)$ replaces that of $i)$ of the same lemma. It turns
out, by going through the analysis in a similar spirit to that between (\ref%
{eq 4.10}) and (\ref{4-15}) and that between (\ref{4-15}) and (\ref{4-18}),
that we still get the same coefficients as in (\ref{4-15}) and (\ref{4-18}).
Then we reach the following with \textit{the same coefficient} as in $(\ref%
{eq 4.11})$%
\begin{eqnarray}
&&%
\mbox{$ \underset{\tiny \alpha_{1},..,\alpha_{q+1}}{\sum}
\underset{j}{\sum}$}\big{(}({f^{\alpha _{1}..\alpha _{q+1}}}_{,\overline{j}%
\alpha _{q+1}}-{f^{\alpha _{1}..\alpha _{q+1}}}_{,\alpha _{q+1}\overline{j}%
})s\big{)}_{,j}\otimes (\overset{q}{\underset{k=1}{\otimes }}%
\mbox{$\frac{\partial}{\partial z^{\alpha_{k}}}$})_{\mid _{z=0}}\hspace*{%
140pt}  \label{7-4} \\
&&\overset{}{=}\frac{c}{2}(2q+n+1)(%
\mbox{$ \underset{\tiny
\alpha_{1},..,\alpha_{q+1}}{\sum}$}(f^{\alpha _{1}..\alpha
_{q+1}}s)_{,\alpha _{q+1}}\otimes (\overset{q}{\underset{k=1}{\otimes }}%
\mbox{$\frac{\partial}{\partial z^{\alpha_{k}}}$}))_{\mid _{z=0}}.  \notag
\end{eqnarray}%
\noindent Substituting $(\ref{7-4})$ into $(\ref{7-3})$ gives (\ref{B+cq+c})
of the proposition, where we have used $%
\mbox{$ \underset{\tiny
\alpha_{1},..,\alpha_{q+1}}{\sum}$}(f^{\alpha _{1}..\alpha
_{q+1}}s)_{,\alpha _{q+1}}$ $\otimes $ $(\overset{q}{\underset{k=1}{\otimes }%
}\mbox{$\frac{\partial}{\partial z^{\alpha_{k}}}$})$ $=$ $I^{(q+1)}\partial
^{-(q+1)}\mbox{$ \underset{\tiny
\alpha_{1},..,\alpha_{q+1}}{\sum}$}(f^{\alpha _{1}..\alpha _{q+1}}s)\otimes (%
\overset{q+1}{\underset{k=1}{\otimes }}%
\mbox{$\frac{\partial}{\partial
z^{\alpha_{k}}}$}).$
\end{proof}

\begin{proof}
(of Proposition \ref{Prop 4.4}) To prove $i),$ for $q\in \mathbb{N}$ write $%
u=f_{\alpha _{1}\cdot \cdot \cdot \text{ }\alpha _{q}}s\otimes (\otimes
_{j=1}^{q}dz^{\alpha _{j}})$ $\in $ $\Omega ^{0,0}(\mathbb{P}^{n},L\otimes
\odot ^{q}T^{\ast }).$ At $p$ with $g_{\alpha \bar{\beta}}(p)=\delta
_{\alpha \bar{\beta}}$ we obtain the first equality below by (\ref{2-36b})
and (\ref{2-36c}) (which works for K\"{a}hler manifolds), and then the
second equality by Lemma \ref{L-4-1} $i)$ and Lemma \ref{Lemma 2.1}:%
\begin{eqnarray}
&&(I^{(q-1)}\partial ^{(q-1)}\overline{I}^{q}\overline{\partial }^{q}-%
\overline{I}^{(q+1)}\overline{\partial }^{(q+1)}I^{q}\partial ^{q})u
\label{A1} \\
&=&\sum_{\alpha _{1}\cdots \alpha _{q},i}[(f_{\alpha _{1}\cdots \alpha _{q},%
\overline{\alpha }_{q}i}-f_{\alpha _{1}\cdots \alpha _{q},i\overline{\alpha }%
_{q}})s-f_{\alpha _{1}\cdots \alpha _{q}}s_{i\overline{\alpha }%
_{q}}]dz^{i}\otimes (\otimes _{j=1}^{q-1}dz^{\alpha _{j}})  \notag \\
&=&-\sum_{\alpha _{1}\cdots \alpha _{q},l,i}(f_{l\alpha _{2}\cdots \alpha
_{q}}R_{\overline{l}\alpha _{1}\overline{\alpha }_{q}i}+f_{\alpha
_{1}l\cdots \alpha _{q}}R_{\overline{l}\alpha _{2}\overline{\alpha }%
_{q}i}+\cdots +f_{\alpha _{1}\cdots \alpha _{q-1}l}R_{\overline{l}\alpha _{q}%
\overline{\alpha }_{q}i})sdz^{i}\otimes (\otimes _{j=1}^{q-1}dz^{\alpha
_{j}})  \notag \\
&&+B\sum_{\alpha _{1}\cdots \alpha _{q}}f_{\alpha _{1}\cdots \alpha
_{q}}s\otimes dz^{\alpha _{q}}\otimes (\otimes _{j=1}^{q-1}dz^{\alpha _{j}}).
\notag
\end{eqnarray}

\noindent It looks straightforward for one to sum up (\ref{A1}) by plugging
into the curvature $R_{i\overline{j}k\overline{l}}.$ This process is tedious
and similar to the summation in (\ref{4-15}) and (\ref{4-18}) up to sign.
Under this reasoning, it turns out that we obtain the similar coefficient as
that in (\ref{eq 4.11}) after replacing $q+1,c$ there by $q,$ $-c$ here.
Altogether, we have now (with the symmetric tensors)%
\begin{equation*}
(\ref{A1})=[-\frac{c}{2}(2q+n-1)+B]f_{\alpha _{1}\cdot \cdot \cdot \text{ }%
\alpha _{q}}s\otimes (\otimes _{j=1}^{q}dz^{\alpha _{j}}).
\end{equation*}

We turn now to $ii)$ of the proposition$.$ With (\ref{f Ricci 1}) and Lemma %
\ref{Lemma 2.1}%
\begin{eqnarray*}
&&(I^{-1}\partial ^{-1}\overline{I}^{0}\overline{\partial }^{0}-\overline{I}%
^{1}\overline{\partial }^{1}I^{0}\partial ^{0})(fs) \\
&=&\sum_{l}[(f_{,\overline{l}l}-f_{,l\overline{l}})s-fs_{l\overline{l}}]%
\overset{}{=}nBfs.
\end{eqnarray*}

For $iii),$ a reasoning similar to $i)$ works with only the sign change in
front of $c$ by using (\ref{f Ricci 2}) rather than (\ref{f Ricci 1}).
\end{proof}

\begin{proof}
(of Proposition \ref{Prop 4.5}) For $i),$ $\Delta ^{q}-\Delta _{q}$ is given
by $i)$ and $ii)$ of Lemma \ref{Lemma 2.2} and (\ref{f Ricci 1}):%
\begin{eqnarray}
&&(\Delta ^{q}-\Delta _{q})f_{\alpha _{1}\cdots \alpha _{q}}s\otimes
(\otimes _{j=1}^{q}dz^{\alpha _{j}})  \label{A2} \\
&=&\sum_{\alpha _{1}\cdots \alpha _{q},i}\{(f_{\alpha _{1}\cdots \alpha _{q},%
\overline{i}i}-f_{\alpha _{1}\cdots \alpha _{q},i\overline{i}})s-f_{\alpha
_{1}\cdots \alpha _{q}}s_{i\overline{i}}\}\otimes (\otimes
_{j=1}^{q}dz^{\alpha _{j}})  \notag \\
&=&-\sum_{\alpha _{1}\cdots \alpha _{q},i,l}\big{(}f_{l\alpha _{2}\cdots
\alpha _{q}}R_{\overline{l}\alpha _{1}\overline{i}i}+f_{\alpha _{1}l\alpha
_{3}\cdots \alpha _{q}}R_{\overline{l}\alpha _{2}\overline{i}i}...+f_{\alpha
_{1}\cdots \alpha _{q-1}l}R_{\overline{l}\alpha _{q}\overline{i}i}\big{)}%
s\otimes (\overset{q}{\underset{k=1}{\otimes }}dz^{\alpha _{k}})  \notag \\
&&-\sum_{\alpha _{1}\cdots \alpha _{q},i}f_{\alpha _{1}\cdots \alpha
_{q}}s_{i\overline{i}}\otimes (\otimes _{j=1}^{q}dz^{\alpha _{j}}).  \notag
\end{eqnarray}

\noindent The remaining reasoning is much simpler than the computation of (%
\ref{eq 4.10}) and (\ref{A1}). We compute it directly as follows. Observe
that $f_{l\alpha _{2}\cdots \alpha _{q}}R_{\overline{l}\alpha _{1}\overline{i%
}i}\neq 0$ only when $l=\alpha _{1},$ and in that case it equals $cf_{\alpha
_{1}\cdots \alpha _{q}}$ if $i=\alpha _{1}$ and $\frac{c}{2}f_{\alpha
_{1}\cdots \alpha _{q}}$ if $i\neq \alpha _{1}.$ Then $\sum_{i,l}f_{l\alpha
_{2}\cdots \alpha _{q}}R_{\overline{l}\alpha _{1}\overline{i}i}$ $=$ $%
f_{\alpha _{1}\cdots \alpha _{q}}(c+\frac{c}{2}(n-1))$ $=$ $\frac{c}{2}%
(n+1)f_{\alpha _{1}\cdots \alpha _{q}}.$ Altogether, we get $\Delta
^{q}-\Delta _{q}$ $=$ ($-\frac{c}{2}(n+1)q+nB)Id,$ giving $i).$

For $ii),$ when $q=0$ there are no curvature terms. Only $s_{i\overline{i}}$
in (\ref{A2}) contribute.

For $iii),$ as similar to $i)$ one uses (\ref{f Ricci 2}) in place of (\ref%
{f Ricci 1}), giving a sign change for the curvature terms.
\end{proof}

\begin{proof}
(of Proposition \ref{9} $ii),$ $iii))$ One sees that the two equalities of
Proposition \ref{9} $ii)$ (resp. $iii))$ correspond to $ii)$ of Proposition\ %
\ref{Prop 4.5} and $ii)$ of Propositions \ref{Prop 4.4} (resp. $iii)$ of
Proposition\ \ref{Prop 4.5} and $iii)$ of Propositions \ref{Prop 4.4}). The
proofs of Propositions \ref{Prop 4.4} and \ref{Prop 4.5} as given above
carry over in a much simpler way on the Abelian variety $M$ (as it is flat).

%
%
%
%
%
\end{proof}

\end{document}